\documentclass[reqno]{amsart}
\usepackage{geometry}
\usepackage{amsmath,amssymb,amsthm,amsfonts}
\allowdisplaybreaks[4]
\usepackage{mathrsfs}
\usepackage{cite}
\usepackage{xcolor} % Blue marks the revised a priori estimates.
\newtheorem{lemma}{Lemma}[section]
\newtheorem{theorem}{Theorem}[section]

\newtheorem{proposition}{Proposition}[section]
\newtheorem{remark}{Remark}[section]

\numberwithin{equation}{section}
\begin{document}
\title[Local Existence and Uniqueness for the Compressible NSCH System with Vacuum]{Local Existence and Uniqueness for the 3D Compressible Navier-Stokes/Cahn-Hilliard System with Vacuum}

\author[Yinghua Li]{Yinghua Li}
\address[Y. Li]{School of Mathematical Sciences, South China Normal University,
	Guangzhou, 510631, China}
\email{yinghua@scnu.edu.cn}
\author[Jingxue Yin]{Jingxue Yin}
\address[J. Yin]{School of Mathematical Sciences, South China Normal University,
	Guangzhou, 510631, China}
\email{yjx@scnu.edu.cn}
\author[Haoran Zheng]{Haoran Zheng$^*$}
\address[H. Zheng]{School of Mathematical Sciences, South China Normal University,
	Guangzhou, 510631, China}
\email{hrzheng@scnu.edu.cn}
\date{}
\maketitle

{\bf Abstract.} {We establish local existence and uniqueness of strong solutions to the
compressible Navier-Stokes/Cahn-Hilliard system in a bounded
smooth domain in $\mathbb R^3$, allowing the initial
density to vanish. The initial data satisfy suitable regularity and compatibility conditions for the momentum and phase equations. The main difficulty arises from the simultaneous degeneracy of the momentum equation and the density-weighted
phase equations, which complicates the construction of compatible positive-density approximations and uniform initial estimates. 
We overcome this difficulty by introducing a time-independent residual into the chemical relation. The residual preserves the phase compatibility exactly, keeps the initial phase field unchanged, and vanishes in $H^1$ as the approximation parameter tends to zero.
For each positive-density approximation, we construct a local strong solution using a semi-discrete Galerkin scheme that discretizes only the phase variables and retains the full transport and momentum equations.
We then derive a priori estimates independent of the density
lower bound, which yield a common lifespan and allow us to pass to the vacuum limit.
Uniqueness is proved using a modified difference energy
that accounts for the coupling between the density and
the chemical potential.
To our knowledge, this is the first local well-posedness result for the three-dimensional compressible Navier-Stokes/Cahn-Hilliard system with initial vacuum.}


\maketitle
% Unmarked title-page acknowledgment, as requested by the journal.
\begingroup
\renewcommand{\thefootnote}{}
\footnotetext{ $^*$Corresponding author
}
\endgroup

\section{Introduction}
\label{sec:problem-main}
In this paper we consider the following compressible Navier-Stokes/Cahn-Hilliard (NSCH) system proposed by Abels and Feireisl \cite{AF08}:
\begin{equation}\label{NSCH_1}
	\left\{
	\begin{aligned}
		& \rho_t + \operatorname{div}(\rho u)=0, && \text{in } \Omega \times (0,T), \\
		& (\rho u)_t  + \operatorname{div}(\rho u \otimes u) - \mathcal{L}u + \nabla p = -\operatorname{div}\left(\nabla \phi \otimes \nabla \phi - \frac{1}{2}|\nabla \phi|^2 \mathbb{I}\right), && \text{in } \Omega \times (0,T),\\
		& (\rho \phi)_t + \operatorname{div}(\rho u\phi) = \Delta \mu, && \text{in } \Omega \times (0,T),\\
		& \rho \mu =  - \Delta \phi + \rho \Psi^{\prime}(\phi), && \text{in } \Omega \times (0,T).
	\end{aligned}
	\right.
\end{equation}
Here, $\Omega\subset\mathbb R^3$ is a bounded smooth domain, $\rho\ge0$ is the density, $u$ is the velocity, $\phi$ is the
phase-field variable, and $\mu$ is the chemical potential. 
The system corresponds to the free energy density $\rho f(\rho,\phi)=\rho\Psi(\phi)+\frac12|\nabla\phi|^2$  with constant viscosities and mobility $m=1$. 
The pressure is given by the isentropic law $p=\rho^\gamma$, $\gamma>1$, and $\Psi$ denotes
the Landau potential $\Psi(s) = \frac{1}{4}(s^2-1)^2$, $s\in \mathbb{R}$.
The Lam\'{e} operator $\mathcal{L}$ is defined by $\mathcal{L} u=\lambda \Delta u + (\lambda + \eta) \nabla \operatorname{div} u$,
where the viscosity coefficients $\lambda$ and $\eta$ satisfy the physical condition
$\lambda > 0, 2\lambda + 3\eta \geq 0$.
Using the identity
\begin{equation*}
	-\operatorname{div}\left(\nabla \phi \otimes \nabla \phi\right) = -\nabla\left(\frac{1}{2}|\nabla \phi|^2\right) - (\Delta \phi)\nabla \phi
\end{equation*}
and the continuity equation \eqref{NSCH_1}$_1$, we can rewrite the system \eqref{NSCH_1} as follows
\begin{equation}\label{NSCH}
	\left\{
	\begin{aligned}
		&  \rho_t + \operatorname{div}(\rho u)=0,
		&& \text{in } \Omega\times(0,T), \\
		& \rho u_t  + \rho u\cdot\nabla u
		- \mathcal{L} u
		+ \nabla  p
		= -(\Delta\phi)\,\nabla\phi,
		&& \text{in } \Omega\times(0,T), \\
		& \rho \phi_t + \rho u\cdot\nabla\phi
		= \Delta\mu,
		&& \text{in } \Omega\times(0,T), \\
		& \rho \mu = -\Delta\phi + \rho \Psi'(\phi),
		&& \text{in } \Omega\times(0,T).
	\end{aligned}
	\right.
\end{equation}
The system \eqref{NSCH} is subject to the initial conditions:
\begin{equation}\label{NSCH_I_C}
	\left(\rho, \rho u, \rho\phi\right)_{t=0} = \left(\rho_0, \rho_0u_0, \rho_0\phi_0\right), \quad \text{in } \Omega,
\end{equation}
and the boundary conditions:
\begin{equation}\label{NSCH_B_C}
	u = 0, \quad \partial_n \phi = \partial_n \mu = 0, \quad \text{on } \partial \Omega \times (0,T).
\end{equation}

The NSCH system is a diffuse interface model for two-phase flows that allows topological changes of the interface. In the classical sharp interface formulation, the interface between two immiscible fluids is represented by a moving hypersurface; the fluid equations are imposed on the two sides of the interface and are coupled through suitable jump conditions. Although this description is physically natural, it becomes difficult to handle when interfaces merge, split, touch the boundary, or undergo large deformations. The phase-field approach avoids this geometric obstruction by replacing the sharp interface with a thin transition layer and introducing an order parameter $\phi$ to distinguish the two phases. The Cahn-Hilliard equation then describes conservative phase separation driven by a chemical potential, while the Navier-Stokes equations govern the bulk fluid motion.

This viewpoint can be traced back to the dynamic classification of Hohenberg and Halperin \cite{HH77}, where the first Navier-Stokes/Cahn-Hilliard model, also known as Model H, was introduced. The system was later derived by Gurtin, Polignone and Vi\~nals \cite{GPV96} within the framework of continuum mechanics, and by Liu and Shen \cite{LS03} through an energetic variational approach. Model H describes incompressible binary mixtures with constant total density, corresponding to the homogeneous case. We refer to \cite{Abel09a,Boyer99,CG12,GG10,GGM16,GMT19} for well-posedness results on it. A natural generalization of Model H is given by nonhomogeneous incompressible mixtures, where the mixture density is treated as an independent unknown and the mass-averaged velocity is assumed to be divergence-free; see \cite{FNG25,GT20,Zhao19}.
%and the preprint by Fang, Duan and Guo (\href{https://arxiv.org/abs/2511.08876}{arXiv:2511.08876}).
For unmatched-density and quasi-incompressible models, we refer to \cite{Abel09,ADG13,AGG24,EZAS23}.

Compared with the incompressible theory, rigorous results for compressible NSCH systems remain far more limited. Compressible and quasi-incompressible diffuse interface models were developed by Lowengrub and Truskinovsky \cite{LT98}. A mathematically tractable compressible diffuse interface model was later proposed by Abels and Feireisl \cite{AF08}. Kotschote and Zacher \cite{KZ15} established local-in-time existence and uniqueness of strong solutions for the compressible NSCH model derived in \cite{LT98}.
More recently, Elbar and Poulain \cite{EP24} studied a generalized compressible NSCH model with friction effects and proved the existence of global weak solutions under simplifying assumptions.
 %while Basari\'c and Giorgini (preprint, \href{https://arxiv.org/abs/2506.07835}{arXiv:2506.07835}) obtained global weak solutions with the singular Flory-Huggins entropy of mixing.
Other related topics, including stationary problems \cite{LW20,LW22}, moving contact lines \cite{GGM16}, dynamic boundary conditions \cite{CFMMPP19}, and low Mach number limits \cite{ALN24}, have also been investigated for NSCH systems.

When vacuum is taken into account in compressible two-phase flows, vanishing-density states arise as a limiting idealization of extremely low-density regions, although vapor typically retains a nonzero density.
Such regions are physically relevant, for example, in liquid-vapor flows involving cavitation or strong expansion, or during rupture events, where an interfacial region is evacuated as a liquid filament pinches off.  Recently, Ding, Li, Yan and Zheng \cite{DLYZ26} established global existence and uniqueness of strong solutions allowing vacuum in the one-dimensional setting. The corresponding problem in the multidimensional case, raised in \cite{GT20} for the incompressible nonhomogeneous model and in \cite{DLYZ26} for the compressible one, remains open.

To the best of our knowledge, this paper establishes the first local well-posedness result for strong solutions to the three-dimensional compressible NSCH system \eqref{NSCH}--\eqref{NSCH_B_C} with initial vacuum. The main difficulty comes from the presence of vacuum. When $\rho$ is allowed to vanish, the momentum equation becomes degenerate, and one cannot divide by the density to recover $u_t$ or $\phi_t$.
For compressible Navier-Stokes equations with vacuum, suitable compatibility conditions are usually imposed to ensure that the initial material acceleration is square-integrable. This approach goes back to Salvi and Stra\v{s}kraba \cite{SS93} and was further developed in \cite{CK03,CCK04,CK06}. In the NSCH system, however, the
	Cahn-Hilliard coupling introduces an additional degeneracy:
	the chemical relation
	\[
	\rho\mu=-\Delta\phi+\rho\Psi'(\phi)
	\]
	cannot be solved for $\mu$ by division by $\rho$ on the vacuum
	set, while the phase equation
	\[
	\rho(\phi_t+u\cdot\nabla\phi)=\Delta\mu
	\]
	does not directly control the unweighted phase derivative $\phi_t$.
On the set where $\rho$ vanishes, the momentum equation loses its inertial and pressure terms, the chemical relation reduces to $\Delta\phi=0$, and the phase equation reduces to $\Delta\mu=0$: the phase field and the chemical potential are harmonic on the vacuum set, the capillary stress vanishes, and no evolution law for $\phi$ is available. Consequently, direct regularization of the density does not automatically preserve the phase compatibility needed for uniform initial estimates.
We resolve this difficulty by preserving the initial phase field and absorbing the incompatibility of the chemical relation into a time-independent residual, which vanishes in $H^1$ as the approximation parameter tends to zero.

%The Navier-Stokes/Allen-Cahn (NSAC) system is closely related to NSCH and provides a useful point of comparison, see \cite{BLesgen99,FPRS10}.
%In recent work, Li, Zheng, and Zhou (preprint, \href{https://arxiv.org/abs/2605.24568}{arXiv:2605.24568}) established local well-posedness and a blow-up criterion for the three-dimensional compressible NSAC system with vacuum. Their analysis imposes a compatibility condition on the initial phase field, but uses time-weighted estimates to avoid a velocity compatibility condition.
%Unlike the second-order Allen-Cahn dynamics, the NSCH phase
%equation involves $\Delta\mu$, while the chemical relation
%$\rho\mu=-\Delta\phi+\rho\Psi'(\phi)$ degenerates at vacuum.
%For positive density, eliminating $\mu$ introduces spatial
%derivatives of terms involving division by $\rho$, making
%estimates independent of the density lower bound substantially
%more difficult to obtain.

While the diffuse interface formulation permits topological changes at the level of the model, analyzing them rigorously, such as whether a thinning filament ruptures in finite time, requires the pointwise information provided by strong solutions. Since topological changes are precisely where vacuum regions form, a local strong theory with vacuum is a prerequisite for such questions.
Our main result is stated as follows.
\begin{theorem}\label{main-thm}
	Let $\Omega\subset\mathbb{R}^3$ be a bounded smooth connected domain and
	$3<q<6$.
	Assume that the initial data satisfy
	\begin{align*}
	&\rho_0\ge0,\qquad 0<\int_\Omega\rho_0\,\mathrm dx,
	\qquad \rho_0\in W^{1,q}(\Omega),\\
	&u_0\in H^2(\Omega)\cap H_0^1(\Omega),\qquad
	\phi_0\in H^3(\Omega),\qquad \mu_0\in H^3(\Omega),
	\end{align*}
	and the compatibility conditions
	\begin{equation}\label{C_C}
		\begin{cases}
			\mathcal{L}u_0
			-
			\nabla(\rho_0^\gamma)
			-
			(\Delta\phi_0)\nabla\phi_0
			=
			\sqrt{\rho_0}\,g_0,
			\quad \text{in }\Omega,
			\\
			\Delta\mu_0
			=
			\rho_0\,h_0,
			\quad \text{in }\Omega,
			\\
			\int_\Omega \rho_0 h_0\,\mathrm d x=0,
			\\
			u_0=0,
			\qquad
			\partial_n\phi_0
			=
			\partial_n\mu_0
			=
			0,
			\quad \text{on }\partial\Omega,\\
			\rho_0\mu_0=-\Delta\phi_0+\rho_0\Psi'(\phi_0),
		\end{cases}
	\end{equation}
	where
	$g_0\in L^2(\Omega)$ and $h_0\in H^1(\Omega)$.
	Then there exists a time $T_0>0$ such that the system  \eqref{NSCH}--\eqref{NSCH_B_C} admits a unique strong solution
	$(\rho,u,\phi,\mu)$
	on $\Omega\times(0,T_0)$ with
	{$\rho\ge0$ a.e. in $\Omega\times(0,T_0)$ and}
	\[
	\left\{
	\begin{aligned}
		&
		\rho
		\in
		L^\infty(0,T_0;W^{1,q}),
		\qquad
		\rho_t
		\in
		L^\infty(0,T_0;L^2),
		\\
		&
		u
		\in
		L^\infty(0,T_0;H^2\cap H_0^1)
		\cap
		L^2(0,T_0;W^{2,q}),
		\\
		&
		\sqrt{\rho}\,u_t
		\in
		L^\infty(0,T_0;L^2),
		\qquad
		u_t
		\in
		L^2(0,T_0;H_0^1),
		\\
		&
		\phi
		\in
		L^\infty(0,T_0;H^3),
		\qquad
		\phi_t
		\in
		L^\infty(0,T_0;H^1),
		\\
		&
		\sqrt{\rho}\,\phi_t
		\in
		L^\infty(0,T_0;L^2),
		\\
		&
		\mu
		\in
		L^\infty(0,T_0;H^2),
		\qquad
		\mu_t
		\in
		L^2(0,T_0;H^1).
	\end{aligned}
	\right.
	\]
	Moreover,
	\[
	\rho\in C([0,T_0];C(\overline\Omega)),
	\qquad
	u\in C([0,T_0];H_0^1),
	\]
	\[
	\phi\in C([0,T_0];H^2),
	\qquad
	\mu\in C([0,T_0];H^1).
	\]
	The equations in \eqref{NSCH} hold almost everywhere in
	$\Omega\times(0,T_0)$.
\end{theorem}
\begin{remark}\label{Remark 1.1}
	The second condition in \eqref{C_C} defines an initial phase derivative
	without dividing by the density.  Set
	\[
	\phi_{1,0}:=h_0-u_0\cdot\nabla\phi_0\in H^1(\Omega).
	\]
	Then \(\Delta\mu_0=\rho_0h_0\) implies
	\[
	\rho_0\phi_{1,0}
	=
	\Delta\mu_0-\rho_0u_0\cdot\nabla\phi_0.
	\]
	Thus the trace of \eqref{NSCH}$_3$ at \(t=0\) gives only
		\(\rho_0\phi_t(0)=\rho_0\phi_{1,0}\).  On the vacuum set this equation
		does not determine \(\phi_t(0)\).  For the positive-density
	problem, we take \(\phi_{1,0}\) as the initial value of $\phi_t$.
    All estimates below involve $\phi_{1,0}$ only through its $H^1$-norm
    (or the weighted $L^2$-norm $\|\sqrt{\rho_0}\phi_{1,0}\|_{L^2}$), and are therefore uniform in the approximation parameter $\delta$. 
\end{remark}

\begin{remark}
	Although the density-weighted initial conditions do not directly prescribe the velocity and phase field on the initial vacuum set, the strong solution in Theorem~\ref{main-thm} is unique throughout $\Omega$, including the vacuum region, as in the heat-conductive compressible Navier--Stokes equations with vacuum \cite{LZ23}. This uniqueness follows from the non-degenerate viscous dissipation and the elliptic coupling between the phase field and the chemical potential, which, together with the boundary conditions, control the differences of solutions even where the density vanishes.
\end{remark}

We now make some comments on our proof. The first difficulty lies in constructing positive-density
approximations that preserve the initial compatibility conditions.
Although replacing $\rho_0$ by a strictly positive density
$\rho_0^\delta$ removes the degeneracy at the approximate level,
it does not automatically preserve both the chemical relation and
the phase compatibility condition. With the initial phase field $\phi_0$ fixed (independent of $\delta$), 
the direct choice
\[
\mu_0^\delta
=\Psi'(\phi_0)-(\rho_0^\delta)^{-1}\Delta\phi_0
\]
introduces division by the approximate density.
Unfortunately, this choice does not ensure uniform control of
$(\rho_0^\delta)^{-1}\Delta\mu_0^\delta$, which determines the
initial material derivative of the phase. Thus, positivity alone
is insufficient to obtain the initial estimates needed for
the vacuum limit.

The key observation is that the phase compatibility condition
can be preserved first, while the resulting mismatch in the
chemical relation is retained as a small, time-independent residual.
More precisely, we construct $\mu_0^\delta$ through a Neumann
problem satisfying
$\Delta\mu_0^\delta=\rho_0^\delta h_0^\delta$, with
$h_0^\delta$ uniformly bounded in $H^1$, and define $R^\delta$
so that
\[
\rho^\delta\mu^\delta
=-\Delta\phi^\delta
+\rho^\delta\Psi'(\phi^\delta)+R^\delta.
\]
In this construction, introducing $R^\delta$ allows us to retain
the initial phase field and reconcile the chemical relation with the phase compatibility condition
without losing uniform control of the initial time derivatives.
The residual is chosen to have zero spatial mean and satisfies
$R^\delta\to0$ in $H^1(\Omega)$.
Its zero mean allows the energy correction
$\int_\Omega R^\delta\phi^\delta\,\mathrm dx$
to be controlled by the phase gradient through the Poincar\'{e} inequality.
The remaining contribution to the modified energy identity,
\[
-\int_\Omega R^\delta u^\delta\cdot
\nabla\phi^\delta\,\mathrm dx,
\]
is controlled by viscous dissipation and Gronwall's inequality.
Moreover, time differentiation of the chemical relation produces
no derivative of $R^\delta$.

A second difficulty concerns the construction of solutions to
the resulting perturbed system while retaining the estimates needed for
the subsequent vacuum limit. We discretize only the phase
variables $(\phi,\mu)$ in a Neumann spectral space, leaving
the transport and Lam\'e equations for $(\rho,u)$ unprojected.
This choice preserves access to the characteristic representation
of the density and the full elliptic estimates for the velocity,
which are essential for controlling $\rho$ in $W^{1,q}$ and
$u$ in $W^{2,q}$ independently of the phase dimension.
However, the discrete chemical relation holds only against
spectral test functions and cannot be used as a pointwise identity.
To preserve the cancellation between the momentum and phase
equations, we therefore use the compatible capillary force
\[
[\rho_m\mu_m-\rho_m\Psi'(\phi_m)-R_m]\nabla\phi_m.
\]
Weighted projections provide admissible chemical test functions
and allow the coupled estimates to close uniformly in $m$.

The approximation is consequently carried out in two stages.
We first pass to the limit $m\to\infty$ at fixed $\delta>0$,
obtaining a strong solution of the positive-density perturbed
system. We then derive coupled estimates for the weighted time
derivatives and phase gradients whose constants are independent
of the density lower bound. These estimates provide a common
lifespan and the compactness required to let $\delta\to0$.
Since $R^\delta\to0$ in $H^1(\Omega)$, the residual disappears
in this limit, and the original system with vacuum is recovered.
Beyond the present system, we believe that the use of a vanishing, time-independent residual may provide a useful approximation strategy for other degenerate coupled systems in which direct regularization fails to preserve the initial compatibility conditions needed for uniform estimates.

The remainder of the paper is organized as follows.
In Section~\ref{sec:perturbed-problem}, we construct compatible
positive-density initial data and introduce the corresponding
perturbed system. In Section~\ref{sec:positive-density}, we
state local existence and a continuation criterion for this system,
with the complete proofs deferred to Appendix~\ref{app:construction}.
In Section~\ref{sec:apriori}, we derive a priori estimates
independent of the density lower bound.
Using these estimates, we obtain a common lifespan, pass to
the vacuum limit, and prove uniqueness in
Section~\ref{sec:vacuum-limit}, thereby completing the proof
of Theorem~\ref{main-thm}.
Appendix~\ref{app:linear-lemmas} collects the auxiliary linear results.
Appendix~\ref{app:construction} gives the positive-density construction
and its semi-discrete estimates, following the order of
Section~\ref{sec:apriori}.

\section{Formulation of the positive-density approximation problem}
\label{sec:perturbed-problem}

In this section, we first construct compatible approximate initial data. Then, we introduce the approximation problem with the resulting time-independent chemical residual.  Throughout this paper, we set
\[
\int_\Omega f \,\mathrm d x:= \int f\,\mathrm d x .
\]
For $1\le p\le \infty$ and integers $k\ge0$, the standard Sobolev spaces are denoted by:
\[
L^p=L^p(\Omega),\quad
W^{k,p}=W^{k,p}(\Omega),\quad
H^k=W^{k,2}(\Omega), \quad
H_0^1=\{v\in H^1(\Omega): v|_{\partial\Omega}=0\}.
\]

\subsection{Preparation of compatible positive-density data}
\label{subsec:compatible-positive-data}

We construct positive-density initial data
$(\rho_0^\delta,u_0,\phi_0,\mu_0^\delta)$ satisfying the
perturbed compatibility conditions according to \eqref{C_C},
with bounds uniform in $\delta$.
With $\phi_0$ fixed, directly defining $\mu_0^\delta$
from the unperturbed chemical relation involves division
by $\rho_0^\delta$ and does not automatically yield
uniform control of the initial phase derivative.
We therefore impose the phase compatibility condition
exactly and absorb the mismatch in the chemical relation
into a small, time-independent residual $R^\delta$
with zero spatial mean.

We first prepare the velocity data. Define
\begin{equation*}
 \rho_0^\delta=(\rho_0^\gamma+\delta)^{1/\gamma},
 \qquad 0<\delta<1.
\end{equation*}
It is clear that $\rho_0^\delta\ge\delta^{1/\gamma}\ge\delta$, $(\rho_0^\delta)^\gamma=\rho_0^\gamma+\delta$ and $\nabla(\rho_0^\delta)^\gamma=\nabla\rho_0^\gamma$.
Noticing that
$
 \displaystyle\nabla\rho_0^\delta
 =\left(\frac{\rho_0}{\rho_0^\delta}\right)^{\gamma-1}
 \nabla\rho_0,
$
it holds that
\begin{equation}\label{rho0-delta-convergence}
 \rho_0^\delta\rightarrow\rho_0
 \quad\hbox{strongly in }W^{1,q}.
\end{equation}

Next, set
\begin{equation*}
 g_0^\delta=\left(\frac{\rho_0}{\rho_0^\delta}\right)^{1/2}g_0.
\end{equation*}
Then $\|g_0^\delta\|_{L^2}\le\|g_0\|_{L^2}$, and it can be verified that
\begin{equation}\label{momentum-compatibility-delta}
 \mathcal Lu_0-\nabla(\rho_0^\delta)^\gamma
 -(\Delta\phi_0)\nabla\phi_0
 =\sqrt{\rho_0^\delta}\,g_0^\delta.
\end{equation}

Next we prepare the phase data. Let
\begin{equation*}
 c_\delta=
 \frac{\int \rho_0^\delta h_0\,\mathrm dx}
 {\int \rho_0^\delta\,\mathrm dx},
 \qquad h_0^\delta=h_0-c_\delta.
\end{equation*}
Because $\int\rho_0h_0=0$, the numerator equals
$\int(\rho_0^\delta-\rho_0)h_0$ and tends to zero.  The denominator tends
to $\int\rho_0>0$.  Hence
\begin{equation}\label{h0-delta-properties}
 c_\delta\to0,\qquad h_0^\delta\to h_0\ \hbox{in }H^1,
 \qquad \int \rho_0^\delta h_0^\delta\,\mathrm dx=0.
\end{equation}
Now, consider the solution $\widehat\mu_0^\delta$ of
\begin{equation}\label{mu0-delta-Neumann}
 \Delta\widehat\mu_0^\delta=\rho_0^\delta h_0^\delta,\qquad
 \partial_n\widehat\mu_0^\delta=0,\qquad
 \int \widehat\mu_0^\delta\,\mathrm dx
 =\int \mu_0\,\mathrm dx, \quad \text{in } \Omega.
\end{equation}
Set $\mu_0^\delta=\widehat\mu_0^\delta+a_\delta$, with
\begin{equation}\label{mu0-delta-shift}
 a_\delta=
 \frac{\displaystyle\int \rho_0^\delta
 \bigl(\Psi'(\phi_0)-\widehat\mu_0^\delta\bigr)\,\mathrm dx}
 {\displaystyle\int \rho_0^\delta\,\mathrm dx}.
\end{equation}
Since $a_\delta$ is constant, we have $
 \Delta\mu_0^\delta=\rho_0^\delta h_0^\delta$ and $\partial_n\mu_0^\delta=0$.
By the product estimate
\begin{equation}\label{W1q-H1-product}
 \|\rho_0^\delta h_0^\delta\|_{H^1}\le C\|\rho_0^\delta\|_{W^{1,q}}\|h_0^\delta\|_{H^1},
 \qquad q>3,
\end{equation}
it follows that $\rho_0^\delta h_0^\delta\to\rho_0h_0$ in $H^1$.  Subtract
$\Delta\mu_0=\rho_0h_0$ from \eqref{mu0-delta-Neumann}.  Neumann elliptic
regularity gives
\[
 \|\widehat\mu_0^\delta-\mu_0\|_{H^3}
 \le C\|\rho_0^\delta h_0^\delta-\rho_0h_0\|_{H^1}
 \rightarrow0.
\]
Integrating the last relation in \eqref{C_C} and using
$\partial_n\phi_0=0$, we obtain
\[
 \int \rho_0\bigl(\mu_0-\Psi'(\phi_0)\bigr)\,\mathrm dx=0.
\]
Thus the numerator in \eqref{mu0-delta-shift} tends to zero and its
denominator tends to $\int\rho_0>0$, which implies $a_\delta\to0$ and
\begin{equation}\label{mu0-delta-convergence}
 \|\mu_0^\delta-\mu_0\|_{H^3}\rightarrow0.
\end{equation}

Finally we define the residual and verify its properties. Set
\begin{equation*}
 R^\delta=\rho_0^\delta\mu_0^\delta+\Delta\phi_0
 -\rho_0^\delta\Psi'(\phi_0).
\end{equation*}
By \eqref{mu0-delta-shift} and $\partial_n\phi_0=0$, one derives
\begin{equation}\label{R-delta-zero-mean}
 \int  R^\delta\,\mathrm dx=0.
\end{equation}
To verify that the perturbation disappears, use the last relation in
\eqref{C_C} to write
\begin{align*}
 R^\delta= (\rho_0^\delta-\rho_0)\mu_0
 +\rho_0^\delta(\mu_0^\delta-\mu_0) -(\rho_0^\delta-\rho_0)\Psi'(\phi_0).
\end{align*}
Then \eqref{rho0-delta-convergence} and \eqref{mu0-delta-convergence} imply
$R^\delta\rightarrow0$ strongly in $H^1$.

In summary, there holds
$$
\sup_{0<\delta<1}\left(\|\rho_0^\delta\|_{W^{1,q}}+\|g_0^\delta\|_{L^{2}}+\|h_0^\delta\|_{H^{1}}+\|\mu_0^\delta\|_{H^{3}}\right)<\infty,
\qquad
\|R^\delta\|_{H^{1}}\rightarrow0,
$$
which are precisely the properties used in Sections \ref{sec:apriori} and \ref{sec:vacuum-limit}.

\subsection{The positive-density approximation problem}
\label{subsec:approximation-system}
Let $0<\delta<1$, take the initial data as
\begin{equation}\label{all-delta-data}
	u_0^\delta=u_0,\qquad \phi_0^\delta=\phi_0,\qquad
	(\rho^\delta,u^\delta,\phi^\delta,\mu^\delta)(0)
	=(\rho_0^\delta,u_0,\phi_0,\mu_0^\delta).
\end{equation}
We consider the positive-density approximation system
\begin{equation}\label{perturbed_system}
\left\{
\begin{aligned}
&\rho_t^\delta+\operatorname{div}(\rho^\delta u^\delta)=0,
&&\text{in }\Omega\times(0,T),\\
&\rho^\delta
\left(u_t^\delta+u^\delta\cdot\nabla u^\delta\right)
-\mathcal Lu^\delta
+\nabla\bigl((\rho^\delta)^\gamma\bigr)
= -\Delta\phi^\delta\nabla\phi^\delta,
&&\text{in }\Omega\times(0,T),\\
&\rho^\delta
\left(\phi_t^\delta+u^\delta\cdot\nabla\phi^\delta\right)
=\Delta\mu^\delta,
&&\text{in }\Omega\times(0,T),\\
&\rho^\delta\mu^\delta
=-\Delta\phi^\delta+\rho^\delta\Psi'(\phi^\delta)+R^\delta,
&&\text{in }\Omega\times(0,T),
\end{aligned}
\right.
\end{equation}
subject to the initial conditions
\begin{equation}\label{I_C}
(\rho^\delta,u^\delta,\phi^\delta,\mu^\delta)|_{t=0}
=
(\rho_0^\delta,u_0^\delta,\phi_0^\delta,\mu_0^\delta),
\end{equation}
and the boundary conditions
\begin{equation}\label{B_C}
	u^\delta=0,
	\qquad
	\partial_n\phi^\delta=\partial_n\mu^\delta=0
	\quad\text{on }\partial\Omega\times(0,T).
\end{equation}
Moreover, the initial data satisfy the following compatibility conditions
\begin{equation}\label{App_C_C}
	\left\{
	\begin{aligned}
		&\mathcal Lu_0^\delta
		-
		\nabla\bigl((\rho_0^\delta)^\gamma\bigr)
		-
		(\Delta\phi_0^\delta)\nabla\phi_0^\delta
		=
		\sqrt{\rho_0^\delta}\,g_0^\delta, &&\text{in }\Omega,\\
		&\Delta\mu_0^\delta=\rho_0^\delta h_0^\delta, &&\text{in }\Omega,\\
		&\rho_0^\delta\mu_0^\delta
		=
		-\Delta\phi_0^\delta
		+
		\rho_0^\delta\Psi'(\phi_0^\delta)
		+R^\delta, &&\text{in }\Omega,\\
		&\int \rho_0^\delta h_0^\delta\,\mathrm d x=0,
	\end{aligned}
	\right.
\end{equation}
where $g_0^\delta\in L^2, h_0^\delta\in H^1, R^\delta\in H^1$.
It is clear that, as $\delta\to0$, the perturbation $R^\delta$ disappears, and the original
system \eqref{NSCH} is recovered.

%{\begin{remark}\label{Physical}
% The positive-density approximation temporarily restores dynamics to the would-be vacuum region: since $\rho^\delta>0$ on $\overline\Omega$, the phase equation determines $\varphi_t^\delta$ everywhere, whereas in the limit system the phase equation determines the material derivative of  $\varphi$ is defined only on $\{\rho>0\}$ (cf. Remark \ref{Remark 1.1} and the discussion in Section 1). The residual $R^\delta$ measures precisely the error committed in the chemical relation by this filling, and the convergence $R^\delta\rightarrow0$ in $H^1(\Omega)$ means that this information is recovered in the limit.
%\end{remark}
%}

\section{Existence for the positive-density approximation problem}
\label{sec:positive-density}

This section states local existence and continuation for the
positive-density system \eqref{perturbed_system}-\eqref{App_C_C}.
The complete proofs, including the phase Galerkin construction, are
collected in Appendix~\ref{app:construction}.
Throughout this section, we suppress the superscript $\delta$.
The main result of this section is stated as follows.
\begin{theorem}[Positive-density local existence]\label{positive-density-thm}
Let $3<q<6$. Assume the initial data satisfy
\[
 \rho_0\in W^{1,q},\qquad 0<\underline\rho\le\rho_0,\qquad
 u_0\in H^2\cap H_0^1,\qquad \phi_0,\mu_0\in H^3,
\]
and the compatibility conditions \eqref{App_C_C}.
Then there exists $T_*>0$ and a strong solution of the positive-density approximation system \eqref{perturbed_system}-\eqref{B_C}, such that
\begin{align*}
 &\rho\in L^\infty(0,T_*;W^{1,q}),\quad
 \rho_t\in L^\infty(0,T_*;L^q),\quad
 \rho\ge\tfrac14\underline\rho,\\
 &u\in L^\infty(0,T_*;H^2\cap H_0^1)
       \cap L^2(0,T_*;W^{2,q}),\\
 &u_t\in L^\infty(0,T_*;L^2)\cap L^2(0,T_*;H_0^1),\\
 &\phi\in L^\infty(0,T_*;H^3),\quad
 \phi_t\in L^\infty(0,T_*;H^1),\\
 &\mu\in L^\infty(0,T_*;H^2)\cap L^2(0,T_*;H^3),\quad
 \mu_t\in L^2(0,T_*;H^1).
\end{align*}
Moreover, as $t\to0^+$,
\[
(\rho(t),u(t),\phi(t),\mu(t))
\rightarrow (\rho_0,u_0,\phi_0,\mu_0)
\quad\text{strongly in } W^{1,q}\times H^2\times H^3\times H^2.
\]
The time $T_*$ is allowed to depend on
$\underline\rho^{-1}$.
\end{theorem}
\begin{proof}
The complete proof is given in Appendix~\ref{app:construction}.
Only the phase variables $(\phi_m,\mu_m)$ are Galerkin truncated;
the density and velocity remain in their full spatial function spaces
on $\Omega$ and solve the transport and momentum equations without
spatial projection. The fixed-dimensional construction, the estimates
of Lemma~\ref{app:positive-density-semidiscrete-estimate}, and the
limit in Subsection~\ref{subsec:positive-step-limit} give the asserted
solution and strong initial traces. The existence time at this stage
may depend on $\underline\rho^{-1}$.
\end{proof}

The following continuation criterion ensures that a positive-density
solution cannot cease to exist while its density remains bounded away
from zero and the relevant norms remain bounded, thereby allowing us
to turn the uniform a priori estimates into a common lifespan for
the approximate solutions.

\begin{proposition}[Continuation criterion at positive density]
	\label{positive-density-continuation}
	Let $(\rho,u,\phi,\mu)$ be a strong solution furnished by
	Theorem~\ref{positive-density-thm} on $[0,T^*)$,
	where $T^*<\infty$.
	Assume that, for some $c_*>0$ and $C_*<\infty$,
	\begin{equation}
		\inf_{\Omega\times(0,T^*)}\rho\ge c_*,
		\label{positive-continuation-assumptions}
	\end{equation}
	and
	\begin{equation}
		\begin{aligned}
			\sup_{0<t<T^*}\Bigl(
			\|\rho(t)\|_{W^{1,q}}
			+\|u(t)\|_{H^2}
			+\|u_t(t)\|_{L^2}
			+\|\phi(t)\|_{H^3}
			+\|\phi_t(t)\|_{H^1}
			+\|\mu(t)\|_{H^2}
			\Bigr)\le C_*.
		\end{aligned}
		\label{positive-continuation-bounds}
	\end{equation}
	Then the solution extends as a positive-density strong
	solution beyond $T^*$.
\end{proposition}

\begin{proof}
The restarting argument is given in
Subsection~\ref{app:continuation-proof}, using the local construction
and its common lifespan at fixed positive density.
\end{proof}

\section{A priori estimates independent of the density lower bound}
\label{sec:apriori}

In this section, we derive a priori estimates, uniform as $\delta\to0$, for the
positive-density solutions constructed in
Section~\ref{sec:positive-density}.
Throughout this section, we suppress the superscript $\delta$
and write $R=R^\delta$, with
$\int  R\,\mathrm dx=0$ by \eqref{R-delta-zero-mean}.
The resulting bounds are independent of the density lower bound.
For simplicity, assume $0<T\le1$.  We consider the quantity
\begin{equation}\label{def_Phi}
\begin{aligned}
\Phi(t):=\sup_{0\le s\le t}\bigg(
&1+\|\rho(s)\|_{W^{1,q}}
+\|\rho(s)\|_{H^1}
+\|R\|_{H^1}^2
+\|\nabla u(s)\|_{L^2}^2
+\|\sqrt{\rho}\,u_t(s)\|_{L^2}^2
\\
&+\|\sqrt{\rho}\,\mu(s)\|_{L^2}^2
+\|\phi(s)\|_{H^3}^2
+\|\sqrt{\rho}\,\phi_t(s)\|_{L^2}^2
+\|\mu(s)\|_{H^2}^2
+\|\nabla\phi_t(s)\|_{L^2}^2
\bigg),
\end{aligned}
\end{equation}
where $3<q<6$ and $0<t<T$.

We record the following preliminary inequality, whose constant does
not require a positive pointwise lower bound for the density.
\begin{lemma}[{\rm cf. \cite{F03}}]\label{Lemma_F03}
	Let $v \in H^1(\Omega)$ and let $\Omega \subset \mathbb{R}^3$ be a
	bounded connected domain. If a non-negative function $\rho$ satisfies
	\[
	0<C_0\leq\int  \rho \mathrm d x \text{ and } \int  \rho^\gamma \mathrm d x \leq E_0
	\]
	for some $\gamma>1$, then there exists a constant
	$C=C(\Omega,\gamma,C_0,E_0)$ such that
	\[
	\| v \|_{L^2}^2 \leq C\left(\|\nabla v\|_{L^2}^2+\left(\int  \rho|v| \mathrm d x \right)^2\right).
	\]
\end{lemma}

Throughout this section, $C$ denotes a generic positive constant depending
only on $\Omega$, the physical parameters, and the indicated norms of the
initial data.  Since $0<T\le1$, all constants can be chosen independently of
$T$ and of the positive lower bound of the density.  Let $\mathcal F$ denote
the class of nondecreasing locally
bounded functions $M:[0,\infty)\to[1,\infty)$,
which are independent of $\delta$.  The exact form of $M\in\mathcal F$ may
vary from line to line. For simplicity, we denote $F(\phi):=\Psi'(\phi)=\phi^3-\phi$ and	$F'(\phi)=3\phi^2-1$.

The main perturbed estimate is as follows.
\begin{proposition}[A priori estimates for the perturbed problem]\label{apriori-prop}
Let $\Omega\subset\mathbb R^3$ be a bounded smooth connected domain and let $3<q<6$.
Assume that $(\rho,u,\phi,\mu)$ is a sufficiently smooth solution to
\eqref{perturbed_system} with the boundary conditions stated in
Section~\ref{sec:perturbed-problem}.  Assume also that
$\int R\,\mathrm dx=0$ and
\[
 \mathcal Lu_0-\nabla\rho_0^\gamma-(\Delta\phi_0)\nabla\phi_0
 =\sqrt{\rho_0}\,g_0,
 \qquad
 \Delta\mu_0=\rho_0h_0,
\]
where $g_0\in L^2$, $h_0\in H^1$, and
$\int\rho_0h_0=0$.  Set
\begin{equation}\label{def_initial_M}
	\begin{aligned}
\mathcal M_0:=1
&+\|\rho_0\|_{W^{1,q}}^2
+\|\rho_0\|_{H^1}^2
+\|R\|_{H^1}^2
+\|u_0\|_{H^2}^2
\\
&+\|\phi_0\|_{H^3}^2
+\|\mu_0\|_{H^2}^2
+\|g_0\|_{L^2}^2
+\|h_0\|_{H^1}^2 .
\end{aligned}
\end{equation}
Then there exists a nondecreasing function $M\in\mathcal F$ such that, for
any $0<t<T$,
\begin{equation}\label{apriori-est}
\Phi(t)+\int_0^t\mathcal D(s)\,\mathrm d s
\le
M\left(
\mathcal M_0+
\int_0^t M(\Phi(s))\,\mathrm d s
\right),
\end{equation}
where
\begin{equation}\label{def_D}
\begin{aligned}
\mathcal D(t):={}&
\|\nabla u_t(t)\|_{L^2}^2
+\|\nabla^2u(t)\|_{L^q}
+\|\nabla u(t)\|_{L^\infty}
\\
&+\|\nabla\phi_t(t)\|_{L^2}^2
+\|\nabla\mu_t(t)\|_{L^2}^2
+\|\sqrt\rho\,\mu_t(t)\|_{L^2}^2
+\|\sqrt\rho\,\phi_t(t)\|_{L^2}^2 .
\end{aligned}
\end{equation}
\end{proposition}

\begin{remark}\label{rem:apriori-closure}
As a consequence of Proposition~\ref{apriori-prop}, the a priori bound closes
on a sufficiently small time interval.  Indeed, set
\[
 K:=2M(\mathcal M_0+1).
\]
As long as $\Phi\le K$, choose $T_0>0$ such that
$T_0M(K)\le1$.  Then $\mathcal M_0+\int_0^tM(\Phi(s))\,\mathrm d s\le\mathcal M_0+1$ and
$\Phi(t)\le M(\mathcal M_0+1)=K/2$ for $0\le t\le T_0$.
The usual first-exit argument therefore gives
\[
\Phi(t)+\int_0^t\mathcal D(s)\,\mathrm d s
\le M(\mathcal M_0+1),
\qquad 0\le t\le T_0.
\]
\end{remark}

The proof of Proposition~\ref{apriori-prop} is based on the following
estimates. The only structural novelty, compared with the unperturbed NSCH
system \eqref{NSCH}, is the fixed coefficient $R$ in the chemical relation. Define
\begin{equation}\label{def_initial_energy}
\mathcal E_0
:=
\int
\left(
\frac12\rho_0|u_0|^2
+\frac{\rho_0^\gamma}{\gamma-1}
+\frac12|\nabla\phi_0|^2
+\rho_0\Psi(\phi_0)
+R\phi_0
\right)\,\mathrm d x .
\end{equation}
Since $R$ is independent of time, the definition of $\mathcal M_0$ gives
	\begin{equation}\label{fixed_residual_H1}
		\|R\|_{L^\infty(0,T;H^1(\Omega))}=\|R\|_{H^1(\Omega)}
		\le \mathcal M_0^{1/2}.
	\end{equation}

\subsection{Part I: Basic energy and the $H^1$ bound for $\phi$}
\begin{lemma}\label{Lemma_energy}
Let $(\rho,u,\phi,\mu)$ be a smooth solution to
\eqref{perturbed_system} with the boundary conditions stated in
Section~\ref{sec:perturbed-problem}, and let $R\in H^1(\Omega)$ be the
prescribed time-independent residual.  Suppose
$\int R\,\mathrm dx=0$.  Set
\[
\mathcal E(t):=\widetilde{\mathcal E}(t)+\int R\phi(t)\,\mathrm d x,
\]
where
\[
\widetilde{\mathcal E}(t):=\int\left(
\frac12\rho |u|^2+\frac{\rho^\gamma}{\gamma-1}
+\frac12|\nabla\phi|^2+\rho\Psi(\phi)\right)\,\mathrm d x.
\]
Then
\begin{equation}\label{total-energy}
\begin{aligned}
\mathcal E(t)+\int_0^t\int\left(
\lambda|\nabla u|^2+(\lambda+\eta)|\operatorname{div}u|^2
+|\nabla\mu|^2\right)\,\mathrm d x\,\mathrm d s
=\mathcal E_0-
\int_0^t\int R u\cdot\nabla\phi\,\mathrm d x\,\mathrm d s.
\end{aligned}
\end{equation}
Consequently, for $0\le t\le T\le1$,
\begin{equation}\label{total-energy-M}
\widetilde{\mathcal E}(t)+\int_0^t\left(
\|\nabla u(s)\|_{L^2}^2+\|\nabla\mu(s)\|_{L^2}^2
\right)\,\mathrm d s
\le M(\mathcal M_0),
\end{equation}
where $M$ is independent of the positive lower bound of $\rho$.
\end{lemma}
\begin{proof}
	Testing the momentum equation by $u$, the phase equation by $\mu$,
    and the chemical equation by $\phi_t$, and using the continuity
    equation and $R_t=0$, yields \eqref{total-energy}, with
    $\mathcal E_0=\mathcal E(0)$.
	Since $\int  R\,\mathrm dx=0$, Poincar\'e's and
	Young's inequalities give
	\[
	\left|\int  R\phi\,\mathrm dx\right|
	\le C\|R\|_{L^2}\|\nabla\phi\|_{L^2}
	\le \frac12 \widetilde{\mathcal E}(t)+C\|R\|_{H^1}^2.
	\]
	Moreover, by Sobolev embedding and Young's inequality,
	\[
	\left|\int  R u\cdot\nabla\phi\,\mathrm dx\right|
	\le C\|R\|_{H^1}\|\nabla u\|_{L^2}\|\nabla\phi\|_{L^2}
	\le \frac{c_0}{2}\|\nabla u\|_{L^2}^2
	+C\|R\|_{H^1}^2\widetilde{\mathcal E}(t).
	\]
	Substituting these bounds into \eqref{total-energy} yields
	\[
	\widetilde{\mathcal E}(t)+c\int_0^t
	\left(\|\nabla u\|_{L^2}^2+\|\nabla\mu\|_{L^2}^2\right)
	\,\mathrm ds
	\le C\left(\widetilde{\mathcal E}(0)+\|R\|_{H^1}^2\right)
	+C\|R\|_{H^1}^2\int_0^t \widetilde{\mathcal E}(s)\,\mathrm ds.
	\]
	Then, \eqref{total-energy-M} follows from Gronwall's inequality and $t\le1$, with constants independent of the density lower bound.
\end{proof}
As a consequence of Lemma~\ref{Lemma_energy}, we have the following
low-order estimate for the phase variable.
\begin{lemma}\label{Lemma_phi_H1}
	Let $(\rho,u,\phi,\mu)$ be a smooth solution to
	\eqref{perturbed_system}-\eqref{B_C} in $\Omega\times(0,T)$.
	Then
	\begin{equation}\label{est_phi}
		\sup_{0\le t\le T}\|\phi(t)\|_{H^1}
		\le
		M(\mathcal M_0).
	\end{equation}
\end{lemma}

\begin{proof}
	The inequality
	$
	|s|^4\le 8\Psi(s)+2,
	s\in\mathbb R,
	$, together with \eqref{total-energy-M}
	and conservation of mass, implies
	\[
	\int\rho|\phi|^4\,\mathrm d x
	\le 8\int\rho\Psi(\phi)\,\mathrm d x
	+2\int\rho\,\mathrm d x
	\le M(\mathcal M_0).
	\]
	H\"older's inequality further yields
	\[
	\left(\int\rho|\phi|\,\mathrm d x\right)^2
	\le
	\left(\int\rho\,\mathrm d x\right)^{3/2}
	\left(\int\rho|\phi|^4\,\mathrm d x\right)^{1/2}
	\le M(\mathcal M_0).
	\]
	Using Lemma \ref{Lemma_F03}, we therefore obtain
	\[
	\begin{aligned}
		\|\phi\|_{L^2}^2
		&\le
		C
		\left(
		\|\nabla\phi\|_{L^2}^2
		+
		\left(\int \rho |\phi|\,\mathrm d x\right)^2
		\right)  \le
		M(\mathcal M_0),
	\end{aligned}
	\]
	which gives \eqref{est_phi}.
\end{proof}
\subsection{Part II: Estimate for the velocity}

\begin{lemma}\label{Lemma_u_H2}
	Assume that $(\rho,u,\phi,\mu)$ is a smooth solution to
	\eqref{perturbed_system}-\eqref{B_C}. Then there exists a nondecreasing function
	$M\in\mathcal F$ such that, for a.e. $t\in(0,T)$,
	\begin{equation}\label{u_est_H2}
		\|\nabla u(t)\|_{H^1}\le M(\Phi(t)).
	\end{equation}
\end{lemma}

\begin{proof}
	By the $H^2$-estimate for the Lam\'e system with homogeneous Dirichlet
	boundary condition, it follows from \eqref{perturbed_system}$_2$ that
	\begin{equation}\label{elliptic_est}
		\begin{aligned}
			\|\nabla u\|_{H^1}
			\le C\big(
			&\|\rho u_t\|_{L^2}
			+\|\rho u\cdot\nabla u\|_{L^2}
			+\|\nabla p\|_{L^2}
			+\|\Delta\phi\nabla\phi\|_{L^2}
			+\|\nabla u\|_{L^2}
			\big).
		\end{aligned}
	\end{equation}
	We estimate the terms on the right-hand side. First,
	\[
	\|\rho u_t\|_{L^2}
	\le
	\|\rho\|_{L^\infty}^{1/2}
	\|\sqrt{\rho}\,u_t\|_{L^2}.
	\]
	By H\"older's inequality, the Gagliardo-Nirenberg inequality and
	Poincar\'e's inequality,
	\[
	\begin{aligned}
		\|\rho u\cdot\nabla u\|_{L^2}
		\le
		\|\rho\|_{L^\infty}
		\|u\|_{L^6}\|\nabla u\|_{L^3}
		\le
		C\|\rho\|_{L^\infty}
		\|\nabla u\|_{L^2}^{3/2}
		\|\nabla u\|_{H^1}^{1/2}.
	\end{aligned}
	\]
	Since $p=\rho^\gamma$, we have
	\[
	\|\nabla p\|_{L^2}
	\le
	C\|\rho\|_{L^\infty}^{\gamma-1}
	\|\nabla\rho\|_{L^2}.
	\]
	Moreover,
	\[
	\begin{aligned}
		\|\Delta\phi\nabla\phi\|_{L^2}
		\le
		\|\nabla\phi\|_{L^\infty}\|\Delta\phi\|_{L^2}
		\le
		C\|\phi\|_{H^3}\|\phi\|_{H^2}.
	\end{aligned}
	\]
	Substituting these estimates into \eqref{elliptic_est}, and using
	Young's inequality, we deduce
	\[
	\|\nabla u\|_{H^1}
	\le
	\frac12\|\nabla u\|_{H^1}+M(\Phi(t)).
	\]
	This gives \eqref{u_est_H2}.
\end{proof}

\begin{lemma}\label{Lemma_ut}
	Assume that $(\rho,u,\phi,\mu)$ is a smooth solution to
	\eqref{perturbed_system}-\eqref{B_C}. Then there exists a nondecreasing function
	$M\in\mathcal F$ such that, for a.e. $t\in(0,T)$,
	\begin{equation}\label{ut-closed}
		\|\sqrt{\rho}u_t(t)\|_{L^2}^2
		+\int_0^t\|\nabla u_t\|_{L^2}^2\,\mathrm d s
		\le
		C\mathcal M_0
		+\int_0^t M(\Phi(s))\,\mathrm d s .
	\end{equation}
\end{lemma}

\begin{proof}
	First, as a consequence of \eqref{u_est_H2}, it follows directly from the continuity equation \eqref{perturbed_system}$_1$ and Sobolev embedding that
	 \begin{equation}\label{rho_rt_est}
		\|\rho_t(t)\|_{L^2}
		\le
		\|u\|_{L^\infty}\|\nabla\rho\|_{L^2}
		+
		\|\rho\|_{L^\infty}\|\operatorname{div}u\|_{L^2}
		\le
		M(\Phi(t)).
	\end{equation}
	Then, differentiating \eqref{perturbed_system}$_2$ with respect to $t$, testing the resulting equation by $u_t$, and integrating over $\Omega$, we use the continuity equation \eqref{perturbed_system}$_1$ and $u_t=0$ on $\partial\Omega$ to obtain
	\begin{equation}\label{energy-ut}
		\begin{aligned}
			&\frac12\frac{\mathrm d}{\mathrm d t}
			\|\sqrt{\rho}u_t\|_{L^2}^2
			+\lambda\|\nabla u_t\|_{L^2}^2
			+(\lambda+\eta)\|\operatorname{div}u_t\|_{L^2}^2   \\
			={}&
			-2\int \rho u\cdot\nabla u_t\cdot u_t\,\mathrm d x
			-\int \rho_t(u\cdot\nabla)u\cdot u_t\,\mathrm d x
			-\int \rho(u_t\cdot\nabla)u\cdot u_t\,\mathrm d x
			+\gamma\int \rho^{\gamma-1}\rho_t\operatorname{div}u_t\,\mathrm d x  \\
			&+
			\int(\nabla\phi_t\cdot\nabla)u_t\cdot\nabla\phi\,\mathrm d x
			+\int(\nabla\phi\cdot\nabla)u_t\cdot\nabla\phi_t\,\mathrm d x
			-
			\int \nabla\phi\cdot\nabla\phi_t\,\operatorname{div}u_t\,\mathrm d x
			:=\sum_{i=1}^7 J_i .
		\end{aligned}
	\end{equation}
	Now we estimate $J_i$ as follows. By H\"older's inequality, Sobolev
	embedding, Poincar\'e's inequality, Lemma \ref{Lemma_u_H2} and \eqref{rho_rt_est}, we have
	\[
	\begin{aligned}
		|J_1|
		&\le
		C\|\rho\|_{L^\infty}^{1/2}
		\|\sqrt{\rho}u_t\|_{L^2}
		\|u\|_{L^\infty}
		\|\nabla u_t\|_{L^2}
		\le
		M(\Phi)\|\nabla u_t\|_{L^2},\\
		|J_2|
		&\le
		\|\rho_t\|_{L^2}
		\|u\|_{L^\infty}
		\|\nabla u\|_{L^3}
		\|u_t\|_{L^6}
		\le
		M(\Phi)\|\nabla u_t\|_{L^2},\\
		|J_3|
		&\le
		\|\rho\|_{L^\infty}^{1/2}
		\|\sqrt{\rho}u_t\|_{L^2}
		\|\nabla u\|_{L^3}
		\|u_t\|_{L^6}
		\le
		M(\Phi)\|\nabla u_t\|_{L^2},
	\end{aligned}
	\]
	and
	\[
	\begin{aligned}
		|J_4|
		\le
		C\|\rho\|_{L^\infty}^{\gamma-1}
		\|\rho_t\|_{L^2}
		\|\operatorname{div}u_t\|_{L^2}
		\le
		M(\Phi)\|\nabla u_t\|_{L^2}.
	\end{aligned}
	\]
	For the capillary terms, we have
	\[
	\begin{aligned}
		|J_5|+|J_6|+|J_7|
		\le
		C\|\nabla u_t\|_{L^2}
		\|\nabla\phi\|_{L^\infty}
		\|\nabla\phi_t\|_{L^2}
		\le
		C\|\nabla u_t\|_{L^2}
		\|\phi\|_{H^3}
		\|\nabla\phi_t\|_{L^2}
		\le
		M(\Phi)\|\nabla u_t\|_{L^2}.
	\end{aligned}
	\]
	Therefore, substituting
	the above estimates into \eqref{energy-ut} and utilizing Young's inequality, we obtain
	\begin{equation}\label{ut-diff-ineq}
		\frac{\mathrm d}{\mathrm d t}
		\|\sqrt{\rho}u_t\|_{L^2}^2
		+C\|\nabla u_t\|_{L^2}^2
		\le
		M(\Phi(t)),
	\end{equation}
	where $C>0$ is independent of the lower bound $\delta$. Evaluating \eqref{perturbed_system}$_2$ at $t=0$ and using
	the first compatibility condition in \eqref{App_C_C}, we obtain
	\[
		\|\sqrt{\rho_0}u_t(0)\|_{L^2}^2
		\le 2\|g_0\|_{L^2}^2
		+2\|\rho_0\|_{L^\infty}
		\|u_0\|_{L^\infty}^2\|\nabla u_0\|_{L^2}^2
		\le C\mathcal M_0,
	\]
	Integrating \eqref{ut-diff-ineq} over $(0,t)$ and using this
	initial bound yields \eqref{ut-closed}.
\end{proof}

The estimate for $\|\nabla u\|_{L^2}$ follows directly from the preceding time-derivative bound.

\begin{lemma}\label{Lemma_nabla_u}
	Assume that $(\rho,u,\phi,\mu)$ is a smooth solution to
	\eqref{perturbed_system}-\eqref{B_C}. Then there exists a nondecreasing function
	$M\in\mathcal F$ such that, for any $t\in(0,T)$,
	\begin{equation}\label{nabla_u-closed}
		\|\nabla u(t)\|_{L^2}^2
		\le
		C\mathcal M_0
		+
		\int_0^t M(\Phi(s))\,\mathrm d s .
	\end{equation}
\end{lemma}

\begin{proof}
	By the fundamental theorem of calculus, Young's inequality, and
	\eqref{ut-closed}, we obtain
	\[
	\begin{aligned}
		\|\nabla u(t)\|_{L^2}^2
		&=\|\nabla u_0\|_{L^2}^2
		+2\int_0^t\int \nabla u:\nabla u_t\,\mathrm dx\,\mathrm ds\\
		&\le \|\nabla u_0\|_{L^2}^2
		+\int_0^t\left(\|\nabla u\|_{L^2}^2
		+\|\nabla u_t\|_{L^2}^2\right)\,\mathrm ds\\
		&\le C\mathcal M_0+\int_0^t M(\Phi(s))\,\mathrm ds,
	\end{aligned}
	\]
	where we used $\|\nabla u(s)\|_{L^2}^2\le\Phi(s)$ and enlarged
	$M$ if necessary. This proves \eqref{nabla_u-closed}.
\end{proof}

\subsection{Part III: Estimates for the density}

\begin{lemma}\label{nabla2_u_est}
	Assume that $(\rho,u,\phi,\mu)$ is a smooth solution to
	\eqref{perturbed_system}-\eqref{B_C}. Then there exists a nondecreasing function
	$M\in\mathcal F$ such that, for any $t\in(0,T)$,
	\begin{equation}\label{nabla2_u_integral_est}
		\int_0^t
		\left(
		\|\nabla^2u\|_{L^q}
		+
		\|\nabla u\|_{L^\infty}
		\right)\,\mathrm d s
		\le
		C\mathcal M_0
		+
		\int_0^t M(\Phi(s))\,\mathrm d s .
	\end{equation}
\end{lemma}

\begin{proof}
	By the $W^{2,q}$-estimate for the Lam\'e system with homogeneous
	Dirichlet boundary condition, we have
	\begin{equation}\label{W2q_Lame_est}
		\begin{aligned}
			\|u\|_{W^{2,q}}
			\le C\big(
			&\|\rho u_t\|_{L^q}
			+
			\|\rho u\cdot\nabla u\|_{L^q}
			+
			\|\nabla(\rho^\gamma)\|_{L^q}
			+
			\|\Delta\phi\nabla\phi\|_{L^q}
			\big).
		\end{aligned}
	\end{equation}
	In particular, by Poincar\'e's inequality,
	\[
	\|\nabla^2u\|_{L^q}
	\le
	C\|u\|_{W^{2,q}}.
	\]
	We estimate the terms on the right-hand side of \eqref{W2q_Lame_est}.
	Since $3<q\le 6$ and $\Omega\subset\mathbb R^3$ is bounded, the Sobolev embedding
	$H^1_0(\Omega)\hookrightarrow L^q(\Omega)$ gives
	\[
	\|\rho u_t\|_{L^q}
	\le
	\|\rho\|_{L^\infty}\|u_t\|_{L^q}
	\le
	C\|\rho\|_{L^\infty}\|\nabla u_t\|_{L^2} \le M(\Phi)\|\nabla u_t\|_{L^2}.
	\]
	Moreover, using the three-dimensional embedding $H^2(\Omega)\hookrightarrow L^\infty(\Omega)$ and \eqref{u_est_H2}, we get
	\[
	\begin{aligned}
		\|\rho u\cdot\nabla u\|_{L^q}
		\le
		\|\rho\|_{L^\infty}
		\|u\|_{L^\infty}
		\|\nabla u\|_{L^q}
		\le
		C\|\rho\|_{L^\infty}\|u\|_{H^2}\|\nabla u\|_{H^1}
		\le
		M(\Phi).
	\end{aligned}
	\]
	For the pressure term, we have
	\[
	\|\nabla(\rho^\gamma)\|_{L^q}
	\le
	C\|\rho\|_{L^\infty}^{\gamma-1}
	\|\nabla\rho\|_{L^q}
	\le
	M(\Phi).
	\]
	Similarly, for the capillary term, one computes
	\[
	\begin{aligned}
		\|\Delta\phi\nabla\phi\|_{L^q}
		&\le
		\|\nabla\phi\|_{L^\infty}
		\|\Delta\phi\|_{L^q}
		\le
		C\|\phi\|_{H^3}\|\phi\|_{H^3}
		\le
		M(\Phi).
	\end{aligned}
	\]
	Consequently,
	\[
	\|\nabla^2u\|_{L^q}
	\le
	M(\Phi)\|\nabla u_t\|_{L^2}
	+
	M(\Phi).
	\]
	By Young's inequality and the fact that the exact form of $M$ may vary
	from line to line,
	\[
	\|\nabla^2u\|_{L^q}
	\le
	\|\nabla u_t\|_{L^2}^2
	+
	M(\Phi).
	\]
	Integrating over $(0,t)$ and using \eqref{ut-closed}, we obtain
	\begin{equation}\label{nabla2u_Lq_int}
		\int_0^t\|\nabla^2u\|_{L^q}\,\mathrm d s
		\le
		C\mathcal M_0
		+
		\int_0^t M(\Phi(s))\,\mathrm d s .
	\end{equation}
	It follows from Sobolev embedding and Poincaré's inequality that
	\[
	\|\nabla u\|_{L^\infty}
	\le
	C\|u\|_{W^{2,q}}
	\le
	C\|\nabla^2u\|_{L^q}
	+
	C\|\nabla u\|_{L^2}.
	\]
	Therefore, by \eqref{nabla2u_Lq_int} and $\|\nabla u(s)\|_{L^2}^2\le \Phi(s)$, one has
	\[
	\begin{aligned}
		\int_0^t\|\nabla u\|_{L^\infty}\,\mathrm d s
		\le
		C\mathcal M_0
		+
		\int_0^t M(\Phi(s))\,\mathrm d s .
	\end{aligned}
	\]
	Combining this with \eqref{nabla2u_Lq_int} gives
	\eqref{nabla2_u_integral_est}.
\end{proof}

\begin{lemma}\label{Lemma_rho_W1q}
	Assume that $(\rho,u,\phi,\mu)$ is a smooth solution to
	\eqref{perturbed_system}-\eqref{B_C}. Then there exists a nondecreasing function
	$M\in\mathcal F$ such that, for any $t\in(0,T)$,
	\begin{equation}\label{rho-closed-simple}
		\|\rho(t)\|_{W^{1,q}}
		\le
		\exp\left\{
		C\mathcal M_0
		+
		\int_0^t M(\Phi(s))\,\mathrm d s
		\right\}.
	\end{equation}
\end{lemma}

\begin{proof}
	Applying Lemma~\ref{app:transport-lemma} with $v=u$, we obtain
	\[
	\|\rho(t)\|_{W^{1,q}}
	\le
	\|\rho_0\|_{W^{1,q}}
	\exp\left(
	C\int_0^t\|u(s)\|_{W^{2,q}}\,\mathrm ds
	\right).
	\]
	Since $u=0$ on $\partial\Omega$, Poincar\'e's inequality gives
	\[
	\|u\|_{W^{2,q}}
	\le C\left(
	\|\nabla u\|_{L^\infty}
	+\|\nabla^2u\|_{L^q}
	\right).
	\]
	Combining this with \eqref{nabla2_u_integral_est} and absorbing
	the initial factor into the exponential by the definition
	of $\mathcal M_0$, we obtain \eqref{rho-closed-simple},
	after enlarging $C$ and $M\in\mathcal F$ if necessary.
\end{proof}

%\begin{corollary}\label{low_energy_closed}
%	Under the assumptions above, one has
%	\[
%	E^0(t)
%	+
%	\int_0^t
%	\left(
%	\|\nabla u(s)\|_{L^2}^2
%	+
%	\|\nabla\mu(s)\|_{L^2}^2
%	\right)\,ds
%	\le
%	M\left(
%	\mathcal M_0+\int_0^tM(\Phi(s))\,ds
%	\right).
%	\]
%\end{corollary}
%
%\begin{proof}
%	By Lemma~\ref{Lemma_energy},
%	\[
%	E^0(t)
%	+
%	\int_0^t
%	\left(
%	\|\nabla u(s)\|_{L^2}^2
%	+
%	\|\nabla\mu(s)\|_{L^2}^2
%	\right)\,ds
%	\le M(\mathcal M_0).
%	\]
%	This proves the assertion.
%\end{proof}

\subsection{Part IV: Estimates for $\mu$ and $\nabla\phi_t$}

\begin{lemma}\label{Lemma_rho_mu_phit}
	Assume that $(\rho,u,\phi,\mu)$ is a sufficiently smooth solution to
	\eqref{perturbed_system}-\eqref{B_C} on $\Omega\times(0,T)$. Then there exists a
	nondecreasing function $M\in\mathcal F$ such that, for any $t\in(0,T)$,
	\begin{equation}\label{rho_mu_phit_est}
		\|\sqrt{\rho}\,\mu(t)\|_{L^2}^2
		+
		\int_0^t\|\sqrt{\rho}\,\phi_t\|_{L^2}^2\,\mathrm d s
		\le
		C\mathcal M_0
		+
		\int_0^t M(\Phi(s))\,\mathrm d s .
	\end{equation}
\end{lemma}
\begin{proof}
	Testing \eqref{perturbed_system}$_3$ and
	\eqref{perturbed_system}$_4$ by $\phi_t$ and $\mu_t$,
	respectively, and differentiating the identity obtained by
	testing \eqref{perturbed_system}$_4$ by $\mu$ with respect to
	$t$, we use the continuity equation \eqref{perturbed_system}$_1$ and integrate by parts
	to obtain
	\begin{equation}\label{rho_mu_energy}
		\begin{aligned}
			\frac12\frac{\mathrm d}{\mathrm dt}
			\|\sqrt{\rho}\,\mu\|_{L^2}^2
			+\|\sqrt{\rho}\,\phi_t\|_{L^2}^2
			=&
			-\int \rho\mu u\cdot\nabla\mu\,\mathrm dx
			-\int \rho\phi_t u\cdot\nabla\phi\,\mathrm dx\\
			&+\int \rho u\cdot\nabla\mu\,F(\phi)\,\mathrm dx
			+\int \rho\mu u\cdot\nabla\phi\,F'(\phi)\,\mathrm dx\\
			&+\int \rho\mu F'(\phi)\phi_t\,\mathrm dx
			=: \sum_{i=1}^5 G_i.
		\end{aligned}
	\end{equation}
	By H\"older's and Young's inequalities and the preceding estimates,
	we have
	\[
	\begin{aligned}
		|G_1|
		&\le
		\|\rho\|_{L^\infty}^{1/2}
		\|\sqrt{\rho}\,\mu\|_{L^2}
		\|u\|_{L^\infty}
		\|\nabla\mu\|_{L^2}
		\le M(\Phi),
		\\
		|G_2|
		&\le
		\|\rho\|_{L^\infty}^{1/2}
		\|\sqrt{\rho}\,\phi_t\|_{L^2}
		\|u\|_{L^\infty}
		\|\nabla\phi\|_{L^2}
		 \le
		\frac14\|\sqrt{\rho}\,\phi_t\|_{L^2}^2
		+
		M(\Phi),
		\\
		|G_3|
		&\le
		\|\rho\|_{L^2}
		\|u\|_{L^\infty}
		\|\nabla\mu\|_{L^2}
		\|F(\phi)\|_{L^\infty}
		\le M(\Phi),
		\\
		|G_4|
		&\le
		\|\rho\|_{L^\infty}^{1/2}
		\|\sqrt{\rho}\,\mu\|_{L^2}
		\|u\|_{L^\infty}
		\|\nabla\phi\|_{L^2}
		\|F'(\phi)\|_{L^\infty}
		\le M(\Phi),
		\\
		|G_5|
		&\le
		\|\sqrt{\rho}\,\mu\|_{L^2}
		\|\sqrt{\rho}\,\phi_t\|_{L^2}
		\|F'(\phi)\|_{L^\infty}
		 \le
		\frac14\|\sqrt{\rho}\,\phi_t\|_{L^2}^2
		+
		M(\Phi).
	\end{aligned}
	\]
	Substituting the estimates for \(G_i\) into \eqref{rho_mu_energy} and absorbing the corresponding terms into the left-hand side, we
	obtain
	\begin{equation}\label{rho_mu_diff_ineq}
		\frac{\mathrm d}{\mathrm d t}
		\|\sqrt{\rho}\,\mu\|_{L^2}^2
		+
		\|\sqrt{\rho}\,\phi_t\|_{L^2}^2
		\le
		M(\Phi).
	\end{equation}
	Integrating \eqref{rho_mu_diff_ineq} over \((0,t)\) and using
	$\|\sqrt{\rho_0}\,\mu_0\|_{L^2}^2\le C\mathcal M_0$, we
	obtain \eqref{rho_mu_phit_est}.
\end{proof}

\begin{lemma}\label{Lemma_mu_phi_t}
	Assume that $(\rho,u,\phi,\mu)$ is a sufficiently smooth solution to
	\eqref{perturbed_system}-\eqref{B_C} on $\Omega\times(0,T)$. Then there exists a
	nondecreasing function $M\in\mathcal F$ such that, for any $t\in(0,T)$,
	\begin{equation}\label{phi_mu_t_energy}
		\|\mu(t)\|_{H^1}^2
		+
		\int_0^t\|\nabla\phi_t\|_{L^2}^2\,\mathrm d s
		\le
		C\mathcal M_0
		+
		\int_0^t M(\Phi(s))\,\mathrm d s .
	\end{equation}
\end{lemma}

\begin{proof}
	Differentiating \eqref{perturbed_system}$_4$ with respect to $t$
	and testing by $\phi_t$, then testing
	\eqref{perturbed_system}$_3$ by $\mu_t$, we combine the resulting
	identities and integrate by parts to obtain
	\begin{align} \label{final_identity}
			\frac12\frac{\mathrm d}{\mathrm d t}\|\nabla\mu\|_{L^2}^2
			+
			\|\nabla\phi_t\|_{L^2}^2
			={}&
			\int u\cdot\nabla\phi\,\Delta\phi_t\,\mathrm d x
			-
			\int u\cdot\nabla\phi\,\rho_tF(\phi)\,\mathrm d x
            -\int u\cdot\nabla\phi\,\rho F'(\phi)\phi_t\,\mathrm d x  \nonumber
			\\
			&+\int u\cdot\nabla\phi\,\rho_t\mu\,\mathrm d x
			+\int \rho_t\mu\phi_t\,\mathrm d x
			-\int \rho_tF(\phi)\phi_t\,\mathrm d x
            \\
			&-\int \rho F'(\phi)\phi_t^2\,\mathrm d x
			:=\sum_{i=1}^7L_i.  \nonumber
	\end{align}
	We first record that, by Lemma~\ref{Lemma_F03} and
	\eqref{rho_mu_phit_est},
	\[
	\|\phi_t\|_{L^2}
	\le
	C\left(
	\|\nabla\phi_t\|_{L^2}
	+
	\|\sqrt{\rho}\,\phi_t\|_{L^2}
	\right),
	\]
	and
	\begin{equation}\label{est_mu}
	\|\mu\|_{L^2}
	\le
	C\|\nabla\mu\|_{L^2}
	+
	C\|\sqrt{\rho}\,\mu\|_{L^2}.
	\end{equation}
	For $L_1$, after integrating by parts, we obtain from H\"older's and Young's inequalities that
	\[
	\begin{aligned}
		|L_1|
        &=
	   \left|\int \nabla\phi_t\cdot((\nabla u)\nabla\phi)\,\mathrm dx
	   +\int \nabla\phi_t\cdot((u\cdot\nabla)\nabla\phi)\,\mathrm dx\right|\\
		&\le
		\|\nabla\phi_t\|_{L^2}
		\left(
		\|\nabla u\|_{L^3}\|\nabla\phi\|_{L^6}
		+\|u\|_{L^\infty}\|\nabla^2\phi\|_{L^2}
		\right)\\
		&\le
		\frac18\|\nabla\phi_t\|_{L^2}^2+M(\Phi).
	\end{aligned}
	\]
	Similarly, using the preceding estimates, we have
	\[
	\begin{aligned}
		|L_2|
		&\le
		\|u\|_{L^\infty}
		\|\nabla\phi\|_{L^2}
		\|\rho_t\|_{L^2}
		\|F(\phi)\|_{L^\infty}
		\le M(\Phi),
		\\
		|L_3|
		&\le
		\|u\|_{L^\infty}
		\|\nabla\phi\|_{L^2}
		\|\rho\|_{L^\infty}
		\|F'(\phi)\|_{L^\infty}
		\|\phi_t\|_{L^2}
		\le
		\frac18\|\nabla\phi_t\|_{L^2}^2+M(\Phi),
		\\
		|L_4|
		&\le
		\|u\|_{L^\infty}
		\|\nabla\phi\|_{L^\infty}
		\|\rho_t\|_{L^2}
		\|\mu\|_{L^2}
		\le M(\Phi),
		\\
		|L_5|
		&\le
		\|\rho_t\|_{L^2}
		\|\mu\|_{L^3}
		\|\phi_t\|_{L^6}
		\le
		\frac18\|\nabla\phi_t\|_{L^2}^2+M(\Phi),
		\\
		|L_6|
		&\le
		\|\rho_t\|_{L^2}
		\|F(\phi)\|_{L^\infty}
		\|\phi_t\|_{L^2}
		\le
		\frac18\|\nabla\phi_t\|_{L^2}^2+M(\Phi),
		\\
		|L_7|
		&\le
		\|\sqrt{\rho}\,\phi_t\|_{L^2}^2
		\|F'(\phi)\|_{L^\infty}
		\le M(\Phi).
	\end{aligned}
	\]
	Substituting these bounds into \eqref{final_identity} and
	absorbing the terms involving $\|\nabla\phi_t\|_{L^2}^2$
	yield
	\[
	\frac{\mathrm d}{\mathrm dt}\|\nabla\mu\|_{L^2}^2
	+c\|\nabla\phi_t\|_{L^2}^2
	\le M(\Phi)
	\]
	for some constant $c>0$.
	Integrating it over $(0,t)$ and using
	$\|\nabla\mu_0\|_{L^2}^2\le C\mathcal M_0$, we obtain,
	after enlarging $C$ and $M$ if necessary,
	\[
	\|\nabla\mu(t)\|_{L^2}^2
	+\int_0^t\|\nabla\phi_t(s)\|_{L^2}^2\,\mathrm ds
	\le
	C\mathcal M_0+\int_0^t M(\Phi(s))\,\mathrm ds.
	\]
	This, together with Lemma \ref{Lemma_rho_mu_phit} and \eqref{est_mu},
	proves \eqref{phi_mu_t_energy}.
\end{proof}

\begin{lemma}\label{Lemma_phi_H2}
	Assume that $(\rho,u,\phi,\mu)$ is a sufficiently smooth solution to
	\eqref{perturbed_system}-\eqref{B_C}. Then there exists a nondecreasing function
	$M\in\mathcal F$ such that, for any $t\in(0,T)$,
	\begin{equation}\label{phi_H2_low_est}
		\|\phi(t)\|_{H^2}
		\le
		M\left(
		\mathcal M_0+\int_0^tM(\Phi(s))\,\mathrm d s
		\right).
	\end{equation}
\end{lemma}

\begin{proof}
	From \eqref{perturbed_system}$_4$, using the Neumann elliptic estimate together with $\partial_n\phi=0$ on $\partial\Omega$, we obtain
	\[
	\|\phi\|_{H^2}
	\le
	C\left(
	\|\rho\mu\|_{L^2}
	+
	\|\rho(\phi^3-\phi)\|_{L^2}
	+
	\|R\|_{L^2}
	+
	\|\phi\|_{L^2}
	\right).
	\]
	We estimate the terms on the right-hand side. First,
	\[
	\|\rho\mu\|_{L^2}
	\le
	\|\rho\|_{L^\infty}^{1/2}
	\|\sqrt{\rho}\mu\|_{L^2}.
	\]
	Moreover, since $H^1(\Omega)\hookrightarrow L^6(\Omega)$ in three dimensions,
	\[
	\begin{aligned}
		\|\rho(\phi^3-\phi)\|_{L^2}
		\le
		\|\rho\|_{L^\infty}
		\left(
		\|\phi\|_{L^6}^3+\|\phi\|_{L^2}
		\right)
		\le
		C\|\rho\|_{L^\infty}
		\left(
		\|\phi\|_{H^1}^3+\|\phi\|_{H^1}
		\right).
	\end{aligned}
	\]
	By Lemmas \ref{Lemma_phi_H1}, \ref{Lemma_rho_W1q},
	\ref{Lemma_rho_mu_phit}, and \eqref{fixed_residual_H1}, we have
	\[
	\|\phi\|_{H^1}
	+
	\|\sqrt{\rho}\mu\|_{L^2}
	+
	\|\rho\|_{L^\infty}
	+
	\|R\|_{L^2}
	\le
	M\left(
	\mathcal M_0+\int_0^tM(\Phi(s))\,\mathrm d s
	\right).
	\]
	Substituting these estimates into the elliptic estimate yields
	\[
	\|\phi(t)\|_{H^2}
	\le
	M\left(
	\mathcal M_0+\int_0^tM(\Phi(s))\,\mathrm d s
	\right).
	\]
	This completes the proof.
\end{proof}

\subsection{Part V: High-order joint estimate for the phase variables}
The main difficulty in the high-order phase estimate is the occurrence of the terms
	\[
	\int \phi_{tt}\rho_t
	\left[\mu-(\phi^3-\phi)\right]\,\mathrm d x
	-
	\int \phi_{tt}\rho(3\phi^2-1)\phi_t\,\mathrm d x,
	\]
	for which the available estimates do not provide direct control.
	We therefore rewrite their sum as
	$\frac12\frac{\mathrm d}{\mathrm dt}\mathcal B(t)$
	plus remainders involving at most first time derivatives.
	This motivates the corrected energy
	$\|\nabla\phi_t\|_{L^2}^2-\mathcal B(t)$,
	with $\mathcal B$ defined in \eqref{def_phase_correction}.
	Adding a suitable multiple of
	$\|\sqrt{\rho}\phi_t\|_{L^2}^2$
	then restores coercivity up to terms controlled by the preceding estimates.

\begin{lemma}\label{Lemma_phase_gradient_modified}
	Assume that $(\rho,u,\phi,\mu)$ is a sufficiently smooth solution to
	\eqref{perturbed_system}-\eqref{B_C}. For every sufficiently small
	$\varepsilon>0$, there exist $c_1>0$ and a nondecreasing function
	$M\in\mathcal F$, independent of $\delta$, such that, for any $t\in(0,T)$,
	\begin{equation}\label{phit_mut_modified_joint}
		\frac{\mathrm d}{\mathrm d t}
		\left(
		\|\nabla\phi_t\|_{L^2}^2-\mathcal B(t)
		\right)
		+
		c_1\|\nabla\mu_t\|_{L^2}^2
		\le
		\varepsilon\|\nabla u_t\|_{L^2}^2
		+
		M(\Phi(t)),
	\end{equation}
	where
	\begin{equation}\label{def_phase_correction}
		\begin{aligned}
			\mathcal B(t):={}&
			2\int \nabla\phi_t\cdot\rho u
			\left[\mu-(\phi^3-\phi)\right]\,\mathrm d x
            -\int \rho(3\phi^2-1)\phi_t^2\,\mathrm d x
			\\
            &+2\int \phi_t\rho u\cdot
			\left[
			\nabla\mu-(3\phi^2-1)\nabla\phi
			\right]\,\mathrm d x.
		\end{aligned}
	\end{equation}
\end{lemma}

\begin{proof}
	We differentiate \eqref{perturbed_system}$_3$ with respect to $t$
	and test the resulting equation by $\mu_t$.
	Using the time derivative of \eqref{perturbed_system}$_4$
	and integrating by parts, we obtain
	\begin{equation}\label{energy_phit_joint}
		\begin{aligned}
			&\frac12\frac{\mathrm d}{\mathrm d t}
			\|\nabla\phi_t\|_{L^2}^2
			+
			\|\nabla\mu_t\|_{L^2}^2
			\\
			={}&
			-\int \rho_t\phi_t\mu_t\,\mathrm d x
			-\int \rho_tu\cdot\nabla\phi\,\mu_t\,\mathrm d x
			-\int \rho u_t\cdot\nabla\phi\,\mu_t\,\mathrm d x
			-\int \rho u\cdot\nabla\phi_t\,\mu_t\,\mathrm d x
			\\
			&+
			\int \phi_{tt}\rho_t
			\left[\mu-(\phi^3-\phi)\right]\,\mathrm d x
			-
			\int \phi_{tt}\rho(3\phi^2-1)\phi_t\,\mathrm d x
			:=\sum_{i=1}^6 H_i .
		\end{aligned}
	\end{equation}
	
	We first estimate $\mu_t$. Integrating the differentiated
	chemical relation and using H\"older's inequality, we obtain
	\[
	\left|\int\rho\mu_t\,\mathrm dx\right|
	=
	\left|
	\int\rho_t(\phi^3-\phi-\mu)\,\mathrm dx
	+\int\rho(3\phi^2-1)\phi_t\,\mathrm dx
	\right|
	\le M(\Phi(t)).
	\]
	Let $\overline{\mu_t}:=|\Omega|^{-1}\int\mu_t\,\mathrm dx$.
	Since $m:=\int\rho\,\mathrm dx
	=\int\rho_0\,\mathrm dx>0$ is conserved,
	H\"older's and Sobolev--Poincar\'e inequalities give
	\[
	\begin{aligned}
		m|\overline{\mu_t}|
		&\le
		\left|\int\rho\mu_t\,\mathrm dx\right|
		+\left|\int\rho(\mu_t-\overline{\mu_t})\,\mathrm dx\right|\\
		&\le
		\left|\int\rho\mu_t\,\mathrm dx\right|
		+C\|\rho\|_{L^{6/5}}\|\nabla\mu_t\|_{L^2}.
	\end{aligned}
	\]
	Consequently,
	\begin{equation}\label{mut_L3_joint}
		\begin{aligned}
			\|\mu_t\|_{L^3}
			&\le
			\|\mu_t-\overline{\mu_t}\|_{L^3}
			+|\Omega|^{1/3}|\overline{\mu_t}|
			\le
			C\|\nabla\mu_t\|_{L^2}
			+C|\overline{\mu_t}|\\
			&\le
			C(1+\|\rho\|_{L^{6/5}})\|\nabla\mu_t\|_{L^2}
			+C\left|\int\rho\mu_t\,\mathrm dx\right|\\
			&\le
			M(\Phi(t))
			\left(1+\|\nabla\mu_t\|_{L^2}\right).
		\end{aligned}
	\end{equation}
	Moreover, Lemma \ref{Lemma_F03},
	H\"older's inequality, and Sobolev embedding yield
	\begin{equation}\label{phit_L6_joint}
		\begin{aligned}
			\|\phi_t\|_{L^6}
			&\le C\left(
			\|\nabla\phi_t\|_{L^2}
			+\int\rho|\phi_t|\,\mathrm dx
			\right)\\
			&\le C\left(
			\|\nabla\phi_t\|_{L^2}
			+\|\sqrt{\rho}\phi_t\|_{L^2}
			\right)
			\le M(\Phi(t)).
		\end{aligned}
	\end{equation}
	
	For $H_i, i=1,\cdots, 4$, relations
	\eqref{mut_L3_joint} and \eqref{phit_L6_joint} give
	\[
	\begin{aligned}
		|H_1|
		&\le
		\|\rho_t\|_{L^2}
		\|\phi_t\|_{L^6}
		\|\mu_t\|_{L^3}
		\le
		M(\Phi(t))
		\left(
		\|\nabla\mu_t\|_{L^2}
		+
		1
		\right)
		\le
		\varepsilon\|\nabla\mu_t\|_{L^2}^2
		+
		M(\Phi(t)),\\
		|H_2|
		&\le
		\|\rho_t\|_{L^2}
		\|u\|_{L^\infty}
		\|\nabla\phi\|_{L^6}
		\|\mu_t\|_{L^3}
		\le
		\varepsilon\|\nabla\mu_t\|_{L^2}^2
		+
		M(\Phi(t)),\\
		|H_3|
		&\le
		\|\sqrt{\rho}u_t\|_{L^2}
		\|\rho\|_{L^\infty}^{1/2}
		\|\nabla\phi\|_{L^6}
		\|\mu_t\|_{L^3}
		\le
		\varepsilon\|\nabla\mu_t\|_{L^2}^2
		+
		M(\Phi(t)),\\
		|H_4|
		&\le
		\|\rho\|_{L^\infty}
		\|u\|_{L^6}
		\|\nabla\phi_t\|_{L^2}
		\|\mu_t\|_{L^3}
		\le
		\varepsilon\|\nabla\mu_t\|_{L^2}^2
		+
		M(\Phi(t)).
	\end{aligned}
	\]
	Next we deal with $H_5$. From \eqref{perturbed_system}$_1$,
	integrating by parts in time gives
	\[
	\begin{aligned}
		H_5
		={}&
		\frac{\mathrm d}{\mathrm d t}
		\bigg\{
		\int \nabla\phi_t\cdot\rho u
		\left[\mu-(\phi^3-\phi)\right]\,\mathrm d x
		+
		\int \phi_t\rho u\cdot
		\left[
		\nabla\mu-(3\phi^2-1)\nabla\phi
		\right]\,\mathrm d x
		\bigg\}  \\
		&-
		\int \nabla\phi_t\cdot\rho_tu
		\left[\mu-(\phi^3-\phi)\right]\,\mathrm d x
		-
		\int \nabla\phi_t\cdot\rho u_t
		\left[\mu-(\phi^3-\phi)\right]\,\mathrm d x  \\
		&-
		\int \nabla\phi_t\cdot\rho u
		\left[
		\mu_t-(3\phi^2-1)\phi_t
		\right]\,\mathrm d x
		-
		\int \phi_t\rho_t u\cdot
		\left[
		\nabla\mu-(3\phi^2-1)\nabla\phi
		\right]\,\mathrm d x  \\
		&-
		\int \phi_t\rho u_t\cdot
		\left[
		\nabla\mu-(3\phi^2-1)\nabla\phi
		\right]\,\mathrm d x
		-
		\int \phi_t\rho u\cdot
		\left[
		\nabla\mu_t
		-(3\phi^2-1)\nabla\phi_t
		-6\phi\phi_t\nabla\phi
		\right]\,\mathrm d x  \\
		=:&
		\frac{\mathrm d}{\mathrm d t}\bigg\{
		\int \nabla\phi_t\cdot\rho u
		\left[\mu-(\phi^3-\phi)\right]\,\mathrm d x
		+
		\int \phi_t\rho u\cdot
		\left[
		\nabla\mu-(3\phi^2-1)\nabla\phi
		\right]\,\mathrm d x
		\bigg\}
		+
		\sum_{j=1}^6 R_j ,
	\end{aligned}
	\]
	By \eqref{mut_L3_joint}, \eqref{phit_L6_joint}, Sobolev embedding and Young's inequality, we have
	\[
	\begin{aligned}
		|R_1|
		&\le
		\|\nabla\phi_t\|_{L^2}
		\|\rho_t\|_{L^2}
		\|u\|_{L^\infty}
		\left\|\mu-(\phi^3-\phi)\right\|_{L^\infty}
		\le
		M(\Phi(t)),\\
		|R_2|
		&\le
		\|\nabla\phi_t\|_{L^2}
		\|\sqrt{\rho}u_t\|_{L^2}
		\|\rho\|_{L^\infty}^{1/2}
		\left\|\mu-(\phi^3-\phi)\right\|_{L^\infty}
		\le
		M(\Phi(t)),\\
		|R_3|
		&\le
		\|\nabla\phi_t\|_{L^2}
		\|\rho\|_{L^\infty}
		\|u\|_{L^6}
		\|\mu_t\|_{L^3}
		+
		\|\nabla\phi_t\|_{L^2}
		\|\rho\|_{L^\infty}
		\|u\|_{L^3}
		\|3\phi^2-1\|_{L^\infty}
		\|\phi_t\|_{L^6}  \\
		&\le
		M(\Phi(t))
		\left(
		\|\nabla\mu_t\|_{L^2}
		+
		1
		\right)
		\le
		\varepsilon\|\nabla\mu_t\|_{L^2}^2
		+
		M(\Phi(t)),\\
		|R_4|
		&\le
		\|\phi_t\|_{L^6}
		\|\rho_t\|_{L^2}
		\|u\|_{L^\infty}
		\left\|
		\nabla\mu-(3\phi^2-1)\nabla\phi
		\right\|_{L^3}
		\le
		M(\Phi(t)),\\
		|R_5|
		&\le
		\|\phi_t\|_{L^6}
		\|\sqrt{\rho}u_t\|_{L^2}
		\|\rho\|_{L^\infty}^{1/2}
		\left\|
		\nabla\mu-(3\phi^2-1)\nabla\phi
		\right\|_{L^3}
		\le
		M(\Phi(t)).
	\end{aligned}
	\]
	Finally, for $R_6$, we estimate the three terms separately:
	\[
	\begin{aligned}
		|R_6|
		&\le
		\|\phi_t\|_{L^6}
		\|\rho\|_{L^\infty}
		\|u\|_{L^3}
		\|\nabla\mu_t\|_{L^2}
		+
		\|\phi_t\|_{L^6}
		\|\rho\|_{L^\infty}
		\|u\|_{L^3}
		\|3\phi^2-1\|_{L^\infty}
		\|\nabla\phi_t\|_{L^2}  \\
		&\quad
		+
		C\|\phi_t\|_{L^6}^2
		\|\rho\|_{L^\infty}
		\|u\|_{L^6}
		\|\phi\|_{L^\infty}
		\|\nabla\phi\|_{L^2}  \\
		&\le
		M(\Phi(t))
		\|\nabla\mu_t\|_{L^2}
		+
		M(\Phi(t)) \\
		&\le
		\varepsilon\|\nabla\mu_t\|_{L^2}^2
		+
		M(\Phi(t)).
	\end{aligned}
	\]
	Consequently,
	\begin{equation}\label{H5_joint_est}
		\begin{aligned}
			H_5
			\le{}&
			\frac{\mathrm d}{\mathrm d t}\bigg\{
			\int \nabla\phi_t\cdot\rho u
			\left[\mu-(\phi^3-\phi)\right]\,\mathrm d x
			+
			\int \phi_t\rho u\cdot
			\left[
			\nabla\mu-(3\phi^2-1)\nabla\phi
			\right]\,\mathrm d x
			\bigg\}
			+
			\varepsilon\|\nabla\mu_t\|_{L^2}^2
			+
			M(\Phi(t)).
		\end{aligned}
	\end{equation}
	For $H_6$, a similar argument shows
	\begin{align}
		H_6
		&=
		-\frac12\frac{\mathrm d}{\mathrm d t}
		\int \rho(3\phi^2-1)\phi_t^2\,\mathrm d x
		+
		\frac12\int \rho_t(3\phi^2-1)\phi_t^2\,\mathrm d x
		+
		3\int \rho\phi\phi_t^3\,\mathrm d x \\\nonumber
        &\le
        -\frac12\frac{\mathrm d}{\mathrm d t}
		\int \rho(3\phi^2-1)\phi_t^2\,\mathrm d x
		+
        C\|\rho_t\|_{L^2}
		\|3\phi^2-1\|_{L^6}
		\|\phi_t\|_{L^6}^2
        +
        C\|\rho\|_{L^\infty}^{1/2}
		\|\phi\|_{L^6}
		\|\sqrt{\rho}\phi_t\|_{L^2}
		\|\phi_t\|_{L^6}^2\\
        &\le
        -\frac12\frac{\mathrm d}{\mathrm d t}
		\int \rho(3\phi^2-1)\phi_t^2\,\mathrm d x
		+
		M(\Phi(t)).\label{H6_joint_est}
	\end{align}
	Combining the estimates for $H_1,\ldots,H_6$, and choosing
	$\varepsilon>0$ sufficiently small, we obtain  \eqref{phit_mut_modified_joint}.
\end{proof}

\begin{lemma}\label{Lemma_phase_high_joint}
	Assume that $(\rho,u,\phi,\mu)$ is a sufficiently smooth solution to
	\eqref{perturbed_system}-\eqref{B_C}. Then there exists a nondecreasing function
	$M\in\mathcal F$ such that, for any $t\in(0,T)$,
	\begin{equation}\label{phase_high_joint_est}
		\begin{aligned}
			&\|\nabla\phi_t(t)\|_{L^2}^2
			+
			\|\sqrt{\rho}\phi_t(t)\|_{L^2}^2
			+
			\int_0^t
			\left(
			\|\nabla\mu_t\|_{L^2}^2
			+
			\|\sqrt{\rho}\mu_t\|_{L^2}^2
			\right)\,\mathrm d s
			\le
			M\left(
			\mathcal M_0
			+
			\int_0^tM(\Phi(s))\,\mathrm d s
			\right).
		\end{aligned}
	\end{equation}
\end{lemma}

\begin{proof}
	Testing the time derivative of \eqref{perturbed_system}$_3$
	by $\phi_t$, we use the time
	derivative of \eqref{perturbed_system}$_4$ and integrate
	by parts to obtain
	\begin{equation}\label{rho_phit_identity_joint}
		\begin{aligned}
			\frac12\frac{\mathrm d}{\mathrm d t}
			\|\sqrt{\rho}\phi_t\|_{L^2}^2
			+
			\|\sqrt{\rho}\mu_t\|_{L^2}^2
			={}&
			-\int \rho_t\mu\mu_t\,\mathrm d x
			+
			\int \rho_t(\phi^3-\phi)\mu_t\,\mathrm d x  \\
			&+
			\int \rho(3\phi^2-1)\phi_t\mu_t\,\mathrm d x
			-
			\int\rho_t\phi_t^2\,\mathrm d x  \\
			&-
			\int\rho_tu\cdot\nabla\phi\,\phi_t\,\mathrm d x
			-
			\int\rho u_t\cdot\nabla\phi\,\phi_t\,\mathrm d x  \\
			=:& \sum_{i=1}^6 K_i .
		\end{aligned}
	\end{equation}
	We estimate the right-hand side of \eqref{rho_phit_identity_joint}.
	By Lemma \ref{Lemma_F03}, H\"older's inequality, Sobolev embeddings and Young's inequality, for every $\varepsilon>0$, we have
	\[
	\begin{aligned}
		|K_1|+|K_2|
		&\le
		\|\rho_t\|_{L^2}
		\left(
		\|\mu\|_{L^\infty}
		+
		\|\phi^3-\phi\|_{L^\infty}
		\right)
		\|\mu_t\|_{L^2} \\
		&\le
		M(\Phi(t))(\|\nabla\mu_t\|_{L^2}
		+
		\|\sqrt{\rho}\mu_t\|_{L^2})
		\le
		\varepsilon\|\sqrt{\rho}\mu_t\|_{L^2}^2
		+
		\varepsilon\|\nabla\mu_t\|_{L^2}^2
		+
		M(\Phi(t)),\\
		|K_3|
		&\le
		\|3\phi^2-1\|_{L^\infty}
		\|\sqrt{\rho}\phi_t\|_{L^2}
		\|\sqrt{\rho}\mu_t\|_{L^2}
		\le
		\varepsilon\|\sqrt{\rho}\mu_t\|_{L^2}^2
		+
		M(\Phi(t)) ,\\
		|K_4|
		&\le
		\|\rho_t\|_{L^2}
		\|\phi_t\|_{L^4}^2
		\le
		M(\Phi(t)), \\
		|K_5|
		&\le
		\|\rho_t\|_{L^2}
		\|u\|_{L^\infty}
		\|\nabla\phi\|_{L^3}
		\|\phi_t\|_{L^6}
		\le
		M(\Phi(t)),
	\end{aligned}
	\]
	and
	\[
	\begin{aligned}
		|K_6|
		&\le
		\|\sqrt{\rho}u_t\|_{L^2}
		\|\sqrt{\rho}\phi_t\|_{L^2}
		\|\nabla\phi\|_{L^\infty}
		\le
		M(\Phi(t)).
	\end{aligned}
	\]
	Substituting the above estimates into \eqref{rho_phit_identity_joint} and
	choosing $\varepsilon>0$ sufficiently small, we obtain
	\begin{equation}\label{rho_phit_diff_joint}
		\frac{\mathrm d}{\mathrm d t}
		\|\sqrt{\rho}\phi_t\|_{L^2}^2
		+
		c\|\sqrt{\rho}\mu_t\|_{L^2}^2
		\le
		\varepsilon\|\nabla\mu_t\|_{L^2}^2
		+
		M(\Phi(t)).
	\end{equation}
	Here $c>0$ is independent of the approximation parameters.

	For a fixed $\kappa>2$, choose $\varepsilon>0$ sufficiently small such that
	the term $\kappa\varepsilon\|\nabla\mu_t\|_{L^2}^2$
	can be absorbed by the left-hand side of
	\eqref{phit_mut_modified_joint}. Adding $\kappa$ times
	\eqref{rho_phit_diff_joint} to \eqref{phit_mut_modified_joint}, we obtain
	\begin{equation}\label{joint_phase_energy}
		\begin{aligned}
			\frac{\mathrm d}{\mathrm d t}
			\left(
			\|\nabla\phi_t(t)\|_{L^2}^2
			-\mathcal B(t)
			+
			\kappa\|\sqrt{\rho}\phi_t(t)\|_{L^2}^2
			\right)
			+
			c\|\nabla\mu_t\|_{L^2}^2
			+
			c\|\sqrt{\rho}\mu_t\|_{L^2}^2
			\le
			\varepsilon\|\nabla u_t\|_{L^2}^2
			+
			M(\Phi(t)),
		\end{aligned}
	\end{equation}
	where \(c>0\) is independent of the approximation parameters.
	Define
	\[
	\mathcal Y(t):=
	\|\nabla\phi_t(t)\|_{L^2}^2
	-\mathcal B(t)
	+
	\kappa\|\sqrt{\rho}\phi_t(t)\|_{L^2}^2 .
	\]
	Then \eqref{joint_phase_energy} becomes
	\begin{equation}\label{Y_diff_ineq}
		\frac{\mathrm d}{\mathrm d t}\mathcal Y(t)
		+
		c\|\nabla\mu_t\|_{L^2}^2
		+
		c\|\sqrt{\rho}\mu_t\|_{L^2}^2
		\le
		\varepsilon\|\nabla u_t\|_{L^2}^2
		+
		M(\Phi(t)).
	\end{equation}
	We next establish the coercivity of \(\mathcal Y(t)\). Recalling the definition \eqref{def_phase_correction} of
	\(\mathcal B(t)\),
	by H\"older's inequality, Sobolev embedding, and Young's inequality, for any
	$\varepsilon>0$,
	\[
	\begin{aligned}
		\left|
		2\int \nabla\phi_t\cdot\rho u
		\left[\mu-(\phi^3-\phi)\right]\,\mathrm d x
		\right|
		&
		\le
		C\|\nabla\phi_t\|_{L^2}
		\|\rho\|_{L^\infty}
		\|u\|_{L^6}
		\|\mu-(\phi^3-\phi)\|_{L^3}
		\\
		&
		\le
		\varepsilon\|\nabla\phi_t\|_{L^2}^2
		+
		C_\varepsilon
		\|\rho\|_{L^\infty}^2
		\|u\|_{L^6}^2
		\|\mu-(\phi^3-\phi)\|_{L^3}^2 .
	\end{aligned}
	\]
	Since \(q>3\), \(W^{1,q}(\Omega)\hookrightarrow L^\infty(\Omega)\), and
	\(H_0^1(\Omega)\hookrightarrow L^6(\Omega)\). Moreover,
	\[
	\|\mu-(\phi^3-\phi)\|_{L^3}
	\le
	\|\mu\|_{L^3}
	+
	\|\phi^3-\phi\|_{L^3}
	\le
	C\|\mu\|_{H^1}
	+
	C(1+\|\phi\|_{L^\infty}^2)\|\phi\|_{L^3}.
	\]
	Using \(H^2(\Omega)\hookrightarrow L^\infty(\Omega)\) in three dimensions,
	we obtain
	\[
	\|\mu-(\phi^3-\phi)\|_{L^3}
	\le
	M\left(
	\|\mu\|_{H^1}
	+
	\|\phi\|_{H^2}
	\right).
	\]
	Therefore, by Lemmas \ref{Lemma_nabla_u}, \ref{Lemma_rho_W1q} and \ref{Lemma_phi_H2}, one has
	\[
	\begin{aligned}
		\left|
		2\int \nabla\phi_t\cdot\rho u
		\left[\mu-(\phi^3-\phi)\right]\,\mathrm d x
		\right|
		\le
		\varepsilon\|\nabla\phi_t\|_{L^2}^2
		+
		M\left(
		\mathcal M_0+\int_0^tM(\Phi(s))\,\mathrm d s
		\right).
	\end{aligned}
	\]
	Similarly, by H\"older's inequality, Sobolev embedding, and Lemma
	\ref{Lemma_F03}, we have
	\[
	\begin{aligned}
		\left|
		2\int \phi_t\rho u\cdot
		\left[
		\nabla\mu-(3\phi^2-1)\nabla\phi
		\right]\,\mathrm d x
		\right|
		\le
		C
		\|\phi_t\|_{L^6}
		\|\rho\|_{L^\infty}
		\|u\|_{L^3}
		\left\|
		\nabla\mu-(3\phi^2-1)\nabla\phi
		\right\|_{L^2}.
	\end{aligned}
	\]
	Since \(H^1(\Omega)\hookrightarrow L^6(\Omega)\), Lemma
	\ref{Lemma_F03} gives
	\[
	\|\phi_t\|_{L^6}
	\le
	C\|\phi_t\|_{H^1}
	\le
	C\left(
	\|\sqrt{\rho}\phi_t\|_{L^2}
	+
	\|\nabla\phi_t\|_{L^2}
	\right).
	\]
	Moreover,
	\[
	\|u\|_{L^3}
	\le
	C\|u\|_{H^1}
	\le
	M\left(
	\mathcal M_0+\int_0^tM(\Phi(s))\,\mathrm d s
	\right),
	\]
	and
	\[
	\begin{aligned}
		\left\|
		\nabla\mu-(3\phi^2-1)\nabla\phi
		\right\|_{L^2}
		&\le
		\|\nabla\mu\|_{L^2}
		+
		\|3\phi^2-1\|_{L^\infty}
		\|\nabla\phi\|_{L^2}  \\
		&\le
		M\left(
		\mathcal M_0+\int_0^tM(\Phi(s))\,\mathrm d s
		\right).
	\end{aligned}
	\]
	Therefore,
	\[
	\begin{aligned}
		\left|
		2\int \phi_t\rho u\cdot
		\left[
		\nabla\mu-(3\phi^2-1)\nabla\phi
		\right]\,\mathrm d x
		\right|
		\le
		\varepsilon
		\left(
		\|\nabla\phi_t\|_{L^2}^2
		+
		\|\sqrt{\rho}\phi_t\|_{L^2}^2
		\right)
		+
			M\left(
		\mathcal M_0+\int_0^tM(\Phi(s))\,\mathrm d s
		\right).
	\end{aligned}
	\]
	Besides, since
	$
	3\phi^2-1\ge -1,
	$
	we have
	\[
	\int \rho(3\phi^2-1)\phi_t^2\,\mathrm d x
	\ge
	-\|\sqrt{\rho}\phi_t\|_{L^2}^2 .
	\]
	Therefore,
	\[
	\begin{aligned}
		\mathcal Y(t)
		&=
		\|\nabla\phi_t(t)\|_{L^2}^2
		-\mathcal B(t)
		+
		\kappa\|\sqrt{\rho}\phi_t(t)\|_{L^2}^2
		\\
		&\ge
		(1-C\varepsilon)\|\nabla\phi_t(t)\|_{L^2}^2
		+
		(\kappa-1-C\varepsilon)
		\|\sqrt{\rho}\phi_t(t)\|_{L^2}^2
		-
		M\left(
		\mathcal M_0+\int_0^tM(\Phi(s))\,\mathrm d s
		\right).
	\end{aligned}
	\]
	Choosing \(\varepsilon>0\) smaller if necessary, we get
	\begin{equation}\label{Y_lower_bound}
		\mathcal Y(t)
		\ge
		c_0\left(
		\|\nabla\phi_t(t)\|_{L^2}^2
		+
		\|\sqrt{\rho}\phi_t(t)\|_{L^2}^2
		\right)
		-
		M\left(
		\mathcal M_0+\int_0^tM(\Phi(s))\,\mathrm d s
		\right).
	\end{equation}

	Integrating \eqref{Y_diff_ineq} over $(0,t)$ and using \eqref{Y_lower_bound}, we obtain
	\[
	\begin{aligned}
		&\|\nabla\phi_t(t)\|_{L^2}^2
		+
		\|\sqrt{\rho}\phi_t(t)\|_{L^2}^2
		+
		\int_0^t
		\left(
		\|\nabla\mu_t\|_{L^2}^2
		+
		\|\sqrt{\rho}\mu_t\|_{L^2}^2
		\right)\,\mathrm d s  \\
		\le&
		C\mathcal Y(0)
		+
		M\left(\mathcal M_0+\int_0^tM(\Phi(s))\,\mathrm d s\right).
	\end{aligned}
	\]
	The initial compatibility condition gives
	\[
	\Delta\mu_0=\rho_0h_0.
	\]
	Evaluating \eqref{perturbed_system}$_3$ at \(t=0\), we formally obtain
	\[
	\rho_0\phi_t(0)+\rho_0u_0\cdot\nabla\phi_0
	=
	\Delta\mu_0
	=
	\rho_0h_0.
	\]
	Hence,
	\[
	\rho_0
	\left(
	\phi_t(0)-h_0+u_0\cdot\nabla\phi_0
	\right)
	=0.
	\]
	We define
	\[
	\phi_{1,0}:=h_0-u_0\cdot\nabla\phi_0.
	\]
	Since \(\rho_0\ge\delta>0\),
	\(\phi_t(0)=\phi_{1,0}\) a.e.  This identity is used at fixed \(\delta\).
	Its estimate is uniform in \(\delta\).  Indeed,
	\(h_0\in H^1\), \(u_0\in H^2\), \(\phi_0\in H^3\), and
	\(H^2\hookrightarrow L^\infty\) imply $\phi_{1,0}\in H^1(\Omega)$
	and
	\[
	\|\phi_{1,0}\|_{H^1}
	\le
	C\left(
	\|h_0\|_{H^1}
	+
	\|u_0\|_{H^2}\|\phi_0\|_{H^3}
	\right).
	\]
	Therefore,
	\[
	\|\nabla\phi_t(0)\|_{L^2}^2
	+
	\|\sqrt{\rho_0}\phi_t(0)\|_{L^2}^2
	\le
	C\mathcal M_0.
	\]
	Moreover, by the definition of $\mathcal B$ and the initial regularity
	assumptions,
	\[
	|\mathcal B(0)|\le C\mathcal M_0.
	\]
	Therefore,
	\[
	|\mathcal Y(0)|\le C\mathcal M_0.
	\]
	Combining the above estimates, we conclude that
	\[
	\begin{aligned}
		\|\nabla\phi_t(t)\|_{L^2}^2
		+
		\|\sqrt{\rho}\phi_t(t)\|_{L^2}^2
		+
		\int_0^t
		\left(
		\|\nabla\mu_t\|_{L^2}^2
		+
		\|\sqrt{\rho}\mu_t\|_{L^2}^2
		\right)\,\mathrm d s
		\le
		M\left(
		\mathcal M_0+\int_0^tM(\Phi(s))\,\mathrm d s
		\right).
	\end{aligned}
	\]
	This completes the proof.
\end{proof}

\begin{lemma}\label{Lemma_mu_phi_elliptic_closed}
	Under the assumptions of Lemma \ref{Lemma_phase_high_joint}, there exists
	a nondecreasing function $M\in\mathcal F$ such that, for any $t\in(0,T)$,
	\begin{equation}\label{mu_phi_elliptic_closed}
		\|\mu(t)\|_{H^2}^2
		+
		\|\phi(t)\|_{H^3}^2
		\le
		M\left(
		\mathcal M_0+\int_0^tM(\Phi(s))\,\mathrm d s
		\right).
	\end{equation}
\end{lemma}

\begin{proof}
	By \eqref{perturbed_system}$_3$ and Neumann elliptic regularity,
	\[
	\begin{aligned}
		\|\mu\|_{H^2}
		&\le C\left(\|\rho\phi_t+\rho u\cdot\nabla\phi\|_{L^2}
		+\|\mu\|_{L^2}\right)\\
		&\le C\left(
		\|\rho\|_{L^\infty}^{1/2}\|\sqrt{\rho}\phi_t\|_{L^2}
		+\|\rho\|_{L^\infty}\|u\|_{H^1}\|\phi\|_{H^2}
		+\|\mu\|_{L^2}\right).
	\end{aligned}
	\]
	Similarly, using \eqref{perturbed_system}$_4$, the multiplication
	estimate $\|fg\|_{H^1}\le C\|f\|_{W^{1,q}}\|g\|_{H^1}$ for $q>3$,
	and $H^2\hookrightarrow L^\infty$, we have
	\[
	\begin{aligned}
		\|\phi\|_{H^3}
		&\le C\left(
		\|\rho\mu\|_{H^1}+\|\rho(\phi^3-\phi)\|_{H^1}
		+\|R\|_{H^1}+\|\phi\|_{L^2}\right)\\
		&\le C\|\rho\|_{W^{1,q}}
		\big(\|\mu\|_{H^1}
		+(1+\|\phi\|_{H^2}^2)\|\phi\|_{H^1}\big)
		+C\|R\|_{H^1}+C\|\phi\|_{L^2}.
	\end{aligned}
	\]
	All terms on the right are controlled by
	Lemmas~\ref{Lemma_nabla_u}, \ref{Lemma_rho_W1q},
	\ref{Lemma_mu_phi_t}, \ref{Lemma_phi_H2}, and
	\ref{Lemma_phase_high_joint}.
	Together with \eqref{fixed_residual_H1}, after enlarging $M$ if necessary, one obtains
	\[
	\|\mu(t)\|_{H^2}^2+\|\phi(t)\|_{H^3}^2
	\le M\left(
	\mathcal M_0+\int_0^t M(\Phi(s))\,\mathrm ds
	\right).
	\]
	The proof is complete.
\end{proof}

 \noindent
 \textbf{Proof of Proposition \ref{apriori-prop}.}
  Collecting all the estimates above, we obtain
 \[
 \Phi(t)+\int_0^t\mathcal D(s)\,\mathrm d s
 \le
 M\left(
 \mathcal M_0+\int_0^tM(\Phi(s))\,\mathrm d s
 \right),
 \]
 where $\mathcal D(t)$ is given by \eqref{def_D}. This proves \eqref{apriori-est} and completes the proof of Proposition \ref{apriori-prop}.

\section{Compactness, convergence, and well-posedness with vacuum}
\label{sec:vacuum-limit}

For each $\delta>0$, Theorem~\ref{positive-density-thm} and
local uniqueness yield a maximal positive-density strong solution
$(\rho^\delta,u^\delta,\phi^\delta,\mu^\delta)$ on
$[0,T_\delta^{\max})$.
We establish a common lifespan, pass to the limit as
$\delta\to0$, and prove uniqueness to complete the proof of
Theorem~\ref{main-thm}. Hereafter, $\mathcal M_0^\delta$ denotes the initial size of the $\delta$-approximation, while $\Phi^\delta$ and
$\mathcal D^\delta$ are the quantities $\Phi$ and $\mathcal D$ evaluated at $(\rho^\delta,u^\delta,\phi^\delta,\mu^\delta)$ with residual $R^\delta$; see \eqref{def_Phi}, \eqref{def_initial_M} and \eqref{def_D}.

\subsection{A common lifespan}

The initial-data estimates in
Section~\ref{sec:perturbed-problem} imply
\begin{equation}\label{uniform-initial-size}
	B_0:=\sup_{0<\delta<1}\mathcal M_0^\delta<\infty.
\end{equation}
Proposition~\ref{apriori-prop} and a standard continuity
argument provide $T_0\in(0,1]$, independent of $\delta$, such that
\[
\sup_{0\le s\le t}\Phi^\delta(s)
+\int_0^t\mathcal D^\delta(s)\,\mathrm ds
\le C(B_0),
\qquad
0\le t<\min\{T_0,T_\delta^{\max}\}.
\]
Since these estimates control
$\int_0^t\|\nabla u^\delta\|_{L^\infty}\,\mathrm ds$,
the characteristic formula gives
\begin{equation}\label{rho-delta-stays-positive}
	\inf_{\Omega\times[0,t]}\rho^\delta
	\ge \delta^{1/\gamma}e^{-C(B_0)}>0.
\end{equation}
For each fixed $\delta$, this lower bound and the preceding
estimates ensure the hypotheses of
Proposition~\ref{positive-density-continuation}.
Consequently, $T_\delta^{\max}>T_0$ for every $0<\delta<1$,
and the approximate solutions satisfy the uniform bound
\begin{equation}\label{uniform-Phi-D}
	\sup_{0<\delta<1}
	\left\{
	\sup_{0\le t\le T_0}\Phi^\delta(t)
	+\int_0^{T_0}\mathcal D^\delta(s)\,\mathrm ds
	\right\}
	\le C(B_0).
\end{equation}
This bound allows us to pass to the vacuum limit
on the common interval $[0,T_0]$.
\subsection{Compactness}
By \eqref{uniform-Phi-D}, the associated elliptic estimates,
and the Aubin--Lions--Simon compactness theorem \cite{Simon87},
there exists a subsequence, still indexed by $\delta$, such that
\begin{alignat*}{2}
	\rho^\delta &\rightharpoonup^* \rho
	\quad\text{in }L^\infty(0,T_0;W^{1,q}),
	&\qquad
	\rho_t^\delta &\rightharpoonup^* \rho_t
	\quad\text{in }L^\infty(0,T_0;L^2),
	\\
	u^\delta &\rightharpoonup^* u
	\quad\text{in }L^\infty(0,T_0;H^2\cap H_0^1),
	&
	u_t^\delta &\rightharpoonup u_t
	\quad\text{in }L^2(0,T_0;H_0^1),
	\\
	\phi^\delta &\rightharpoonup^* \phi
	\quad\text{in }L^\infty(0,T_0;H^3),
	&
	\phi_t^\delta &\rightharpoonup^* \phi_t
	\quad\text{in }L^\infty(0,T_0;H^1),
	\\
	\mu^\delta &\rightharpoonup^* \mu
	\quad\text{in }L^\infty(0,T_0;H^2),
	&
	\mu_t^\delta &\rightharpoonup \mu_t
	\quad\text{in }L^2(0,T_0;H^1).
\end{alignat*}
Furthermore, the Lam\'e elliptic estimate gives
\[
\|u^\delta(t)\|_{W^{2,q}}
\le M(\Phi^\delta(t))
\bigl(1+\|\nabla u_t^\delta(t)\|_{L^2}\bigr).
\]
Together with \eqref{uniform-Phi-D}, this yields
\[
\sup_{0<\delta<1}
\int_0^{T_0}\|u^\delta(t)\|_{W^{2,q}}^2\,\mathrm dt
\le C(B_0).
\]
Hence, after further extraction if necessary,
\[
u^\delta\rightharpoonup u
\quad\text{in }L^2(0,T_0;W^{2,q}).
\]

Moreover, for every $0<\alpha<1-3/q$,
\begin{alignat}{2}
	\rho^\delta &\to\rho
	\quad\text{in }C([0,T_0];C^\alpha(\overline\Omega)),
	&\qquad
	u^\delta &\to u
	\quad\text{in }C([0,T_0];H^1),
	\label{vacuum-strong-rho-u}\\
	\phi^\delta &\to\phi
	\quad\text{in }C([0,T_0];H^2),
	&
	\mu^\delta &\to\mu
	\quad\text{in }C([0,T_0];H^1).
	\label{vacuum-strong-phi-mu}
\end{alignat}
The uniform convergence of $\sqrt{\rho^\delta}$ identifies
the weighted weak-star limits as
\begin{align}
	\sqrt{\rho^\delta}\,u_t^\delta
	&\rightharpoonup^*\sqrt\rho\,u_t
	&&\text{in }L^\infty(0,T_0;L^2),
	\label{vacuum-weighted-ut-limit}\\
	\sqrt{\rho^\delta}\,\phi_t^\delta
	&\rightharpoonup^*\sqrt\rho\,\phi_t
	&&\text{in }L^\infty(0,T_0;L^2).
	\label{vacuum-weighted-phit-limit}
\end{align}
Finally, the strong convergence of $\phi^\delta$ and
Sobolev multiplication imply
\begin{align}
	\Psi'(\phi^\delta)&\to\Psi'(\phi)
	&&\text{in }C([0,T_0];H^1),
	\label{vacuum-strong-nonlinearity}\\
	(\Delta\phi^\delta)\nabla\phi^\delta
	&\to(\Delta\phi)\nabla\phi
	&&\text{in }C([0,T_0];L^{3/2}).
	\label{vacuum-strong-capillary}
\end{align}
By the uniform bounds and the convergences established above,
we pass to the limit in the approximate equations in the sense
of distributions. Since $R^\delta\to0$ in $H^1(\Omega)$,
the limit $(\rho,u,\phi,\mu)$ satisfies \eqref{NSCH}.
By the established regularity, these equations hold almost
everywhere in $\Omega\times(0,T_0)$.

The strong time convergences, together with the convergence
of the approximate initial data, yield the weighted initial
conditions \eqref{NSCH_I_C}.
The boundary conditions are preserved under the weak Sobolev
convergences, and $\rho\ge0$ follows from the uniform
convergence of $\rho^\delta$.
Finally, combining \eqref{uniform-Phi-D}, weak lower semicontinuity,
and the weighted convergences
\eqref{vacuum-weighted-ut-limit}-\eqref{vacuum-weighted-phit-limit}
yields the regularity and estimates stated in
Theorem~\ref{main-thm}.
This completes the existence proof of Theorem~\ref{main-thm}.

\subsection{Uniqueness}
\label{sec:uniqueness}
We prove uniqueness in the class of strong solutions from
Theorem~\ref{main-thm}. The same argument also applies to two
positive-density solutions with the same time-independent residual $R$:
this residual cancels when the chemical relations are subtracted.
Let $(\rho_i,u_i,\phi_i,\mu_i)$, $i=1,2$, be two strong solutions of the system
\eqref{NSCH}--\eqref{NSCH_B_C} on $[0,T]$ with the same initial data. Set
\[
\rho:=\rho_1-\rho_2,\qquad
u:=u_1-u_2,\qquad
\phi:=\phi_1-\phi_2,\qquad
\mu:=\mu_1-\mu_2.
\]
The differences of these solutions satisfy
\begin{equation}\label{diff_NSCH}
	\left\{
	\begin{aligned}
		&\rho_t+\operatorname{div}(\rho u_1)+\operatorname{div}(\rho_2u)=0,
		\\
		&\rho_1\bigl(u_t+u_1\cdot\nabla u\bigr)
		-\mathcal L u
		+\nabla(\rho_1^\gamma-\rho_2^\gamma)
		\\
		&\qquad
		=
		-\rho\bigl(u_{2t}+u_2\cdot\nabla u_2\bigr)
		-\rho_1u\cdot\nabla u_2
		-(\Delta\phi)\nabla\phi_1
		-(\Delta\phi_2)\nabla\phi,
		\\
		&\rho_1\bigl(\phi_t+u_1\cdot\nabla\phi\bigr)
		-\Delta\mu
		=
		-\rho\bigl(\phi_{2t}+u_2\cdot\nabla\phi_2\bigr)
		-\rho_1u\cdot\nabla\phi_2,
		\\
		&\rho_1\mu
		=
		-\Delta\phi
		+\rho_1\bigl[\Psi'(\phi_1)-\Psi'(\phi_2)\bigr]
		+\rho\bigl(\Psi'(\phi_2)-\mu_2\bigr),
	\end{aligned}
	\right.
\end{equation}
with the boundary conditions
\[
u=0,\qquad \partial_n\phi=\partial_n\mu=0,
\qquad \text{on } \partial\Omega\times(0,T),
\]
and the initial conditions
\[
\rho(0)=0,\qquad \rho_0u(0)=0,\qquad \rho_0\phi(0)=0,
\qquad \text{in }\Omega.
\]

To begin with, we define
\begin{equation}\label{uniqueness_energy}
\mathcal H(t):=
\|\rho(t)\|_{L^2}^2
+\|\sqrt{\rho_1}u(t)\|_{L^2}^2
+\|\sqrt{\rho_1}\phi(t)\|_{L^2}^2
+\|\nabla\phi(t)\|_{L^2}^2
\end{equation}
and
\begin{equation}\label{uniqueness_coefficient}
\mathcal G(t):=
1+\|\nabla u_1(t)\|_{L^\infty}
+\|\nabla u_2(t)\|_{L^\infty}
+\|\mu_{2t}(t)\|_{H^1}^2
+\|u_{2t}(t)\|_{H^1}^2.
\end{equation}
The regularity in Theorem~\ref{main-thm} implies
$\mathcal G\in L^1(0,T)$. We will prove that
\begin{equation}\label{uniqueness_target}
\mathcal H(t)+c\int_0^t
\bigl(\|\nabla u(s)\|_{L^2}^2
+\|\nabla\mu(s)\|_{L^2}^2\bigr)\,\mathrm ds
\le C\int_0^t\mathcal G(s)\mathcal H(s)\,\mathrm ds
\end{equation}
for some $c>0$. Gronwall's inequality will then give
$\mathcal H=0$, and the dissipation terms will determine
the velocity and chemical potential also in the vacuum region.

Throughout this proof, $C$ may depend on the uniform bounds of the two
solutions in the regularity class of Theorem~\ref{main-thm}, the domain,
and the positive total mass, but is independent of time. In particular,
Sobolev embedding gives
\[
 u_i,\ \mu_i,\ \phi_i,\ \nabla\phi_i\in L^\infty(0,T;L^\infty),
 \qquad \phi_{it}\in L^\infty(0,T;L^6),\qquad i=1,2.
\]
The constants in the weighted Poincar\'e inequalities below depend on
the positive conserved total mass and the uniform density bound,
but not on a positive lower bound for the density.

\medskip
\noindent
\textbf{Step 1. Density difference.}
Multiplying \eqref{diff_NSCH}$_1$ by $\rho$ and integrating by parts over
$\Omega$, we obtain
\[
\frac12\frac{\mathrm d}{\mathrm d t}\|\rho\|_{L^2}^2
+
\frac12\int \rho^2\operatorname{div}u_1\,\mathrm d x
+
\int \rho\,\operatorname{div}(\rho_2u)\,\mathrm d x
=0.
\]
Since $\rho_i \in L^\infty(0,T;W^{1,q}\cap H^1)$, it follows from H\"older's and Young's inequalities that, for any $\varepsilon>0$,
\begin{equation}\label{density_final}
\begin{aligned}
	\frac{\mathrm d}{\mathrm d t}\|\rho\|_{L^2}^2
	&\le
	C\|\nabla u_1\|_{L^\infty}\|\rho\|_{L^2}^2
	+
	2\left|
	\int \rho\,\operatorname{div}(\rho_2u)\,\mathrm d x
	\right| \\
	&\le
	C\|\nabla u_1\|_{L^\infty}\|\rho\|_{L^2}^2
	+
	C\|\rho\|_{L^2}
	\|u\|_{L^{\frac{2q}{q-2}}}
	\|\nabla\rho_2\|_{L^q}
	+C\|\rho_2\|_{L^\infty}\|\rho\|_{L^2}\|\nabla u\|_{L^2}
	\\
	&\le
	\varepsilon\|\nabla u\|_{L^2}^2
	+
	C_\varepsilon
	(
	\|\nabla u_1\|_{L^\infty}
	+1
	)
	\|\rho\|_{L^2}^2 ,
\end{aligned}
\end{equation}
where we used $q>3$ and therefore
$H_0^1(\Omega)\hookrightarrow L^{\frac{2q}{q-2}}(\Omega)$.

\medskip
\noindent
\textbf{Step 2. Phase difference.}
Testing \eqref{diff_NSCH}$_4$ by $\mu$ and using
$\partial_n\phi=0$, we have
\begin{equation}\label{weighted_mu_control}
	\begin{aligned}
		\|\sqrt{\rho_1}\mu\|_{L^2}^2
		&=
		\int \nabla\phi\cdot\nabla\mu\,\mathrm d x
		+
		\int
		\rho_1\bigl[\Psi'(\phi_1)-\Psi'(\phi_2)\bigr]\mu\,\mathrm d x
		+
		\int
		\rho\bigl(\Psi'(\phi_2)-\mu_2\bigr)\mu\,\mathrm d x
		\\
		&\le
		\|\nabla\phi\|_{L^2}\|\nabla\mu\|_{L^2}
		+
		C\|\sqrt{\rho_1}\phi\|_{L^2}
		\|\sqrt{\rho_1}\mu\|_{L^2}
		+
		C\|\rho\|_{L^2}
		\bigl(
		\|\sqrt{\rho_1}\mu\|_{L^2}
		+
		\|\nabla\mu\|_{L^2}
		\bigr)
		\\
		&\le
		\frac12\|\sqrt{\rho_1}\mu\|_{L^2}^2
		+
		C\Bigl(
		\|\nabla\phi\|_{L^2}^2
		+
		\|\sqrt{\rho_1}\phi\|_{L^2}^2
		+
		\|\nabla\mu\|_{L^2}^2
		+
		\|\rho\|_{L^2}^2
		\Bigr).
	\end{aligned}
\end{equation}
Consequently, after absorbing the first term on the right-hand side of
\eqref{weighted_mu_control}, it follows from Lemma \ref{Lemma_F03} that
\[
	\|\mu\|_{L^2}^2
    +
    \|\mu\|_{L^3}^2
	\le
	C\Bigl(
	\|\nabla\mu\|_{L^2}^2
	+
	\|\nabla\phi\|_{L^2}^2
	+
	\|\sqrt{\rho_1}\phi\|_{L^2}^2
	+
	\|\rho\|_{L^2}^2
	\Bigr).
\]
Therefore,
\begin{equation}\label{mu_control}
	\begin{aligned}
		\|\sqrt{\rho_1}\mu\|_{L^2}^2
		+
		\|\mu\|_{L^2}^2
		+
		\|\mu\|_{L^3}^2
		\le
		C\Bigl(
		\|\nabla\mu\|_{L^2}^2
		+
		\|\nabla\phi\|_{L^2}^2
		+
		\|\sqrt{\rho_1}\phi\|_{L^2}^2
		+
		\|\rho\|_{L^2}^2
		\Bigr).
	\end{aligned}
\end{equation}

Testing \eqref{diff_NSCH}$_3$ by $\mu$ and
\eqref{diff_NSCH}$_4$ by $\phi_t$, we combine the resulting
identities and add the time derivative of
$\frac12\|\sqrt{\rho_1}\phi\|_{L^2}^2$ to obtain
\begin{equation}\label{phase_basic_identity}
	\begin{aligned}
		&\frac12\frac{\mathrm d}{\mathrm d t}
		\Bigl(
		\|\sqrt{\rho_1}\phi\|_{L^2}^2
		+
		\|\nabla\phi\|_{L^2}^2
		\Bigr)
		+
		\|\nabla\mu\|_{L^2}^2
		\\
		=&
		\frac12\int \rho_{1t}\phi^2\,\mathrm d x
		-\int
		\rho_1
		\Bigl(
		\Psi'(\phi_1)-\Psi'(\phi_2)-\phi
		\Bigr)\phi_t\,\mathrm d x
		-\int \rho_1u_1\cdot\nabla\phi\,\mu\,\mathrm d x\\
		&-\int \rho_1u\cdot\nabla\phi_2\,\mu\,\mathrm d x
		-\int \rho\phi_{2t}\mu\,\mathrm d x
		-\int \rho u_2\cdot\nabla\phi_2\,\mu\,\mathrm d x
		+
		\int
		\rho\bigl(\mu_2-\Psi'(\phi_2)\bigr)\phi_t\,\mathrm d x \\
		:=& \sum_{i=1}^7 K_i.
	\end{aligned}
\end{equation}
We estimate the terms in \eqref{phase_basic_identity} separately.
As for $K_1$, since $\phi_i\in L^\infty(0,T;H^3)$ and $\rho_{it} \in L^\infty(0,T;L^2)$, $i=1,2$, we have
\begin{equation}\label{k1_estimate}
	K_1
	\le
	C\|\rho_{1t}\|_{L^2}\|\phi\|_{L^4}^2
	\le
	C\Bigl(
	\|\sqrt{\rho_1}\phi\|_{L^2}^2
	+
	\|\nabla\phi\|_{L^2}^2
	\Bigr)
	\le
	C \mathcal H(t) .
\end{equation}

As for $K_2$, we first notice that, by the mean value theorem,
\[
\Psi'(\phi_1)-\Psi'(\phi_2)
=
\Psi''(\theta\phi_1+(1-\theta)\phi_2)(\phi_1-\phi_2)
\]
for some $\theta=\theta(x,t)\in(0,1)$. Then,  the bounds
$\phi_i\in L^\infty(0,T;H^3)$, $i=1,2$, together with
Sobolev embedding and the smoothness of $\Psi$, imply
\begin{equation}\label{psi_hidden_estimates_1}
	\left|\Psi'(\phi_1)-\Psi'(\phi_2)-\phi\right|
	\le C|\phi|
\end{equation}
and
\begin{equation}\label{psi_hidden_estimates_2}
	\left\|\nabla\bigl(\Psi'(\phi_1)-\Psi'(\phi_2)-\phi\bigr)
	\right\|_{L^2}
	\le C\|\phi\|_{H^1}.
\end{equation}
Using \eqref{diff_NSCH}$_3$ and integrating by parts, we can rewrite $K_2$ as
\[
\begin{aligned}
	K_2
	={}&\int
	\nabla\bigl(\Psi'(\phi_1)-\Psi'(\phi_2)-\phi\bigr)
	\cdot\nabla\mu\,\mathrm dx\\
	&+\int
	\bigl(\Psi'(\phi_1)-\Psi'(\phi_2)-\phi\bigr)
	\Bigl(
	\rho_1u_1\cdot\nabla\phi
	+\rho(\phi_{2t}+u_2\cdot\nabla\phi_2)
	+\rho_1u\cdot\nabla\phi_2
	\Bigr)\,\mathrm dx.
\end{aligned}
\]
Applying  Lemma~\ref{Lemma_F03}, H\"older's inequality and Young's inequality, we obtain
\begin{equation}\label{potential_phi_t_estimate}
	\begin{aligned}
		K_2
		&\le
		\left\|
		\nabla\Bigl(
		\Psi'(\phi_1)-\Psi'(\phi_2)-\phi
		\Bigr)
		\right\|_{L^2}
		\|\nabla\mu\|_{L^2}
		+
		C\|\rho_1\|_{L^\infty}\|u_1\|_{L^\infty}
		\|\phi\|_{L^2}\|\nabla\phi\|_{L^2}
		\\
		&\quad
		+
		C\|\rho\|_{L^2}
		\|\phi\|_{L^3}
		\bigl(
		\|\phi_{2t}\|_{L^6}
		+\|u_2\|_{L^\infty}\|\nabla\phi_2\|_{L^6}
		\bigr)
		+
		C\|\nabla\phi_2\|_{L^\infty}
		\|\sqrt{\rho_1}u\|_{L^2}
		\|\sqrt{\rho_1}\phi\|_{L^2}
		\\
		&\le
		C\bigl(
		\|\nabla\phi\|_{L^2}
		+\|\phi\|_{L^2}
		\bigr)
		\|\nabla\mu\|_{L^2}
		+
		C
		\bigl(
		\|\phi\|_{H^1}^2
		+\|\rho\|_{L^2}\|\phi\|_{H^1}
		+\|\sqrt{\rho_1}u\|_{L^2}\|\sqrt{\rho_1}\phi\|_{L^2}
		\bigr)
		\\
		&\le
		\varepsilon\|\nabla\mu\|_{L^2}^2
		+
		C_\varepsilon
		\mathcal H(t),
	\end{aligned}
\end{equation}

Next, it follows from \eqref{mu_control} that
\begin{equation}\label{mu_terms_estimate}
	\begin{aligned}
		\sum_{i=3}^6 K_i
		&\le
		C\|\nabla\phi\|_{L^2}\|\mu\|_{L^2}
		+
		C\|\sqrt{\rho_1}u\|_{L^2}\|\sqrt{\rho_1}\mu\|_{L^2}
		\\
		&\quad
		+
		C\|\rho\|_{L^2}\|\phi_{2t}\|_{L^6}\|\mu\|_{L^3}
		+
		C\|\rho\|_{L^2}\|u_2\|_{L^\infty}
		\|\nabla\phi_2\|_{L^6}\|\mu\|_{L^3}
		\\
		&\le
		\varepsilon\|\nabla\mu\|_{L^2}^2
		+
		C_\varepsilon
		\mathcal H(t).
	\end{aligned}
\end{equation}

For the last term in \eqref{phase_basic_identity}, set
$\Theta_2:=\mu_2-\Psi'(\phi_2)$. The regularity of the two solutions
and Sobolev embedding imply
\[
 \|\Theta_2\|_{L^\infty}+\|\nabla\Theta_2\|_{L^3}\le C,
 \qquad
 \|\Theta_{2t}\|_{L^3}
 \le C\bigl(1+\|\mu_{2t}\|_{H^1}\bigr).
\]
Consequently, H\"older's inequality and $u\in H_0^1$ give
\begin{align*}
 \left|\int_\Omega\rho\Theta_{2t}\phi\,\mathrm dx\right|
 &\le C\bigl(1+\|\mu_{2t}\|_{H^1}\bigr)
       \|\rho\|_{L^2}\|\phi\|_{H^1},\\
 \left|\int_\Omega(\rho u_1+\rho_2u)\cdot
       \nabla(\Theta_2\phi)\,\mathrm dx\right|
 &\le C\bigl(\|\rho\|_{L^2}+\|\nabla u\|_{L^2}\bigr)
       \|\phi\|_{H^1}.
\end{align*}
Using the density difference equation \eqref{diff_NSCH}$_1$,
Young's inequality, and the weighted Poincar\'e inequality, we obtain
\begin{equation}\label{cross_phi_t}
	\begin{aligned}
		K_7
		&=
		\frac{\mathrm d}{\mathrm d t}
		\int
		\rho\Theta_2\phi\,\mathrm d x
		-\int
		\rho\Theta_{2t}\phi\,\mathrm d x
		-\int
		(\rho u_1+\rho_2u)\cdot
		\nabla\Bigl[
		\Theta_2\phi
		\Bigr]\,\mathrm d x
		\\
		&\le
		\frac{\mathrm d}{\mathrm d t}
		\int
		\rho\Theta_2\phi\,\mathrm d x
		+
		\varepsilon\|\nabla u\|_{L^2}^2
		+
		C_\varepsilon\mathcal G(t)
		\mathcal H(t).
	\end{aligned}
\end{equation}

Combining \eqref{k1_estimate}-\eqref{cross_phi_t} and taking
$\varepsilon>0$ sufficiently small, we arrive at
\begin{equation}\label{phase_final_differential}
	\begin{aligned}
		&\frac{\mathrm d}{\mathrm d t}
		\Biggl[
		\|\sqrt{\rho_1}\phi\|_{L^2}^2
		+
		\|\nabla\phi\|_{L^2}^2
		-
		2\int
		\rho\bigl(\mu_2-\Psi'(\phi_2)\bigr)\phi\,\mathrm d x
		\Biggr]
		+
		\|\nabla\mu\|_{L^2}^2
		\\
		&\le
		\varepsilon\|\nabla u\|_{L^2}^2
		+
		C_\varepsilon\mathcal G(t)
		\mathcal H(t).
	\end{aligned}
\end{equation}

We first verify that the modified phase energy vanishes at
$t=0$. The common initial data give $\rho(0)=0$ and
$\rho_0\phi(0)=0$.
By the time-continuity in Theorem~\ref{main-thm}, the difference
of the chemical relations holds at $t=0$ in $L^2(\Omega)$:
\[
\rho_0\mu(0)
=-\Delta\phi(0)
+\rho_0\bigl[\Psi'(\phi_1(0))-\Psi'(\phi_2(0))\bigr].
\]
Testing it by $\phi(0)$ and using the homogeneous Neumann
condition and $\rho_0\phi(0)=0$, we obtain
$\|\nabla\phi(0)\|_{L^2}^2=0$.
Since $\Omega$ is connected and $\int\rho_0\,\mathrm dx>0$,
it follows that $\phi(0)=0$. Consequently,
\begin{equation}\label{phase_initial_zero}
	\Biggl[
	\|\sqrt{\rho_1}\phi\|_{L^2}^2
	+\|\nabla\phi\|_{L^2}^2
	-2\int\rho\bigl(\mu_2-\Psi'(\phi_2)\bigr)\phi\,\mathrm dx
	\Biggr]_{t=0}=0.
\end{equation}

By Sobolev embedding and the weighted Poincar\'e inequality,
\[
	\left|2\int
	\rho\bigl(\mu_2-\Psi'(\phi_2)\bigr)\phi\,\mathrm dx\right|
	\le \frac12\Bigl(
	\|\sqrt{\rho_1}\phi\|_{L^2}^2
	+\|\nabla\phi\|_{L^2}^2
	\Bigr)+C\|\rho\|_{L^2}^2.
\]
Integrating \eqref{phase_final_differential} over $(0,t)$,
using \eqref{phase_initial_zero}, and absorbing the cross
term, we obtain, after adjusting $\varepsilon$ and
$C_\varepsilon$,
\begin{equation}\label{phase_final_integral}
	\begin{aligned}
		&\|\sqrt{\rho_1}\phi(t)\|_{L^2}^2
		+\|\nabla\phi(t)\|_{L^2}^2
		+\int_0^t\|\nabla\mu(s)\|_{L^2}^2\,\mathrm ds\\
		\le& C\|\rho(t)\|_{L^2}^2
		+\varepsilon\int_0^t\|\nabla u(s)\|_{L^2}^2\,\mathrm ds
		+C_\varepsilon\int_0^t
		\mathcal G(s)
		\mathcal H(s)\,\mathrm ds.
	\end{aligned}
\end{equation}
\medskip
\noindent
\textbf{Step 3. Velocity difference.}
Taking the inner product of \eqref{diff_NSCH}$_2$ with $u$ and integrating
over $\Omega$, we first note that
\[
\begin{aligned}
	\int \rho_1(u_t+u_1\cdot\nabla u)\cdot u\,\mathrm d x
	=
	\frac12\frac{\mathrm d}{\mathrm d t}\|\sqrt{\rho_1}u\|_{L^2}^2.
\end{aligned}
\]
Moreover, by the coercivity of $\mathcal L$ and $u|_{\partial\Omega}=0$,
\[
-\int \mathcal L u\cdot u\,\mathrm d x
\ge c_0\|\nabla u\|_{L^2}^2.
\]
Hence
\begin{equation}\label{vel_basic}
	\begin{aligned}
		\frac12\frac{\mathrm d}{\mathrm d t}\|\sqrt{\rho_1}u\|_{L^2}^2
		+c_0\|\nabla u\|_{L^2}^2
		\le&
		\left|\int (\rho_1^\gamma-\rho_2^\gamma)\operatorname{div}u\,\mathrm d x\right|
		+
		\left|\int \rho(u_{2t}+u_2\cdot\nabla u_2)\cdot u\,\mathrm d x\right|
		\\
		&
		+
		\left|\int \rho_1(u\cdot\nabla u_2)\cdot u\,\mathrm d x\right|
		+
		\left|\int (\Delta\phi)\nabla\phi_1\cdot u\,\mathrm d x\right|
		\\
		&+
		\left|\int (\Delta\phi_2)\nabla\phi\cdot u\,\mathrm d x\right|.
	\end{aligned}
\end{equation}

Since $\rho_i\in L^\infty(0,T;L^\infty)$ and $s\mapsto s^\gamma$ is
Lipschitz continuous on bounded intervals, we have
\[
|\rho_1^\gamma-\rho_2^\gamma|\le C|\rho|.
\]
Therefore, for any $\varepsilon>0$,
\begin{equation}\label{pressure_est}
		\left|
		\int
		(\rho_1^\gamma-\rho_2^\gamma)\operatorname{div}u\,\mathrm d x
		\right|
		\le
		C\|\rho\|_{L^2}\|\nabla u\|_{L^2}
		\le
		\varepsilon\|\nabla u\|_{L^2}^2
		+
		C_\varepsilon\|\rho\|_{L^2}^2 .
\end{equation}
Using $H_0^1(\Omega)\hookrightarrow L^6(\Omega)$ and
$H^1(\Omega)\hookrightarrow L^3(\Omega)$, we get
\begin{equation}\label{u2_terms_est}
	\begin{aligned}
		\left|
		\int
		\rho(u_{2t}+u_2\cdot\nabla u_2)\cdot u\,\mathrm d x
		\right|
		&\le
		\|\rho\|_{L^2}
		\bigl(
		\|u_{2t}\|_{L^3}
		+
		\|u_2\|_{L^\infty}\|\nabla u_2\|_{L^3}
		\bigr)
		\|u\|_{L^6}
		\\
		&\le
		C
		\|\rho\|_{L^2}
		\bigl(
		\|u_{2t}\|_{H^1}
		+
		\|u_2\|_{L^\infty}\|\nabla u_2\|_{L^3}
		\bigr)
		\|\nabla u\|_{L^2}
		\\
		&\le
		\varepsilon\|\nabla u\|_{L^2}^2
		+
		C_\varepsilon
		\bigl(
		\|u_{2t}\|_{H^1}^2
		+
		1
		\bigr)
		\|\rho\|_{L^2}^2 .
	\end{aligned}
\end{equation}
Moreover,
\begin{equation}\label{u_grad_u2_est}
	\begin{aligned}
		\left|
		\int
		\rho_1(u\cdot\nabla u_2)\cdot u\,\mathrm d x
		\right|
		&\le
		\|\nabla u_2\|_{L^\infty}
		\|\sqrt{\rho_1}u\|_{L^2}^2 .
	\end{aligned}
\end{equation}
For the capillary terms, by integration by parts and $u=0$ on
$\partial\Omega$, we have
\begin{equation}\label{capillary_1_est}
	\begin{aligned}
		\left|
		\int
		(\Delta\phi)\nabla\phi_1\cdot u\,\mathrm d x
		\right|
		&=
		\left|
		-\int
		\nabla\phi\cdot\nabla(\nabla\phi_1\cdot u)\,\mathrm d x
		\right|
		\\
		&\le
		C\|\nabla\phi\|_{L^2}
		\bigl(
		\|\nabla^2\phi_1\|_{L^3}\|u\|_{L^6}
		+
		\|\nabla\phi_1\|_{L^\infty}\|\nabla u\|_{L^2}
		\bigr)
		\\
		&\le
		\varepsilon\|\nabla u\|_{L^2}^2
		+
		C_\varepsilon
		\|\nabla\phi\|_{L^2}^2 .
	\end{aligned}
\end{equation}
Similarly,
\begin{equation}\label{capillary_2_est}
	\begin{aligned}
		\left|
		\int
		(\Delta\phi_2)\nabla\phi\cdot u\,\mathrm d x
		\right|
		\le
		\|\Delta\phi_2\|_{L^3}
		\|\nabla\phi\|_{L^2}
		\|u\|_{L^6}
		\le
		\varepsilon\|\nabla u\|_{L^2}^2
		+
		C_\varepsilon
		\|\nabla\phi\|_{L^2}^2 .
	\end{aligned}
\end{equation}
Hence,
we infer from \eqref{vel_basic}--\eqref{capillary_2_est} that
\begin{equation}\label{velocity_final}
\frac{\mathrm d}{\mathrm dt}\|\sqrt{\rho_1}u\|_{L^2}^2
+c\|\nabla u\|_{L^2}^2
\le C\mathcal G(t)\mathcal H(t).
\end{equation}

\medskip
\noindent
\textbf{Step 4. Gronwall argument.}
Integrating \eqref{density_final} and \eqref{velocity_final}
and using the common initial density and momentum, we obtain
\begin{equation}\label{rho_final_integral}
\|\rho(t)\|_{L^2}^2
\le\varepsilon\int_0^t\|\nabla u(s)\|_{L^2}^2\,\mathrm ds
+C_\varepsilon\int_0^t\mathcal G(s)\mathcal H(s)\,\mathrm ds
\end{equation}
and
\begin{equation}\label{velocity_final_integral}
\|\sqrt{\rho_1}u(t)\|_{L^2}^2
+c\int_0^t\|\nabla u(s)\|_{L^2}^2\,\mathrm ds
\le C\int_0^t\mathcal G(s)\mathcal H(s)\,\mathrm ds.
\end{equation}
Substituting \eqref{rho_final_integral} into
\eqref{phase_final_integral}, adding the density and velocity
estimates, and choosing $\varepsilon$ sufficiently small,
we absorb the terms involving
$\int_0^t\|\nabla u\|_{L^2}^2\,\mathrm ds$ and obtain
\begin{equation}\label{gronwall_pre}
\mathcal H(t)+c\int_0^t
\bigl(\|\nabla u(s)\|_{L^2}^2
+\|\nabla\mu(s)\|_{L^2}^2\bigr)\,\mathrm ds
\le C\int_0^t\mathcal G(s)\mathcal H(s)\,\mathrm ds.
\end{equation}
This is \eqref{uniqueness_target}. Since $\mathcal G\in L^1(0,T)$,
Gronwall's inequality gives $\mathcal H(t)=0$ for $0\le t\le T$.
Thus $\rho=0$, and the weighted Poincar\'e inequality gives
$\phi=0$. By \eqref{gronwall_pre}, both dissipation integrals
vanish. The homogeneous Dirichlet condition then implies $u=0$,
and \eqref{mu_control} yields $\mu=0$. Therefore,
\[
(\rho_1,u_1,\phi_1,\mu_1)=(\rho_2,u_2,\phi_2,\mu_2)
\quad\text{on }\Omega\times(0,T),
\]
which proves uniqueness and completes the proof of
Theorem~\ref{main-thm}.

\section*{Statements and Declarations}

%\subsection*{Funding}
%Yinghua Li's work is supported by the National Natural Science Foundation of China (Grant No. 12371205) and the Natural Science Foundation of Guangdong Province (Grant Nos. 2025A1515012026 and 2025A1515040001).

\subsection*{Competing interests}
On behalf of all authors, the corresponding author states that there is no conflict of interest.

\subsection*{Data availability}
Data sharing is not applicable to this article as no datasets were generated or analysed during the current study.

\subsection*{Use of generative AI and AI-assisted technologies}
During the preparation of this work, the authors used ChatGPT 6.1 in order to improve language and readability, and to assist with the formatting of mathematical expressions. After using this tool, the authors reviewed and edited the content as needed and take full responsibility for the content of the publication. All mathematical proofs and arguments were independently verified by the authors.

\subsection*{Acknowledgments} 
Yinghua Li's work is supported by the National Natural Science Foundation of China (Grant No. 12371205) and the Natural Science Foundation of Guangdong Province (Grant Nos. 2025A1515012026 and 2025A1515040001).
Jingxue Yin's work is supported by the National Natural Science Foundation of China (Grant No. 12171166), Jiangxi Provincial Natural Science Foundation (Grant No. 20243BCE51015), and the Natural Science Foundation of Guangdong Province (Grant No. 2025A1515012026).
\vskip2mm
\appendix

\section{Auxiliary linear lemmas}
\label{app:linear-lemmas}

\begin{lemma}[Transport equation]\label{app:transport-lemma}
For $3<q<\infty$, let $v\in L^1(0,T;W^{2,q}\cap H_0^1)$ and consider
\[
 \rho_t+\operatorname{div}(\rho v)=0,\qquad \rho(0)=\rho_0,
\]
where $0\le\rho_0\in W^{1,q}$.
Then there is a unique solution $\rho\in C([0,T];W^{1,q})$, which satisfies
\begin{align}
 \|\rho(t)\|_{W^{1,q}}
 &\le \|\rho_0\|_{W^{1,q}}
 \exp\left(C\int_0^t\|v(s)\|_{W^{2,q}}\,\mathrm ds\right),
 \label{app:transport-W1q}\\
 \inf_x\rho(x,t)
 &\ge\inf_x\rho_0(x)
 \exp\left(-\int_0^t\|\operatorname{div}v(s)\|_{L^\infty}\,\mathrm ds\right).
 \label{app:transport-lower}
\end{align}
Let $\rho_i=\rho[v_i]$, $i=1,2$, be the solutions of the
continuity equation corresponding to velocities $v_i$ and
the same initial density $\rho_0$.  If
$\sum_{i=1}^2\|v_i\|_{L^1(0,T; W^{2,q})}\le L$, then, whenever
$v_1-v_2\in L^2(0,T;H^1)$,
\begin{equation}\label{app:transport-difference}
 \|\rho_1-\rho_2\|_{L^\infty(0,T;L^2)}
 \le C T^{1/2}\|v_1-v_2\|_{L^2(0,T;H^1)}.
\end{equation}
Here the constant may depend on $L$, $\|\rho_0\|_{W^{1,q}}$, $q$,
and the domain, but not on the two velocities separately.
\end{lemma}

\begin{proof}
Since $W^{2,q}\hookrightarrow W^{1,\infty}$, the velocity generates
a unique flow $X$ mapping $\overline\Omega$ onto itself. The method of characteristics then gives
\[
 \rho(t,X(t,x))=\rho_0(x)
 \exp\left(-\int_0^t\operatorname{div}v(s,X(s,x))\,\mathrm ds\right),
\]
which proves \eqref{app:transport-lower}. Differentiating the continuity equation in space and applying
H\"older's inequality and Sobolev embedding, we obtain
$\frac{\mathrm d}{\mathrm dt}\|\rho\|_{W^{1,q}}
\le C\|v\|_{W^{2,q}}\|\rho\|_{W^{1,q}}$.
It can be verified from Gronwall's inequality that \eqref{app:transport-W1q} holds, and a standard approximation argument gives
$\rho\in C([0,T];W^{1,q})$.

To estimate the difference, set $r=\rho_1-\rho_2$ and $V=v_1-v_2$. Then
\[
 r_t+v_1\cdot\nabla r+r\operatorname{div}v_1
 =-V\cdot\nabla\rho_2-\rho_2\operatorname{div}V,\qquad r(0)=0.
\]
Testing it by $r$, using $H^1\hookrightarrow L^{2q/(q-2)}$ and
\eqref{app:transport-W1q}, one computes, with
$y_\varepsilon=(\|r\|_2^2+\varepsilon^2)^{1/2}$,
\[
 y_\varepsilon'\le C\|\nabla v_1\|_\infty y_\varepsilon
                  +C_L\|V\|_{H^1},\qquad y_\varepsilon(0)=\varepsilon.
\]
Applying Gronwall's inequality, letting $\varepsilon\to0$,
and using Cauchy's inequality in time yields
\eqref{app:transport-difference}.
Uniqueness follows from the characteristic formula.
\end{proof}

\begin{lemma}[Linear parabolic Lam\'e system]
	\label{app:lame-parabolic-lemma}
	Let $\Omega\subset\mathbb R^3$ be a bounded domain with
	$C^2$ boundary, and let $3<q\le6$.
	Assume that
	\[
	0<a_*\le a(x,t)\le a^*,\qquad
	a\in L^\infty(0,T;W^{1,q}),\qquad
	a_t\in L^\infty(0,T;L^q).
	\]
	Let $w_0\in H^2\cap H_0^1$, $f\in C([0,T];L^2)\cap L^2(0,T;L^q)$ and
	$f_t\in L^2(0,T;H^{-1}),$
	where $H^{-1}=(H_0^1)'$, and define
	\[
	w_1:=a(0)^{-1}\bigl(\mathcal Lw_0+f(0)\bigr)\in L^2.
	\]
	Then the initial-boundary value problem
	\[
	\begin{cases}
		aw_t-\mathcal Lw=f
		&\text{in }\Omega\times(0,T),\\
		w=0  &\text{on }\partial\Omega\times(0,T),\\
		w(0)=w_0
		&\text{in }\Omega
	\end{cases}
	\]
	admits a unique solution satisfying
	\[
	\begin{aligned}
		w&\in C([0,T];H^2\cap H_0^1)
		\cap L^2(0,T;W^{2,q}),\\
		w_t&\in C([0,T];L^2)
		\cap L^2(0,T;H_0^1),
		\qquad w_t(0)=w_1.
	\end{aligned}
	\]
	Moreover,
	\begin{align}
		&\|w\|_{L^\infty(0,T;H^2)}
		+\|w\|_{L^2(0,T;W^{2,q})}
		+\|w_t\|_{L^\infty(0,T;L^2)}
		+\|w_t\|_{L^2(0,T;H_0^1)}
		\notag\\
		&\qquad\le C\Bigl(
		\|w_0\|_{H^2}
		+\|w_1\|_{L^2}
		+\|f\|_{L^\infty(0,T;L^2)}
		+\|f\|_{L^2(0,T;L^q)}
		+\|f_t\|_{L^2(0,T;H^{-1})}
		\Bigr).
		\label{app:lame-parabolic-estimate}
	\end{align}
	Here $C$ depends only on $\Omega$, $\lambda$, $\eta$, $q$,
	$a_*^{-1}$, $a^*$,
	$\|a\|_{L^\infty(0,T;W^{1,q})}$,
	$\|a_t\|_{L^\infty(0,T;L^q)}$,
	and an upper bound for $T$.
\end{lemma}
\begin{proof}
	Define the bilinear form
	$\mathcal A(u,z)
	:=\lambda(\nabla u,\nabla z)
	+(\lambda+\eta)(\operatorname{div}u,\operatorname{div}z).$
	By coercivity and Poincar\'e's inequality,
	$\mathcal A(z,z)\ge c\|z\|_{H_0^1}^2$.
	Let $V_N$ be the span of the first $N$ Dirichlet eigenfunctions
	of $-\mathcal L$, and let $w_{0N}$ be the Ritz projection of
	$w_0$ onto $V_N$. We consider the Galerkin problem
	\[
	(aw_{N,t},z)+\mathcal A(w_N,z)=(f,z),
	\qquad z\in V_N,
	\qquad w_N(0)=w_{0N}.
	\]
	Since $a\ge a_*>0$, the mass matrix is positive definite.
	The Ritz projection property and the definition of $w_1$ imply
	\[
	(a(0)w_{N,t}(0),z)=(a(0)w_1,z),
	\qquad z\in V_N.
	\]
	Taking $z=w_{N,t}(0)$ and applying Cauchy--Schwarz gives
	\[
	E_N(0)\le a^*\|w_1\|_{L^2}^2,
	\qquad
	E_N(t):=\int a(t)|w_{N,t}(t)|^2\,\mathrm dx.
	\]
	
	Differentiating the Galerkin equations in time and testing
	by $w_{N,t}$, we obtain
	\[
	\frac12E_N'
	+\mathcal A(w_{N,t},w_{N,t})
	=\langle f_t,w_{N,t}\rangle
	-\frac12\int a_t|w_{N,t}|^2\,\mathrm dx.
	\]
	For $\beta=2q/(2q-3)$, Sobolev interpolation and Young's
	inequality yield
	\[
	\left|\int a_t|z|^2\,\mathrm dx\right|
	\le \varepsilon\|z\|_{H_0^1}^2
	+C_\varepsilon\|a_t\|_{L^q}^{\beta}\|z\|_{L^2}^2.
	\]
	Estimating the duality term, absorbing the terms involving
	$\|w_{N,t}\|_{H_0^1}^2$, and applying Gronwall's inequality,
	we obtain
	\[
	\|w_{N,t}\|_{L^\infty(0,T;L^2)}
	+\|w_{N,t}\|_{L^2(0,T;H_0^1)}
	\le C\left(
	\|w_1\|_{L^2}
	+\|f_t\|_{L^2(0,T;H^{-1})}
	\right).
	\]
	Together with the convergence of $w_{0N}$ in $H_0^1$, this
	bounds $w_N$ in $H^1(0,T;H_0^1)$ uniformly in $N$.
	Weak compactness then yields a solution $w$ with
	$w(0)=w_0$.
	
	Applying Dirichlet elliptic estimates to
	$\mathcal Lw=aw_t-f$ and using
	$H_0^1(\Omega)\hookrightarrow L^q(\Omega)$ for $q\le6$, we obtain
	\[
	\begin{aligned}
		&\|w\|_{L^\infty(0,T;H^2)}
		+\|w\|_{L^2(0,T;W^{2,q})}\\
		&\qquad\le C\Bigl(
		\|w_t\|_{L^\infty(0,T;L^2)}
		+\|w_t\|_{L^2(0,T;H_0^1)}
		+\|f\|_{L^\infty(0,T;L^2)}
		+\|f\|_{L^2(0,T;L^q)}
		\Bigr).
	\end{aligned}
	\]
	This proves \eqref{app:lame-parabolic-estimate} with the
	stated dependence of the constant.
	
	It remains to establish time continuity and identify the
	initial value of $w_t$. Since $q>3$, the embedding $W^{1,q}(\Omega)\hookrightarrow L^\infty(\Omega)$ holds. The assumptions $a\in L^\infty(0,T;W^{1,q}(\Omega))$ and $a_t\in L^\infty(0,T;L^q(\Omega))$, together with interpolation, imply $a\in C([0,T];L^\infty(\Omega))$. The uniform lower bound $a\ge a_*>0$ then yields $a^{-1}\in C([0,T];L^\infty(\Omega))$.  
	Set $m:=aw_t$. Since multiplication by a $W^{1,q}$ function
	is bounded on $H_0^1$, we have
	\[
	m\in L^2(0,T;H_0^1),
	\qquad
	m_t=\mathcal Lw_t+f_t\in L^2(0,T;H^{-1}).
	\]
	The Lions--Magenes lemma \cite{LionsMagenes1972} therefore
	implies $m\in C([0,T];L^2)$. Hence
	$w_t=a^{-1}m\in C([0,T];L^2)$.
	Since $\mathcal Lw=m-f\in C([0,T];L^2)$,
	elliptic regularity further gives
	$w\in C([0,T];H^2\cap H_0^1)$.
	Evaluating the equation at $t=0$, we conclude that
	$w_t(0)=w_1$.
	
	Finally, testing the homogeneous difference equation by
	the time derivative of the difference and integrating
	in time proves uniqueness.
\end{proof}

\begin{lemma}[Weighted phase projection]
	\label{app:weighted-projection-lemma}
	Let $\Omega\subset\mathbb R^3$ be a bounded domain with
	$C^2$ boundary, and let $3<q<\infty$.
	Let $\{e_j\}_{j\ge1}$ be an $L^2$-orthonormal basis of
	Neumann Laplace eigenfunctions, with $e_1$ constant, and set
	\[
	X_m:=\operatorname{span}\{e_1,\ldots,e_m\}.
	\]
	Suppose that $a\in W^{1,q}(\Omega)$ and $a\ge a_*>0$.
	For $f\in H^1(\Omega)$, define $\mathcal P_m^a f\in X_m$ by
	\[
	(a\mathcal P_m^a f,\chi)=(af,\chi),
	\qquad \chi\in X_m.
	\]
	Then $\mathcal P_m^a$ is well defined and satisfies
	\begin{equation}
		\|\mathcal P_m^a f\|_{H^1}
		\le C\|f\|_{H^1},
		\label{app:weighted-projection-estimate}
	\end{equation}
	where $C$ depends only on $\Omega$, $q$, $a_*^{-1}$,
	and $\|a\|_{W^{1,q}}$, and is independent of $m$.
\end{lemma}

\begin{proof}
Since $a\ge a_*>0$, the weighted mass matrix is positive
definite, so the defining equation admits a unique solution
$z:=\mathcal P_m^a f$. Taking $\chi=z$ gives
\begin{equation}\label{app:weighted-projection-L2}
	\|z\|_{L^2}
	\le a_*^{-1}\|a\|_{L^\infty}\|f\|_{L^2}.
\end{equation}
Since $-\Delta z\in X_m$ and $\partial_n z=0$, we may take
$\chi=-\Delta z$. Integrating by parts, we obtain
\[
 \int a|\nabla z|^2=\int a\nabla f\cdot\nabla z
                       +\int(f-z)\nabla a\cdot\nabla z.
\]
For $p=2q/(q-2)$ and $\theta=3/q<1$, Sobolev interpolation gives
$\|z\|_p\le C\|z\|_2^{1-\theta}\|\nabla z\|_2^\theta+C\|z\|_2$.
H\"older's and Young's inequalities therefore imply
\[
 a_*\|\nabla z\|_2^2\le\tfrac12a_*\|\nabla z\|_2^2
                       +C(\|f\|_{H^1}^2+\|z\|_2^2).
\]
Absorbing the gradient term and using
\eqref{app:weighted-projection-L2} yields
\eqref{app:weighted-projection-estimate}.
All constants depend only on the stated quantities and
are independent of $m$.
\end{proof}

\section{Proofs for the positive-density problem}
\label{app:construction}
This appendix proves Theorem~\ref{positive-density-thm} and
Proposition~\ref{positive-density-continuation}. Only the phase
variables are Galerkin truncated; the density and velocity retain
their full spatial equations on the original bounded domain $\Omega$.
We use the auxiliary results of Appendix~\ref{app:linear-lemmas}.

\subsection{The semidiscrete system}\label{app:semidiscrete-system}

Let $\{e_j\}_{j\ge1}$ be an $L^2$-orthonormal basis of Neumann Laplace
eigenfunctions, including the constant eigenfunction. Set
$X_m=\operatorname{span}\{e_1,\ldots,e_m\}$ and let $Q_m$ denote the
$L^2$-orthogonal projection onto $X_m$. We use the weighted projection
$\mathcal P_m^a$ from Lemma~\ref{app:weighted-projection-lemma}.

For the initial data, set $\phi_{0m}=Q_m\phi_0$ and
$\widehat\mu_{0m}=Q_m\mu_0$.  Take
\begin{equation}\label{app:initial-phase-projection}
\begin{aligned}
 b_m&=
 \frac{\displaystyle\int R\,\mathrm dx-
 \displaystyle\int
 \bigl(\rho_0\widehat\mu_{0m}+\Delta\phi_{0m}
 -\rho_0F(\phi_{0m})\bigr)\,\mathrm dx}
 {\displaystyle\int \rho_0\,\mathrm dx},\qquad
 \mu_{0m}=\widehat\mu_{0m}+b_m,\\
 R_m&=\rho_0\mu_{0m}+\Delta\phi_{0m}-\rho_0F(\phi_{0m}).
\end{aligned}
\end{equation}
Then $\int  R_m\,\mathrm dx=\int  R\,\mathrm dx = 0$ for every
$m$.
The correction $b_m$ changes only the constant eigenfunction, so
$\mu_{0m}\in X_m$ and its Neumann condition is preserved.
It can be verified from the construction above that
$(\phi_{0m},\mu_{0m},R_m)\to(\phi_0,\mu_0,R)$
in $H^3\times H^3\times H^1$.
Indeed, the Neumann spectral characterization gives
$\phi_{0m}\to\phi_0$ and $\widehat\mu_{0m}\to\mu_0$ in $H^3$, which further implies $b_m\to0$.  It follows that $\mu_{0m}\to\mu_0$ in $H^3$, and
\begin{align*}
 R_m-R=\rho_0(\mu_{0m}-\mu_0)+\Delta(\phi_{0m}-\phi_0)
 -\rho_0\bigl(F(\phi_{0m})-F(\phi_0)\bigr)
 \rightarrow0\quad\hbox{in }H^1.
\end{align*}
For the compatibility conditions, set
\begin{equation}\label{app:h0m-g0m}
	h_{0m} =\rho_0^{-1}\Delta\mu_{0m},\qquad
	g_{0m} =\rho_0^{-1/2}
	\left[
	\mathcal Lu_0-\nabla\rho_0^\gamma
	-(\Delta\phi_{0m})\nabla\phi_{0m}
	\right].
\end{equation}
The multiplier property of $\rho_0^{-1}\in W^{1,q}$ and
$H^3\hookrightarrow W^{1,\infty}$ give
\[
 \|h_{0m}-h_0\|_{H^1}
 \le C\|\mu_{0m}-\mu_0\|_{H^3},\qquad
 \|g_{0m}-g_0\|_2
 \le C(\|\phi_{0m}\|_{H^3}+\|\phi_0\|_{H^3})
       \|\phi_{0m}-\phi_0\|_{H^3}.
\]
Thus $h_{0m}\to h_0$ in $H^1$ and $g_{0m}\to g_0$ in $L^2$;
both sequences are uniformly bounded. Here and below these
constants may depend on $\underline\rho^{-1}$.
Moreover, since $\mu_{0m}\in X_m$ satisfies
$\partial_n\mu_{0m}=0$, one has
\begin{equation}\label{app:h0m-zero-mean}
	\int\rho_0h_{0m}\,\mathrm dx
	=
	\int\Delta\mu_{0m}\,\mathrm dx
	=
	\int_{\partial\Omega}\partial_n\mu_{0m}\,\mathrm dS
	=0.
\end{equation}

We now consider the following semi-discrete system,
in which $\rho_m$ and $u_m$ remain full spatial functions, whereas
$\phi_m,\mu_m\in X_m$.
\begin{align}
	&\rho_{m,t}+\operatorname{div}(\rho_m u_m)=0,
	\label{app:m-continuity}\\
	&\rho_m(u_{m,t}+u_m\cdot\nabla u_m)-\mathcal Lu_m
	+\nabla\rho_m^\gamma
	=\bigl[\rho_m\mu_m-\rho_mF(\phi_m)-R_m\bigr]\nabla\phi_m,
	\label{app:m-momentum}\\
	&(\rho_m(\phi_{m,t}+u_m\cdot\nabla\phi_m),\chi)
	+(\nabla\mu_m,\nabla\chi)=0,
	\qquad \chi\in X_m,
	\label{app:m-phase}\\
	&(\rho_m\mu_m,\zeta)
	=(\nabla\phi_m,\nabla\zeta)
	+(\rho_mF(\phi_m)+R_m,\zeta),
	\qquad \zeta\in X_m.
	\label{app:m-chemical}
\end{align}
The system is supplemented with the initial and boundary conditions
\begin{align}
	&(\rho_m,u_m,\phi_m,\mu_m)|_{t=0}
	=(\rho_0,u_0,\phi_{0m},\mu_{0m}),
	\label{app:m-initial}\\
	&u_m=0,\qquad
	\partial_n\phi_m=\partial_n\mu_m=0
	\quad\text{on }\partial\Omega\times(0,T_m).
	\label{app:m-boundary}
\end{align}
Furthermore, the initial data satisfy the discrete chemical and phase
compatibility conditions:
\begin{align}
	\rho_0\mu_{0m}
	&=-\Delta\phi_{0m}+\rho_0F(\phi_{0m})+R_m,
	\label{app:m-initial-chemical}\\
	\Delta\mu_{0m}
	&=\rho_0h_{0m},
	\qquad
	\int\rho_0h_{0m}\,\mathrm dx=0.
	\label{app:m-initial-phase}
\end{align}
\begin{remark}
	At the semi-discrete level, the chemical-potential relation holds
	only when tested against functions in $X_m$. Consequently, the
	capillary force
	$[\rho_m\mu_m-\rho_mF(\phi_m)-R_m]\nabla\phi_m$
	need not coincide with $-(\Delta\phi_m)\nabla\phi_m$ for fixed $m$.
	This choice preserves the cancellation of the coupling terms
	in the basic energy estimate, with the contribution of $R_m$
	retained explicitly. The original capillary force is recovered
	in the limit $m\to\infty$.
\end{remark}

We first record several identities used in the subsequent
estimates. Since $\Delta\phi_m,\Delta\mu_m\in X_m$, the phase
equations give
\begin{equation}\label{app:phase-spectral-identities}
	\begin{aligned}
		\Delta\mu_m
		&=Q_m\bigl[\rho_m(\phi_{m,t}+u_m\cdot\nabla\phi_m)\bigr],\\
		-\Delta\phi_m
		&=Q_m\bigl[\rho_m\mu_m-\rho_mF(\phi_m)-R_m\bigr].
	\end{aligned}
\end{equation}
Set
$\mathcal A_m^{\rm ch}:=\rho_m\mu_m-\rho_mF(\phi_m)-R_m$
and
$\mathcal G_m:=\mathcal A_m^{\rm ch}+\Delta\phi_m$.
Then
\[
\mathcal G_m=(I-Q_m)\mathcal A_m^{\rm ch},
\qquad Q_m\mathcal G_m=0,
\]
and the force can be written as
\[
\mathcal A_m^{\rm ch}\nabla\phi_m
=-(\Delta\phi_m)\nabla\phi_m+\mathcal G_m\nabla\phi_m.
\]
Since $Q_m$ and $R_m$ are independent of time,
\begin{equation}\label{app:chemical-residual-time}
	\partial_t\mathcal G_m
	=(I-Q_m)\bigl[
	\rho_m\mu_{m,t}+\rho_{m,t}\mu_m
	-\rho_{m,t}F(\phi_m)
	-\rho_mF'(\phi_m)\phi_{m,t}
	\bigr]
\end{equation}
in the sense of distributions. The Neumann spectral projection
bounds imply
\begin{align*}
	\|\mathcal G_m\|_{L^2}
	&\le \|\mathcal A_m^{\rm ch}\|_{L^2},
	\qquad
	\|\mathcal G_m\|_{H^1}
	\le C\|\mathcal A_m^{\rm ch}\|_{H^1},\\
	\|\partial_t\mathcal G_m\|_{L^2}
	&\le \|\rho_m\mu_{m,t}\|_{L^2}
	+\|\rho_{m,t}\mu_m\|_{L^2}
	+\|\rho_{m,t}F(\phi_m)\|_{L^2}
	+\|\rho_mF'(\phi_m)\phi_{m,t}\|_{L^2},
\end{align*}
where $C$ is independent of $m$. Moreover,
\eqref{app:m-initial-chemical} gives $\mathcal G_m(0)=0$.
Although $Q_m\mathcal G_m=0$, the pairing
$(\mathcal G_m,u_{m,t}\cdot\nabla\phi_m)$ need not vanish,
since its second factor generally lies outside $X_m$.

Finally, testing the phase equations by $1\in X_m$ and
using the continuity equation and the boundary conditions,
we obtain
\begin{equation}\label{app:discrete-constant-modes}
	\begin{aligned}
		\int\rho_m\mu_m\,\mathrm dx
		&=\int\rho_mF(\phi_m)\,\mathrm dx+\int R_m\,\mathrm dx,\\
		\frac{\mathrm d}{\mathrm dt}\int\rho_m\phi_m\,\mathrm dx
		&=0,
		\qquad
		\frac{\mathrm d}{\mathrm dt}\int\rho_m\,\mathrm dx=0.
	\end{aligned}
\end{equation}
These relations provide the weighted mean bounds required
for the Poincar\'e estimates.

\subsection{Local solvability at fixed $m$}

Fix $m\ge1$. In this step, we will apply the Banach fixed-point theorem to construct a unique local solution
$(\rho_m,u_m,\phi_m,\mu_m)$ of
\eqref{app:m-continuity}-\eqref{app:m-boundary}
on some interval $[0,T_m]$, with
\begin{align}
	&\rho_m\in C([0,T_m];W^{1,q}),
	\qquad
	\rho_{m,t}\in L^\infty(0,T_m;L^q),
	\label{app:fixed-m-rho-regularity}\\
	&u_m\in L^\infty(0,T_m;H^2\cap H_0^1)
	\cap L^2(0,T_m;W^{2,q}),
	\label{app:fixed-m-u-regularity}\\
	&u_{m,t}\in L^\infty(0,T_m;L^2)
	\cap L^2(0,T_m;H_0^1),
	\label{app:fixed-m-ut-regularity}\\
	&\phi_m\in W^{2,\infty}(0,T_m;X_m),
	\qquad
	\mu_m\in W^{1,\infty}(0,T_m;X_m).
	\label{app:fixed-m-phase-regularity}
\end{align}
The construction also ensures that
\begin{equation}\label{app:fixed-m-positive}
	\rho_m(x,t)\ge\frac12\underline\rho
	\qquad
	\text{on }\overline\Omega\times[0,T_m].
\end{equation}
Here $T_m$ and the corresponding bounds may depend on $m$
and $\underline\rho^{-1}$.

\textbf{Step 1: Fixed-point space and trial functions.}
To begin with, for $T>0$ to be chosen below, define
	\begin{align*}
		\mathcal V_T:=\Bigl\{v:\;
		v\in L^\infty(0,T;H^2\cap H_0^1)
		\cap L^2(0,T;W^{2,q}),v_t\in L^\infty(0,T;L^2)
		\cap L^2(0,T;H_0^1)
		\Bigr\}.
	\end{align*}
	Set $\mathcal P_{m,T}:=W^{1,\infty}(0,T;X_m)$ and
		$\mathcal X_{m,T}:=\mathcal V_T\times\mathcal P_{m,T}$.
	For $(v,\psi) \in \mathcal X_{m,T}$, we use the norm
	\begin{align*}
		\|(v,\psi)\|_{\mathcal X_{m,T}}
		:={}&
		\|v\|_{L^\infty(0,T; H^2)}
		+\|v\|_{L^2(0,T; W^{2,q})}
		+\|v_t\|_{L^\infty(0,T; L^2)}
		+\|v_t\|_{L^2(0,T; H_0^1)}\\
		&+\|\psi\|_{L^\infty(0,T; X_m)}
		+\|\psi_t\|_{L^\infty(0,T; X_m)}.
	\end{align*}
	Choose a sufficiently large $M$ and define
	\begin{equation}\label{app:fixed-ball}
		\mathfrak B_{M,T}
		:=
		\left\{
		(v,\psi)\in\mathcal X_{m,T}:
		\|(v,\psi)\|_{\mathcal X_{m,T}}\le M,\quad
		v(0)=u_0,\quad
		\psi(0)=\phi_{0m}
		\right\}.
	\end{equation}
	It is obvious that the set $\mathfrak B_{M,T}$ is nonempty.
	For $(v,\psi)\in\mathfrak B_{M,T}$, we have
	$
	v\in W^{1,\infty}(0,T;L^2)
	\cap W^{1,2}(0,T;H_0^1)$.
	In particular, $v\in C([0,T];H_0^1)$.
	Since $v(0)=u_0$, the fundamental theorem of calculus and
	Cauchy's inequality yield
	\begin{align}
		\sup_{0\le t\le T}\|v(t)-u_0\|_{L^2}
		&\le T\|v_t\|_{L^\infty(0,T;L^2)}
		\le MT,
		\notag\\
		\sup_{0\le t\le T}\|v(t)-u_0\|_{H^1}
		&\le T^{1/2}\|v_t\|_{L^2(0,T;H^1)}
		\le MT^{1/2}.
		\label{app:fixed-v-small-time}
	\end{align}
	Similarly, since $\psi(0)=\phi_{0m}$,
	\begin{equation}\label{app:fixed-psi-small-time}
		\sup_{0\le t\le T}
		\|\psi(t)-\phi_{0m}\|_{X_m}
		\le T\|\psi_t\|_{L^\infty(0,T;X_m)}
		\le MT.
	\end{equation}

\textbf{Step 2: The transport equation.}
Let $\rho=\rho[v]$ be the solution of
\[
\rho_t+\operatorname{div}(\rho v)=0,
\qquad
\rho(0)=\rho_0.
\]
For $(v,\psi)\in\mathfrak B_{M,T}$, H\"older's inequality in time and
$W^{2,q}\hookrightarrow W^{1,\infty}$ give
\begin{equation}\label{app:fixed-velocity-small}
	\begin{aligned}
		&	\int_0^T\|v(s)\|_{W^{2,q}}\,\mathrm ds
	\le
		T^{1/2}\|v\|_{L^2(0,T;W^{2,q})}
		\le MT^{1/2},\\
		&\int_0^T
		\bigl(
		\|v(s)\|_{L^\infty}+\|\nabla v(s)\|_{L^\infty}
		+\|\nabla^2v(s)\|_{L^q}
		\bigr)\,\mathrm ds
		\le CMT^{1/2}.
	\end{aligned}
\end{equation}
Lemma~\ref{app:transport-lemma} therefore gives
\begin{align}
	\sup_{0\le s\le T}\|\rho(s)\|_{W^{1,q}}
	&\le
	\|\rho_0\|_{W^{1,q}}
	\exp(CMT^{1/2}),
	\label{app:fixed-rho-bound}\\
	\rho(x,t)
	&\ge
	\underline\rho\exp(-CMT^{1/2}).
	\label{app:fixed-rho-lower}
\end{align}
Hence, after taking $CMT^{1/2}\le\log2$, one derives
$
\rho(x,t)\ge\frac12\underline\rho.
$

We next prove that $\rho(t)\to\rho_0$ in $W^{1,q}$ uniformly over
$\mathfrak B_{M,T}$.  Let $X(s;t,x)$ be the flow of $v$ and set the backward flow map $Y(t,x):=X(0;t,x)$. The characteristic formula is therefore
\begin{equation}\label{app:rho-flow-representation}
	\rho(t,x)=\rho_0(Y(t,x))A(t,x),
\end{equation}
where
\[
A(t,x):=
\exp\left(
-\int_0^t
\operatorname{div}v(s,X(s;t,x))\,\mathrm ds
\right).
\]
By \eqref{app:fixed-velocity-small} and the flow equation,
\begin{align}
	\|Y(t)-\operatorname{Id}\|_{L^\infty}
	&\le
	\int_0^t\|v(s)\|_{L^\infty}\,\mathrm ds
	\le CMT^{1/2},
	\label{app:flow-close-identity}\\
	\|D Y(t)-I\|_{L^\infty}
	&\le
	\exp\left(
	\int_0^t\|\nabla v(s)\|_{L^\infty}\,\mathrm ds
	\right)-1
	\le C_M T^{1/2}.
	\label{app:flow-gradient-close}
\end{align}
Here $DY$ is the spatial Jacobian matrix of $Y$.
After reducing $T$, the maps $Y(t,\cdot)$ and $Y(t,\cdot)^{-1}$ are
uniformly Lipschitz, and their Jacobians are bounded above and below by
positive constants independent of $(v,\psi)$.

Besides, the definition of $A$ and the preceding flow estimates imply
\begin{align}
	\|A(t)-1\|_{L^\infty}
	&\le C_M T^{1/2},
	\label{app:A-close-one}\\
	\|\nabla A(t)\|_{L^q}
	&\le
	C\int_0^t\|\nabla\operatorname{div}v(s)\|_{L^q}\,\mathrm ds
	\le C_M T^{1/2}.
	\label{app:A-gradient-small}
\end{align}
Indeed, the chain rule gives
\[
\nabla A(t,x)
=-A(t,x)\int_0^t
D_xX(s;t,x)^{\mathsf T}
\nabla\operatorname{div}v(s,X(s;t,x))\,\mathrm ds.
\]
Using the uniform bounds for $A$, $D_xX(s;t,\cdot)$,
and the Jacobian determinants of the inverse flow maps,
the change of variables $y=X(s;t,x)$ yields
\eqref{app:A-gradient-small}.

We now estimate the difference in \eqref{app:rho-flow-representation}.
By the characteristic formula,
\begin{align}
	\|\rho(t)-\rho_0\|_{L^q}
	\le
	\|A(t)\|_{L^\infty}
	\|\rho_0\circ Y(t)-\rho_0\|_{L^q}
	+\|\rho_0\|_{L^q}\|A(t)-1\|_{L^\infty}\le C_M T^{1/2}.
	\label{app:rho-Lq-time-modulus}
\end{align}
Here we used $\rho_0\in W^{1,q}$ and the standard composition
estimate along the flow, together with
\eqref{app:A-close-one}. Then,
after differentiating \eqref{app:rho-flow-representation}, one has
\begin{align}
	\nabla\rho(t)-\nabla\rho_0
	=
	A D Y^{\mathsf T}
	\bigl[
	(\nabla\rho_0)\circ Y-\nabla\rho_0
	\bigr] +
	\bigl(A D Y^{\mathsf T}-I\bigr)\nabla\rho_0
	+(\rho_0\circ Y)\nabla A.
	\label{app:rho-gradient-difference}
\end{align}
Since $\nabla\rho_0\in L^q$, approximation by continuous
vector fields and the uniform Jacobian bounds yield a modulus
$\omega_{\rho_0,q}$, independent of the particular flow, such that
\begin{equation}\label{app:composition-modulus}
	\|(\nabla\rho_0)\circ Y-\nabla\rho_0\|_{L^q}
	\le
	\omega_{\rho_0,q}
	\bigl(\|Y-\operatorname{Id}\|_{L^\infty}\bigr),
	\qquad
	\omega_{\rho_0,q}(r)\rightarrow0
	\quad\text{as }r\rightarrow0.
\end{equation}
Indeed, choose $g_\varepsilon\in C(\overline\Omega;\mathbb R^3)$
with $\|\nabla\rho_0-g_\varepsilon\|_q<\varepsilon$.
Change of variables and the uniform inverse-Jacobian bound give
\[
 \|(\nabla\rho_0)\circ Y-\nabla\rho_0\|_q
 \le C\varepsilon+|\Omega|^{1/q}
 \sup_{\substack{a,b\in\overline\Omega\\|a-b|\le r}}
 |g_\varepsilon(a)-g_\varepsilon(b)|,
 \qquad r=\|Y-\operatorname{Id}\|_\infty.
\]
First letting $r\to0$ and then $\varepsilon\to0$ proves
\eqref{app:composition-modulus} uniformly over the trial ball.
Combining
\eqref{app:flow-close-identity}-\eqref{app:A-gradient-small},
\eqref{app:rho-gradient-difference}, and
\eqref{app:composition-modulus}, we obtain
\[
\sup_{0\le t\le T}
\|\nabla\rho(t)-\nabla\rho_0\|_{L^q}
\le
C_M T^{1/2}
+
C\omega_{\rho_0,q}(CMT^{1/2}).
\]
Together with \eqref{app:rho-Lq-time-modulus},
there exists a function $\omega_M^\rho:[0,1]\to[0,\infty)$ such that
$
\omega_M^\rho(T)\rightarrow0$ as $T\rightarrow0,$
and
\begin{equation}\label{app:transport-uniform-modulus}
	\sup_{(v,\psi)\in\mathfrak B_{M,T}}
	\sup_{0\le t\le T}
	\|\rho[v](t)-\rho_0\|_{W^{1,q}}
	\le\omega_M^\rho(T).
\end{equation}

\textbf{Step 3: Finite-dimensional phase system.}
Let $\rho=\rho[v]$ be the density obtained above.  After reducing $T$,
if necessary, we have
$
0<\rho_*:=\frac12\underline\rho
\le \rho(x,t)\le\rho^*:=2\overline\rho$
in $\Omega\times[0,T]$.
We seek $\phi,\mu:[0,T]\to X_m$ satisfying
\begin{align}
	&(\rho\phi_t,\chi)+(\rho v\cdot\nabla\phi,\chi)
	+(\nabla\mu,\nabla\chi)=0,
	&&\chi\in X_m,
	\label{app:frozen-phase}\\
	&(\rho\mu,\zeta)
	=(\nabla\phi,\nabla\zeta)
	+(\rho F(\psi)+R_m,\zeta),
	&&\zeta\in X_m,
	\label{app:frozen-chemical}\\
	&\phi(0)=\phi_{0m}.
	\label{app:frozen-phase-initial}
\end{align}
We write
$
\phi=\sum_{j=1}^m a_j e_j$,
and
	$\mu=\sum_{j=1}^m b_j e_j.
	$
	Then define the matrices $\mathbf M=\mathbf M_m$, $\mathbf S=\mathbf S_m$,
	$\mathbf C=\mathbf C_m$, and the vector $\mathbf N=\mathbf N_m$ by
	\begin{align*}
		(\mathbf M)_{ij}&=(\rho e_j,e_i),
		&
		(\mathbf S)_{ij}&=(\nabla e_j,\nabla e_i),\\
		(\mathbf C)_{ij}&=(\rho v\cdot\nabla e_j,e_i),
		&
		(\mathbf N)_i&=(\rho F(\psi)+R_m,e_i).
	\end{align*}
	Therefore \eqref{app:frozen-phase} and \eqref{app:frozen-chemical} become
	\begin{equation}\label{app:phase-matrix}
		\mathbf M\dot a+\mathbf C a+\mathbf S b=0,
		\qquad
		\mathbf M b=\mathbf S a+\mathbf N.
	\end{equation}
Since $\{e_j\}_{j=1}^m$ is orthonormal in $L^2$, for every
$z\in\mathbb R^m$,
\begin{equation}\label{app:phase-mass-coercivity}
	\rho_*|z|^2
		\le z^{\mathsf T}\mathbf M(t)z
	\le \rho^*|z|^2,
	\qquad
		\|\mathbf M(t)^{-1}\| \le\rho_*^{-1}.
\end{equation}
Noticing that $R_m$ is independent of time, the coefficients satisfy
\begin{align*}
	\dot{\mathbf M}_{ij}
	=(\rho_t e_j,e_i), \quad
	\dot{\mathbf C}_{ij}
	=(\rho_t v\cdot\nabla e_j,e_i)
	+(\rho v_t\cdot\nabla e_j,e_i),\quad
	\dot{\mathbf N}_i
	=(\rho_tF(\psi)+\rho F'(\psi)\psi_t,e_i).
\end{align*}
The regularity of $\rho$, $v$, and $\psi$, together with
the smoothness of the fixed basis functions, therefore gives
\begin{equation}\label{app:phase-coefficients}
	\mathbf M,\mathbf C
	\in W^{1,\infty}(0,T;\mathbb R^{m\times m}),
	\qquad
	\mathbf N\in W^{1,\infty}(0,T;\mathbb R^m).
\end{equation}
Eliminating $b$ from \eqref{app:phase-matrix}, we obtain
\begin{equation}\label{app:phase-explicit-ODE-revised}
	\dot a=-A_m a-G_m,
\end{equation}
where
$
A_m=\mathbf M^{-1}\mathbf C
+\mathbf M^{-1}\mathbf S\mathbf M^{-1}\mathbf S$,
and
$G_m=\mathbf M^{-1}\mathbf S\mathbf M^{-1}\mathbf N.
$
The uniform positive definiteness of $\mathbf M$ and
\eqref{app:phase-coefficients} imply
\[
\mathbf M^{-1},A_m,G_m\in W^{1,\infty}(0,T).
\]
Standard linear ODE theory therefore yields a unique solution
$a\in W^{2,\infty}(0,T;\mathbb R^m)$ with the prescribed
initial value. Moreover,
\begin{equation}\label{app:b-regularity}
b=\mathbf M^{-1}(\mathbf S a+\mathbf N)
\in W^{1,\infty}(0,T;\mathbb R^m).
\end{equation}
The initial chemical compatibility ensures that
$b(0)$ agrees with the coefficient vector of $\mu_{0m}$.
Thus \eqref{app:frozen-phase}-\eqref{app:frozen-phase-initial}
has a unique solution satisfying
\begin{equation}\label{app:fixed-phase-regularity}
	\phi\in W^{2,\infty}(0,T;X_m),
	\qquad
	\mu\in W^{1,\infty}(0,T;X_m).
\end{equation}

At $t=0$, \eqref{app:m-initial-chemical} and the invertibility
of $\mathbf M(0)$ give $\mu(0)=\mu_{0m}$. The phase equation
then yields $(\rho_0\phi_t(0),\chi)
=(\rho_0(h_{0m}-u_0\cdot\nabla\phi_{0m}),\chi)$ for $\chi\in X_m$.
Hence

\begin{equation}\label{app:discrete-phit-initial}
	\phi_t(0)
	=
	\mathcal P_m^{\rho_0}
	\bigl(
	h_{0m}-u_0\cdot\nabla\phi_{0m}
	\bigr).
\end{equation}
Consequently, Lemma \ref{app:weighted-projection-lemma} gives
\begin{align*}
	\|\phi_t(0)\|_{H^1}
	\le
	C\left(
	\|h_{0m}\|_{H^1}
	+\|u_0\|_{H^2}\|\phi_{0m}\|_{H^3}
	\right),
\end{align*}
and hence
\begin{equation}\label{app:discrete-phit-uniform}
	\sup_m\|\phi_t(0)\|_{H^1}<\infty.
\end{equation}

The coefficient bounds and the ODE give
$\|a\|_{W^{2,\infty}}+\|b\|_{W^{1,\infty}}\le C_{m,M}$,
uniformly over the trial ball for small $T$.
Integration in time and finite-dimensional norm equivalence yield

\begin{equation}\label{app:fixed-output-phase-modulus}
	\sup_{0\le t\le T}\left(
	\|\phi(t)-\phi_{0m}\|_{X_m}
	+\|\mu(t)-\mu_{0m}\|_{X_m}
	+\|\phi_t(t)-\phi_t(0)\|_{X_m}
	\right)
	\le C_{m,M}T.
\end{equation}
Since $\phi(0)$ and $\phi_t(0)$ are determined solely by the
initial data, it follows that
\begin{equation}\label{app:discrete-phase-self-map}
	\|\phi\|_{W^{1,\infty}(0,T;X_m)}
	\le C_{0,m}+C_{m,M}T,
	\qquad
	\|\mu\|_{W^{1,\infty}(0,T;X_m)}
	\le C_{m,M},
\end{equation}
where $C_{0,m}$ is independent of $M$ and $T$.

\textbf{Step 4: Velocity.}
We now construct the velocity output. For a given
$(v,\psi)\in\mathfrak B_{M,T}$, we have already determined the density $\rho=\rho[v]$
and the phase variables $(\phi,\mu)$.
Set
\[
\mathcal A^{\rm ch}[\rho,\phi,\mu]
:=\rho\mu-\rho F(\phi)-R_m
\]
and
\[
f[v,\psi]
:=-\rho(v\cdot\nabla)v-\nabla\rho^\gamma
+\mathcal A^{\rm ch}[\rho,\phi,\mu]\nabla\phi.
\]
We seek $u$ as the solution of the linear problem
\begin{equation}\label{app:fixed-full-velocity}
	\begin{cases}
		\rho u_t-\mathcal Lu=f[v,\psi]
		&\text{in }\Omega\times(0,T),\\
		u=0&\text{on }\partial\Omega\times(0,T),\\
		u(0)=u_0&\text{in }\Omega.
	\end{cases}
\end{equation}
To this end, we verify the hypotheses of
Lemma~\ref{app:lame-parabolic-lemma}.

First, the transport estimates ensure that $\rho$ has the
required regularity and is bounded away from zero.
The bounds for $v$, $\rho$, and the phase outputs, together
with finite-dimensional norm equivalence, give
\[
\|f\|_{L^\infty(0,T;L^q)}\le C_{m,M}.
\]
In particular, since $q>3$ and $\Omega$ is bounded,
$f\in L^\infty(0,T;L^2)$ and
\begin{equation}\label{app:fixed-f-Lq-small}
	\|f\|_{L^2(0,T;L^q)}
	\le C_{m,M}T^{1/2}.
\end{equation}
To bound $f_t$ in $H^{-1}$, the only convection term requiring
integration by parts is $\rho(v\cdot\nabla)v_t$:
\[
 \langle\rho(v\cdot\nabla)v_t,z\rangle
 =-\int\rho v_t\cdot(v\cdot\nabla z)
  -\int\operatorname{div}(\rho v)v_t\cdot z,
 \qquad z\in H_0^1.
\]
The other convection terms and the pressure derivative are
bounded by $C_{m,M}(\|\rho_t\|_q+\|v_t\|_2)$.
Since $R_{m,t}=0$,
$\mathcal A_t^{\rm ch}=\rho_t(\mu-F(\phi))
+\rho(\mu_t-F'(\phi)\phi_t)$.
The fixed-dimensional phase bounds therefore give
\[
 \|f_t\|_{H^{-1}}
 \le C_{m,M}(\|\rho_t\|_q+\|v_t\|_2
                  +\|\phi_t\|_{X_m}+\|\mu_t\|_{X_m}).
\]
The right-hand side is uniformly bounded on the trial ball, so

\begin{equation}\label{app:fixed-ft-small}
	\|f_t\|_{L^2(0,T;H^{-1})}
	\le C_{m,M}T^{1/2}.
\end{equation}
The continuity of $\rho$ in $W^{1,q}$, of $v$ in $H^1$
(with bounded $H^2$ norm), and of the phase outputs gives
$f\in C([0,T];L^2)$. Expanding differences of the factors,
using the positive density bound for the pressure term, yields
\[
 \|f(t)-f(0)\|_2\le C_{m,M}\bigl(
 \|\rho(t)-\rho_0\|_{W^{1,q}}+\|v(t)-u_0\|_{H^1}
 +\|\phi(t)-\phi_{0m}\|_{X_m}+\|\mu(t)-\mu_{0m}\|_{X_m}\bigr).
\]
Thus, with $\omega_{m,M}^0(T)
=C_{m,M}(\omega_M^\rho(T)+T^{1/2})$, we have

\begin{equation}\label{app:fixed-f-continuity}
	\sup_{0\le t\le T}\|f(t)-f(0)\|_{L^2}
	\le\omega_{m,M}^0(T),
	\qquad
	\omega_{m,M}^0(T)\rightarrow0
	\quad\text{as }T\rightarrow0.
\end{equation}
Combining this with \eqref{app:fixed-ft-small} and
\eqref{app:fixed-f-Lq-small}, there exists a modulus $\omega_{m,M}(T)$ such that
\begin{equation}\label{app:fixed-forcing-small-time}
	\|f_t\|_{L^2(0,T;H^{-1})}
	+\|f-f(0)\|_{L^\infty(0,T;L^2)}
	+\|f\|_{L^2(0,T;L^q)}
	\le\omega_{m,M}(T),
	\quad
	\omega_{m,M}(T)\rightarrow0
	\quad\text{as }T\rightarrow0.
\end{equation}

 Finally, the initial chemical compatibility gives
 \[
 \mathcal A^{\rm ch}[\rho_0,\phi_{0m},\mu_{0m}]
 =-\Delta\phi_{0m}.
 \]
 Thus the initial acceleration prescribed by
 \eqref{app:fixed-full-velocity} is
 \[
 u_t(0) =\rho_0^{-1}\bigl(\mathcal Lu_0+f(0)\bigr)
 =\rho_0^{-1/2}g_{0m}-u_0\cdot\nabla u_0
 \in L^2.
 \]
 Consequently, Lemma~\ref{app:lame-parabolic-lemma} now yields a unique
 solution $u$ satisfying
 \[
 u\in L^\infty(0,T;H^2\cap H_0^1)
 \cap L^2(0,T;W^{2,q}),
 \qquad
 u_t\in L^\infty(0,T;L^2)
 \cap L^2(0,T;H_0^1).
 \]
 This defines the velocity output associated with $(v,\psi)$.

 \textbf{Step 5: Self-mapping and contraction.}
 We complete the construction by applying the Banach fixed-point theorem.
 Throughout this step, $m$ is fixed and $0<T\le1$.

 \medskip
 \noindent\textit{Definition of the map and the metric.}
 For $(v,\psi)\in\mathfrak B_{M,T}$, the preceding steps determine
 successively the density $\rho=\rho[v]$, the phase variables
 $(\phi,\mu)$, and the velocity $u$. Define
 \[
 \mathcal T_m(v,\psi):=(u,\phi).
 \]
 At this stage, $\mathcal T_m$ takes values in $\mathcal X_{m,T}$;
 the self-mapping property will be established below. Equip
 $\mathfrak B_{M,T}$ with the metric
 \[
 \begin{aligned}
 	d_T\bigl((v_1,\psi_1),(v_2,\psi_2)\bigr)
 	:=\|v_1-v_2\|_{L^\infty(0,T;L^2)}
 	+\|v_1-v_2\|_{L^2(0,T;H^1)}
 	+\|\psi_1-\psi_2\|_{C([0,T];X_m)}.
 \end{aligned}
 \]
 It should be mentioned that the strong norm defining the ball supplies the regularity needed for the construction, whereas this weaker metric is used for contraction.

 \medskip
 \noindent\textit{Completeness.}
 The metric space $(\mathfrak B_{M,T},d_T)$ is complete.
 Indeed, a $d_T$-Cauchy sequence $(v_n,\psi_n)$ converges to
 $(v,\psi)$ in
 \[
 \bigl(C([0,T];L^2)\cap L^2(0,T;H_0^1)\bigr)
 \times C([0,T];X_m),
 \]
 where we use the continuous representatives of $v_n$.
 The uniform $\mathcal X_{m,T}$ bounds, weak and weak-* compactness,
 and identification of time derivatives in distributions imply
 that $(v,\psi)\in\mathcal X_{m,T}$. By weak lower semicontinuity,
 \[
 \|(v,\psi)\|_{\mathcal X_{m,T}}\le M.
 \]
 The initial conditions pass to the limit by uniform convergence
 in time. Hence $(v,\psi)\in\mathfrak B_{M,T}$, proving completeness.

 \medskip
 \noindent\textit{Self-mapping.}
 We first separate the contribution of the initial data from
 the terms that become small as $T\rightarrow0$.
 If $CMT^{1/2}\le\log2$, the transport bounds give
 $
 \tfrac12\underline\rho\le\rho(x,t)\le2\overline\rho$.
 In particular, the density bounds depend only on the initial
 bounds. To estimate $u_t$, set $w=u_t$ and
 \[
 E_u(t):=\int\rho|w|^2\,\mathrm dx.
 \]
 Differentiating the velocity equation in time gives
 \[
 \rho w_t-\mathcal Lw=f_t-\rho_t w.
 \]
 Testing by $w$ yields
 \[
 \frac12E_u'(t)+(-\mathcal Lw,w)
 =\langle f_t,w\rangle
 -\frac12\int\rho_t|w|^2\,\mathrm dx.
 \]
Sobolev interpolation, Young's inequality, and the positivity
 of $\rho$ imply
 \[
 E_u'(t)+c\|\nabla w\|_2^2
 \le C\|\rho_t\|_q^{2q/(2q-3)}E_u(t)
      +C\|f_t\|_{H^{-1}}^2.
 \]
 Since $\|\rho_t\|_{L^\infty L^q}\le C_M$,
 \eqref{app:fixed-ft-small} and Gronwall's inequality give
 \[
 \sup_{s\le T}E_u(s)+c\int_0^T\|\nabla u_t\|_2^2\,\mathrm dt
 \le C\bigl(E_u(0)+C_{m,M}T\bigr)e^{C_MT}.
 \]
 Here $C,c$ are independent of $M$, and
 \[
 E_u(0)
 =\int \rho_0
 \left|\rho_0^{-1/2}g_{0m}-u_0\cdot\nabla u_0\right|^2
 \,\mathrm dx
 \]
 depends only on the fixed initial data and $m$.
  Thus the norms of
 $u_t$ are bounded by an initial-data constant plus a quantity
 tending to zero for fixed $m,M$. Applying the elliptic
 estimates to $\mathcal Lu=\rho u_t-f$ yields
 \[
 \begin{aligned}
 	\|u\|_{L^\infty H^2}
 	&\le C\bigl(\|u_t\|_{L^\infty L^2}
 	+\|f\|_{L^\infty L^2}\bigr),\\
 	\|u\|_{L^2W^{2,q}}
 	&\le C\bigl(\|u_t\|_{L^2H^1}
 	+\|f\|_{L^2L^q}\bigr).
 \end{aligned}
 \]
 Then using  \eqref{app:discrete-phase-self-map} and
 \eqref{app:fixed-forcing-small-time}, one has
 \begin{equation}\label{app:fixed-self-map-estimate}
 	\|\mathcal T_m(v,\psi)\|_{\mathcal X_{m,T}}
 	\le C_{0,m}+\widetilde\omega_{m,M}(T),
 	\qquad
 	\widetilde\omega_{m,M}(T)\rightarrow0
 	\quad\text{as }T\rightarrow0.
 \end{equation}
 The constant $C_{0,m}$ depends only on $m$, the fixed initial
 data, and the fixed equation parameters; in particular, it
 does not depend on $M$. Choose $M>2C_{0,m}$ large enough that
 the ball is nonempty, and then choose $T$ sufficiently small
 that $CMT^{1/2}\le\log2$ and
 $\widetilde\omega_{m,M}(T)\le M/2$.
 The initial values of the outputs were verified in the preceding
 steps. Consequently,
 \[
 \mathcal T_m(\mathfrak B_{M,T})\subset\mathfrak B_{M,T}.
 \]
 The radius $M$ is fixed from now on; $T$ may be reduced further.

 \medskip
 \noindent\textit{Difference estimates and contraction.}
 Take $(v_k,\psi_k)\in\mathfrak B_{M,T}$, and denote the corresponding
 outputs by $(\rho_k,\phi_k,\mu_k,u_k)$, $k=1,2$. Set
 \[
 V:=v_1-v_2,\quad r:=\rho_1-\rho_2,\quad
 \Phi:=\phi_1-\phi_2,\quad H:=\mu_1-\mu_2,\quad U:=u_1-u_2.
 \]

 The transport difference estimate \eqref{app:transport-difference} gives
 \begin{equation}\label{app:fixed-density-difference}
 	\|r\|_{L^\infty(0,T;L^2)}
 	\le C_M T^{1/2}\|V\|_{L^2(0,T;H^1)}.
 \end{equation}
 For the phase variables, the inverse-matrix identity and
 the local Lipschitz continuity of $F$ imply
 \begin{align}
 	\|\mathbf M_1-\mathbf M_2\|
 	+\|\mathbf M_1^{-1}-\mathbf M_2^{-1}\|
 	&\le C_{m,M}\|r\|_{L^2},\notag\\
 	\|\mathbf C_1-\mathbf C_2\|
 	&\le C_{m,M}(\|r\|_{L^2}+\|V\|_{L^2}),\notag\\
 	|\mathbf N_1-\mathbf N_2|
 	&\le C_{m,M}(\|r\|_{L^2}+\|\psi_1-\psi_2\|_{X_m}).
 	\label{app:phase-matrix-differences}
 \end{align}
 Subtracting the coefficient ODEs and using $a_1(0)=a_2(0)$,
 Gronwall's inequality yields
 \[
 \sup_{0\le t\le T}|a_1(t)-a_2(t)|
 \le C_{m,M}\left(
 T\|r\|_{L^\infty(0,T;L^2)}
 +T^{1/2}\|V\|_{L^2(0,T;L^2)}
 +T\|\psi_1-\psi_2\|_{C([0,T];X_m)}\right).
 \]
 Thus, by finite-dimensional norm equivalence and $T\le1$,
 \begin{equation}\label{app:fixed-phase-difference}
 	\|\Phi\|_{C([0,T];X_m)}
 	\le C_{m,M}T^{1/2}
 	\left(\|V\|_{L^2(0,T;H^1)}
 	+\|\psi_1-\psi_2\|_{C([0,T];X_m)}\right).
 \end{equation}
 The algebraic formula \eqref{app:b-regularity} also gives
 \begin{equation}\label{app:fixed-mu-difference}
 	\|H\|_{C([0,T];X_m)}
 	\le C_{m,M}\left(
 	\|r\|_{L^\infty(0,T;L^2)}
 	+\|\Phi\|_{C([0,T];X_m)}
 	+\|\psi_1-\psi_2\|_{C([0,T];X_m)}\right).
 \end{equation}

We finally estimate the velocity difference.
Subtracting the two velocity equations and expanding the
convection terms, we obtain
\begin{align*}
	\rho_1U_t-\mathcal LU
	={}&-r u_{2,t}
	-r(v_1\cdot\nabla)v_1
	-\rho_2(V\cdot\nabla)v_1
	-\rho_2(v_2\cdot\nabla)V\\
	&-\nabla(\rho_1^\gamma-\rho_2^\gamma)
	+\mathcal A^{\rm ch}[\rho_1,\phi_1,\mu_1]\nabla\phi_1
	-\mathcal A^{\rm ch}[\rho_2,\phi_2,\mu_2]\nabla\phi_2,
\end{align*}
with $U(0)=0$ and $U|_{\partial\Omega}=0$.

Write $\mathcal A_k^{\rm ch}=\mathcal A^{\rm ch}[\rho_k,\phi_k,\mu_k]$.
The same $R_m$ appears in both chemical coefficients, and
$\mathcal A_1^{\rm ch}-\mathcal A_2^{\rm ch}
=r(\mu_1-F(\phi_1))+\rho_2H-\rho_2(F(\phi_1)-F(\phi_2))$.
The ball bounds and finite-dimensional norm equivalence therefore give

\begin{equation}\label{app:fixed-capillary-difference}
	\begin{aligned}
		\|\mathcal A^{\rm ch}[\rho_1,\phi_1,\mu_1]\nabla\phi_1
		-\mathcal A^{\rm ch}[\rho_2,\phi_2,\mu_2]\nabla\phi_2\|_{L^2}
		\le C_{m,M}
		\bigl(\|r\|_{L^2}+\|H\|_{X_m}+\|\Phi\|_{X_m}\bigr).
	\end{aligned}
\end{equation}
By \eqref{app:fixed-density-difference},
\eqref{app:fixed-phase-difference}, and
\eqref{app:fixed-mu-difference}, it follows that
\begin{equation}\label{app:fixed-capillary-difference-time}
	\begin{aligned}
		&\int_0^T
		\|\mathcal A^{\rm ch}[\rho_1,\phi_1,\mu_1]\nabla\phi_1
		-\mathcal A^{\rm ch}[\rho_2,\phi_2,\mu_2]\nabla\phi_2\|_{L^2}^2
		\,\mathrm dt\\
		&\qquad\le C_{m,M}T\left(
		\|V\|_{L^2(0,T;H^1)}^2
		+\|\psi_1-\psi_2\|_{C([0,T];X_m)}^2
		\right).
	\end{aligned}
\end{equation}
Test the velocity difference equation by $U$. We first notice that
\[
 \int\rho_2(v_2\cdot\nabla)V\cdot U
 =-\int\rho_2V\cdot(v_2\cdot\nabla U)
  -\int\operatorname{div}(\rho_2v_2)V\cdot U.
\]
Since $q>3$, $\operatorname{div}(\rho_2v_2)$ is bounded in $L^3$.
Besides, Sobolev and Young inequalities bound the convection and pressure
terms by $\varepsilon\|\nabla U\|_2^2
+C_{m,M,\varepsilon}(\|V\|_2^2+\|r\|_2^2)$.
The remaining density terms satisfy
\[
 \left|\int r u_{2,t}\cdot U\right|
 \le\varepsilon\|\nabla U\|_2^2
    +C_\varepsilon\|u_{2,t}\|_{H^1}^2\|r\|_2^2,
 \qquad
 \left|\int\rho_{1,t}|U|^2\right|
 \le\varepsilon\|\nabla U\|_2^2+C_{m,M,\varepsilon}\|U\|_2^2.
\]
Finally, the capillary pairing is bounded by
$\varepsilon\|\nabla U\|_2^2
+C_\varepsilon\|\mathcal A_1^{\rm ch}\nabla\phi_1
-\mathcal A_2^{\rm ch}\nabla\phi_2\|_2^2$.

Therefore, after choosing $\varepsilon$ sufficiently small and using the
coercivity of $-\mathcal L$, we obtain
\begin{equation}\label{app:fixed-velocity-differential}
	\begin{aligned}
		\frac{\mathrm d}{\mathrm dt}
		\int \rho_1|U|^2\,\mathrm dx
		+c\|\nabla U\|_{L^2}^2
		\le&
		C_{m,M}\|U\|_{L^2}^2
		+C_{m,M}(1+\|u_{2,t}\|_{H^1}^2)\|r\|_{L^2}^2
		+C_{m,M}\|V\|_{L^2}^2\\
		&+C\|\mathcal A^{\rm ch}[\rho_1,\phi_1,\mu_1]\nabla\phi_1
		-\mathcal A^{\rm ch}[\rho_2,\phi_2,\mu_2]\nabla\phi_2\|_{L^2}^2.
	\end{aligned}
\end{equation}
The output bounds and \eqref{app:fixed-density-difference} imply
\[
 \int_0^T(1+\|u_{2,t}\|_{H^1}^2)\|r\|_2^2\,\mathrm dt
 \le C_{m,M}T\|V\|_{L^2H^1}^2,
 \qquad
 \int_0^T\|V\|_2^2\,\mathrm dt
 \le T\|V\|_{L^\infty L^2}^2.
\]
Combining these with \eqref{app:fixed-capillary-difference-time},
$U(0)=0$, and the lower density bound, Gronwall's inequality yields

\begin{equation}\label{app:fixed-velocity-difference}
	\begin{aligned}
		&\|U\|_{L^\infty(0,T;L^2)}^2+\|U\|_{L^2(0,T;H^1)}^2\\
        &\quad\le C_{m,M}T\left(
        \|V\|_{L^\infty(0,T;L^2)}^2+\|V\|_{L^2(0,T;H^1)}^2
        +\|\psi_1-\psi_2\|_{C([0,T];X_m)}^2\right).
	\end{aligned}
\end{equation}
 Together with \eqref{app:fixed-phase-difference}, this proves
 \[
 d_T\bigl(\mathcal T_m(v_1,\psi_1),\mathcal T_m(v_2,\psi_2)\bigr)
 \le C_{m,M}T^{1/2}
 d_T\bigl((v_1,\psi_1),(v_2,\psi_2)\bigr).
 \]
 Reducing $T$ further so that $C_{m,M}T^{1/2}<1$, we conclude
 that $\mathcal T_m$ is a contraction on the complete metric
 space $(\mathfrak B_{M,T},d_T)$.

 \medskip
 \noindent\textit{Existence and uniqueness.}
 The Banach fixed-point theorem yields a unique fixed point
 $(u,\phi)\in\mathfrak B_{M,T}$. At this fixed point,
 $v=u$ and $\psi=\phi$, so the constructed quadruple
 $(\rho,u,\phi,\mu)$ solves the semi-discrete system with the
 prescribed initial and boundary conditions. The density remains
 positive by \eqref{app:fixed-rho-lower}, and the phase construction
 gives
 \[
 \phi\in W^{2,\infty}(0,T;X_m),
 \qquad \mu\in W^{1,\infty}(0,T;X_m).
 \]
 To prove uniqueness in the stated regularity class, let
 $(\rho_i,u_i,\phi_i,\mu_i)$, $i=1,2$, be two solutions
 with the same initial data. The preceding difference estimates
 give, on a sufficiently short interval $[0,\tau]$,
 \[
 d_\tau\bigl((u_1,\phi_1),(u_2,\phi_2)\bigr)
 \le C\tau^{1/2}
 d_\tau\bigl((u_1,\phi_1),(u_2,\phi_2)\bigr),
 \]
 where $C$ depends on common bounds for the two solutions.
 Choosing $\tau$ so that $C\tau^{1/2}<1$ yields
 $u_1=u_2$ and $\phi_1=\phi_2$.
 The transport equation and the discrete chemical relation
 then give $\rho_1=\rho_2$ and $\mu_1=\mu_2$.
 Repeating the argument from any time at which the solutions
 agree proves uniqueness on their common interval of existence.

The local construction and uniqueness therefore determine
a unique maximal solution for each fixed $m$.
We denote its maximal existence time by $T_m\in(0,\infty]$
and its interval of existence by $[0,T_m)$.

\subsection{Uniform estimates and a common lifespan in $m$}
\label{subsec:positive-step-uniform}

For the semi-discrete solution, let $\Phi_m$ and $\mathcal D_m$
be the quantities defined in \eqref{def_Phi} and \eqref{def_D},
evaluated at $(\rho_m,u_m,\phi_m,\mu_m,R_m)$.
Let $\mathcal M_{0m}$ be the initial size defined by
\eqref{def_initial_M}, with $R$, $\phi_0$, $\mu_0$, $g_0$,
and $h_0$ replaced by $R_m$, $\phi_{0m}$, $\mu_{0m}$,
$g_{0m}$, and $h_{0m}$, respectively, where $g_{0m}$ and
$h_{0m}$ are given by \eqref{app:h0m-g0m}. The convergence of the
initial approximations and \eqref{app:discrete-phit-uniform} give
\begin{equation}\label{app:uniform-initial-size-m}
 B_*:=\sup_m\mathcal M_{0m}<\infty,\qquad
 \sup_m\|\phi_{m,t}(0)\|_{H^1}<\infty,\qquad \int R_m=0.
\end{equation}
To record the additional positive-density regularity, define
\begin{equation}\label{app:Y-definition}
\begin{aligned}
 \mathfrak Y_m(t):={}&\Phi_m(t)
 +\sup_{s\le t}\bigl(\|u_m(s)\|_{H^2}^2+\|u_{m,t}(s)\|_2^2
                         +\|\phi_{m,t}(s)\|_{H^1}^2\bigr)\\
 &+\int_0^t\bigl(\|u_{m,t}\|_{H^1}^2+\|u_m\|_{W^{2,q}}^2
              +\|\mu_{m,t}\|_{H^1}^2+\|\mu_m\|_{H^3}^2\bigr)\,\mathrm ds.
\end{aligned}
\end{equation}
In the following lemma we suppress $m$ and write $T^*$ for the
current semi-discrete lifespan. We also use the viscous bilinear form
$\mathcal A(v,w)=\lambda(\nabla v,\nabla w)
        +(\lambda+\eta)(\operatorname{div}v,\operatorname{div}w)$.

\begin{lemma}[Semidiscrete estimate at positive density]
	\label{app:positive-density-semidiscrete-estimate}
	Let $c>0$ and $0<S<T^*$.  Suppose
	\begin{equation}\label{app:temporary-positive-density}
		\rho(x,t)\ge c,
		\qquad (x,t)\in\Omega\times[0,S].
	\end{equation}
	Then there is a nondecreasing locally bounded function $M_c$, independent
	of $m$, such that, for $0\le t\le S$,
	\begin{equation}\label{app:discrete-apriori}
		\Phi(t)+\int_0^t\mathcal D(s)\,\mathrm ds
		\le
		M_c\left(
		\mathcal M_0+\int_0^tM_c(\Phi(s))\,\mathrm ds
		\right).
	\end{equation}
	Here $M_c:[0,\infty)\to[1,\infty)$ is a continuous
	nondecreasing function independent of the phase dimension.
\end{lemma}

\begin{proof}
The proof follows Parts I--VI of Section~\ref{sec:apriori}
in the same order. We cite the unchanged
calculations and detail only the modifications caused by the phase
projection. All phase tests belong to $X_m$, and the momentum force
remains $\mathcal A_m^{\rm ch}\nabla\phi_m$. Constants may depend on
$c^{-1}$ but are independent of $m$.

For simplicity, we suppress the index $m$ on the solution
variables, $R_m$, $\mathcal A_m^{\rm ch}$, $\Phi_m$,
$\mathcal D_m$, and $\mathcal M_{0m}$, and write
$P:=\mathcal P_m^\rho$ and $Q:=Q_m$.
Thus $\mathcal A^{\rm ch}=\rho\mu-\rho F(\phi)-R$.
The nondecreasing function $M_c$ may vary from line to line
and may be enlarged.

\medskip
\noindent\textit{Part I: Basic energy and the $H^1$ bound for $\phi$.}
Applying the arguments of Lemmas~\ref{Lemma_energy} and
\ref{Lemma_phi_H1} to the semi-discrete system, with $R_m$
in place of $R$, we obtain
\begin{equation}\label{app:detailed-basic-bound}
	\sup_{0\le s\le t}\|\phi(s)\|_{H^1}^2
	+\int_0^t\bigl(
	\|\nabla u(s)\|_{L^2}^2+\|\nabla\mu(s)\|_{L^2}^2
	\bigr)\,\mathrm ds
	\le M_c\left(
	\mathcal M_0+\int_0^t M_c(\Phi(s))\,\mathrm ds
	\right).
\end{equation}
\medskip
\noindent\textit{Part II: Estimates for the velocity.}
In the elliptic argument of Lemma~\ref{Lemma_u_H2}, the only changed
term is the force. Since $q>3$, the product estimate gives
\[
\begin{aligned}
	\|\mathcal A^{\rm ch}\|_{H^1}
	\le C\|\rho\|_{W^{1,q}}
	\bigl(\|\mu\|_{H^1}+\|F(\phi)\|_{H^1}\bigr)
	+C\|R\|_{H^1}
	\le M_c(\Phi),
\end{aligned}
\]
and hence
$\|\mathcal A^{\rm ch}\nabla\phi\|_{L^2}\le M_c(\Phi)$.
Estimating the convection term as in Lemma~\ref{Lemma_u_H2}
and using the transport estimate \eqref{rho_rt_est}, we obtain
\begin{equation}\label{app:positive-density-auxiliary-pointwise}
	\begin{aligned}
		\|u\|_{H^2}+\|\rho_t\|_{L^q}
		+\|u_t\|_{L^2}+\|\phi_t\|_{H^1}
		+\|\rho_t[\mu-F(\phi)]-\rho F'(\phi)\phi_t\|_{L^2}
		\le M_c(\Phi).
	\end{aligned}
\end{equation}
Here the unweighted time derivatives are controlled by
$\rho\ge c$ and the corresponding terms in $\Phi$.

To estimate the differentiated force in Lemma~\ref{Lemma_ut}, we differentiate the
chemical equation in time:
\[
(\rho\mu_t,\chi)
=(\nabla\phi_t,\nabla\chi)
-(\rho_t[\mu-F(\phi)]-\rho F'(\phi)\phi_t,\chi),
\qquad \chi\in X_m.
\]
For any $h\in H^1$, the definition of $P$ and $\mu_t\in X_m$
give $(\rho\mu_t,h)=(\rho\mu_t,Ph)$. Therefore,
\begin{align}
	(\rho\mu_t,h)
	&=(\nabla\phi_t,\nabla Ph)
	-(\rho_t[\mu-F(\phi)]-\rho F'(\phi)\phi_t,Ph),
	\label{app:weighted-transfer-chemical}\\
	(\partial_t\mathcal A^{\rm ch},h)
	&=(\nabla\phi_t,\nabla Ph)
	+(\rho_t[\mu-F(\phi)]-\rho F'(\phi)\phi_t,h-Ph).
	\label{app:weighted-transfer-force}
\end{align}
The second identity follows from
$\partial_t\mathcal A^{\rm ch}
=\rho\mu_t+\rho_t[\mu-F(\phi)]-\rho F'(\phi)\phi_t$.
Similarly,
\begin{equation}\label{app:weighted-transfer-H}
	\begin{aligned}
		((\rho\mu)_t,h)
		=(\nabla\phi_t,\nabla Ph)
		+(\rho_tF(\phi)+\rho F'(\phi)\phi_t,Ph)
		+(\rho_t\mu,h-Ph).
	\end{aligned}
\end{equation}
These identities are applied at each fixed time and do not
require differentiating $P$.

By Lemma~\ref{app:weighted-projection-lemma} and
\eqref{app:positive-density-auxiliary-pointwise},
\begin{equation}\label{app:weighted-transfer-estimate}
 |(\rho\mu_t,h)|+|(\partial_t\mathcal A^{\rm ch},h)|
 \le M_c(\Phi)\|h\|_{H^1}.
\end{equation}
Taking $h=u_t\cdot\nabla\phi$, we have
$\|h\|_{H^1}\le C\|\phi\|_{H^3}\|\nabla u_t\|_2$ and hence
\begin{equation}\label{app:modified-capillary-time-bound}
\begin{aligned}
 |\langle\partial_t(\mathcal A^{\rm ch}\nabla\phi),u_t\rangle|
 &\le |(\partial_t\mathcal A^{\rm ch},u_t\cdot\nabla\phi)|
       +\|\mathcal A^{\rm ch}\|_3\|\nabla \phi_t\|_2\|u_t\|_6\\
 &\le\varepsilon\|\nabla u_t\|_2^2+M_{c,\varepsilon}(\Phi).
\end{aligned}
\end{equation}
The remaining terms are exactly those estimated in
Lemma~\ref{Lemma_ut}. Choosing $\varepsilon$ sufficiently
small and absorbing the gradient terms, we obtain
\begin{equation}\label{app:modified-ut-differential}
 \frac{\mathrm d}{\mathrm dt}\|\sqrt\rho u_t\|_2^2
       +c_u\|\nabla u_t\|_2^2\le M_c(\Phi),\qquad c_u>0.
\end{equation}
The initial compatibility bounds therefore imply
\begin{equation}\label{app:closed-velocity-time}
 \sup_{0\le s\le t}\|\sqrt{\rho(s)}u_t(s)\|_2^2
       +\int_0^t\|\nabla u_t\|_2^2\,\mathrm ds\le M_c\left(\mathcal M_0+\int_0^t M_c(\Phi(s))\,\mathrm ds\right).
\end{equation}

The proof of Lemma~\ref{Lemma_nabla_u} uses only the time derivative
of $\|\nabla u\|_2^2$ and the preceding $u_t$ estimate. It applies
without change, with \eqref{app:closed-velocity-time} in place of
\eqref{ut-closed}, and gives
\begin{equation}\label{app:closed-velocity-gradient}
 \sup_{0\le s\le t}\|u(s)\|_{H^1}^2\le M_c\left(\mathcal M_0+\int_0^t M_c(\Phi(s))\,\mathrm ds\right).
\end{equation}

\medskip
\noindent\textit{Part III: Estimates for the density.}
Since
\[
\|\mathcal A^{\rm ch}\nabla\phi\|_{L^q}
\le C\|\mathcal A^{\rm ch}\|_{H^1}
\|\nabla\phi\|_{L^\infty}
\le M_c(\Phi),
\]
the elliptic argument of Lemma~\ref{nabla2_u_est} gives
\begin{equation}\label{app:detailed-u-W2q}
	\|u\|_{W^{2,q}}
	\le M_c(\Phi)\bigl(1+\|\nabla u_t\|_{L^2}\bigr).
\end{equation}
Applying the arguments of Lemmas~\ref{nabla2_u_est} and
\ref{Lemma_rho_W1q} and using
\eqref{app:closed-velocity-time}, we obtain
\begin{equation}\label{app:closed-density}
	\begin{aligned}
		\sup_{0\le s\le t}\|\rho(s)\|_{W^{1,q}}
		+\int_0^t\bigl(
		\|\nabla^2u(s)\|_{L^q}
		+\|\nabla u(s)\|_{L^\infty}
		\bigr)\,\mathrm ds
		\le M_c\left(
		\mathcal M_0+\int_0^t M_c(\Phi(s))\,\mathrm ds
		\right).
	\end{aligned}
\end{equation}

\medskip
\noindent\textit{Part IV: Estimates for $\mu$ and $\nabla\phi_t$.}
The calculation in Lemma~\ref{Lemma_rho_mu_phit} uses only the
chemical tests $\mu,\mu_t$ and the phase test $\phi_t$, all in $X_m$.
Its identity \eqref{rho_mu_energy} and the five estimates following
it are therefore unchanged, and give
\begin{equation}\label{app:closed-weighted-mu}
 \sup_{0\le s\le t}\|\sqrt{\rho(s)}\mu(s)\|_2^2
       +\int_0^t\|\sqrt\rho \phi_t\|_2^2\,\mathrm ds
 \le M_c\left(\mathcal M_0+\int_0^t M_c(\Phi(s))\,\mathrm ds\right).
\end{equation}
For Lemma~\ref{Lemma_mu_phi_t}, the only inadmissible substitution
would be the pointwise chemical identity in the pairing with
$u\cdot\nabla\phi$. Instead, test the phase equation with $\mu_t$
and the differentiated chemical equation with $\phi_t$ to obtain
\begin{equation}\label{app:transferred-low-phase-identity}
 \tfrac12\frac{\mathrm d}{\mathrm dt}\|\nabla\mu\|_2^2
       +\|\nabla \phi_t\|_2^2
 =-(\rho\mu_t,u\cdot\nabla\phi)+(\rho_t[\mu-F(\phi)]-\rho F'(\phi)\phi_t,\phi_t).
\end{equation}
By \eqref{app:weighted-transfer-estimate} and
$\|u\cdot\nabla\phi\|_{H^1}\le C\|u\|_{H^2}\|\phi\|_{H^3}$,
the right-hand side is bounded by $M_c(\Phi)$.
After integration, \eqref{app:closed-weighted-mu} and $\rho\ge c$
also control $\|\mu\|_2$, giving the analogue of
\eqref{phi_mu_t_energy}.

In Lemma~\ref{Lemma_phi_H2}, replace the chemical relation by its
projected version $-\Delta\phi=Q\mathcal A^{\rm ch}$. Since $Q$
is an $L^2$ contraction,
\[
 \|\phi\|_{H^2}
 \le C\|\rho\|_\infty
       (\|\mu\|_2+\|\phi\|_{H^1}^3+\|\phi\|_2)
       +C\|R\|_2+C\|\phi\|_2.
\]
The already established bounds thus give
\begin{equation}\label{app:detailed-low-closure}
\begin{aligned}
 \sup_{0\le s\le t}\bigl(
              \|\mu(s)\|_{H^1}+\|\phi(s)\|_{H^2}\bigr)
 +\int_0^t\|\nabla \phi_t\|_2^2\,\mathrm ds
 \le M_c\left(\mathcal M_0+\int_0^t M_c(\Phi(s))\,\mathrm ds\right).
\end{aligned}
\end{equation}

\medskip
\noindent\textit{Part V: High-order joint estimate for the phase variables.}
The arguments of Lemmas~\ref{Lemma_phase_gradient_modified}
and \ref{Lemma_phase_high_joint} apply to the semi-discrete
system, since the tests $\mu_t$, $\phi_t$, $\phi_{tt}$, and $1$
all belong to $X_m$.
The lower-order bounds required to control the correction
term follow from \eqref{app:detailed-low-closure}.
Moreover,
\[
\phi_t(0)
=\mathcal P_m^{\rho_0}
\bigl(h_{0m}-u_0\cdot\nabla\phi_{0m}\bigr),
\]
whose uniform $H^1$ bound is given by
\eqref{app:discrete-phit-uniform}.
Using these bounds and \eqref{app:closed-velocity-time},
we obtain
\begin{equation}\label{app:detailed-high-closure}
	\begin{aligned}
		\sup_{0\le s\le t}\Bigl(
		\|\nabla\phi_t(s)\|_{L^2}^2
		+\|\sqrt{\rho(s)}\phi_t(s)\|_{L^2}^2
		\Bigr)
		+\int_0^t\Bigl(
		\|\nabla\mu_t(s)\|_{L^2}^2
		+\|\sqrt{\rho(s)}\mu_t(s)\|_{L^2}^2
		\Bigr)\,\mathrm ds\\
		\le M_c\left(
		\mathcal M_0+\int_0^t M_c(\Phi(s))\,\mathrm ds
		\right).
	\end{aligned}
\end{equation}

 Then, using \eqref{app:phase-spectral-identities} and the uniform
$L^2$ and $H^1$ bounds for $Q$, we apply the elliptic argument
of Lemma~\ref{Lemma_mu_phi_elliptic_closed}.
Together with \eqref{app:detailed-low-closure} and
\eqref{app:detailed-high-closure}, this yields
\[
\|\mu(t)\|_{H^2}+\|\phi(t)\|_{H^3}
\le M_c\left(
\mathcal M_0+\int_0^t M_c(\Phi(s))\,\mathrm ds
\right).
\]
The momentum elliptic estimate then gives the same bound
for $\|u(t)\|_{H^2}$, using
$\|\mathcal A^{\rm ch}\nabla\phi\|_{L^2}
\le \|\mathcal A^{\rm ch}\|_{L^3}\|\nabla\phi\|_{L^6}$
and estimating the convection term as in Part II.

Collecting the estimates established in Parts I--VI,
we obtain \eqref{app:discrete-apriori}.
The time-differentiated calculations are justified by
Lemma~\ref{app:lame-parabolic-lemma} for the velocity
and by \eqref{app:fixed-m-phase-regularity} for the phase variables.
\end{proof}

\noindent\textit{Common lifespan at fixed positive density.}
Restore the index $m$ and set $c=\underline\rho/2$.
Let $B_*$ also dominate $\sup_m\Phi_m(0)$.
Applying the continuity argument in
Remark~\ref{rem:apriori-closure} with $M_c$ in place of $M$,
we choose
\[
K:=2M_c(B_*+1)
\]
and $T>0$, independent of $m$, such that
$TM_c(K)\le1$. The resulting estimate improves
$\Phi_m\le K$ to $\Phi_m\le K/2$ as long as $\rho_m\ge c$.

Moreover, \eqref{app:detailed-u-W2q} and
\eqref{app:closed-velocity-time} imply
\[
\int_0^t\|\nabla u_m(s)\|_{L^\infty}\,\mathrm ds
\le C M_c(K)\left(
t+t^{1/2}M_c(B_*+1)^{1/2}
\right),
\qquad t\le T.
\]
Reducing $T$ so that the right-hand side is at most
$\frac13\log(4/3)$, the characteristic formula gives
$\rho_m\ge3\underline\rho/4$.
The usual first-exit argument, together with the local
construction at fixed $m$, extends all approximations
to the common interval $[0,T]$.

Finally, \eqref{app:detailed-u-W2q} and the established
$u_{m,t}$ estimate give a uniform bound for
$\|u_m\|_{L^2(0,T;W^{2,q})}$.
By the projected phase equation, the $H^1$-stability of $Q_m$,
and the product estimate,
\[
\|\mu_m\|_{H^3}
\le C\|Q_m[\rho_m(\phi_{m,t}
+u_m\cdot\nabla\phi_m)]\|_{H^1}
+C\|\mu_m\|_{L^2}
\le C_{\underline\rho,K}.
\]
Combining these bounds with the preceding dissipation
estimates, we obtain
\[
\rho_m\ge\tfrac12\underline\rho,
\qquad
\sup_m\mathfrak Y_m(T_{\underline\rho})
\le C_{\underline\rho},
\qquad T_{\underline\rho}:=T>0.
\]
Both $T_{\underline\rho}$ and $C_{\underline\rho}$ are
independent of $m$, but may depend on
$\underline\rho^{-1}$ and $B_*$.

\subsection{Passage to the limit and completion of the proof}
\label{subsec:positive-step-limit}

By the uniform bounds for $\mathfrak Y_m$ and standard
compactness arguments, we extract a subsequence converging
weakly or weakly-star in the corresponding spaces, with
time derivatives identified distributionally, and strongly
as follows:
\begin{equation}\label{app:m-strong-convergence}
	\begin{aligned}
		\rho_m&\to\rho
		&&\text{in }C([0,T_{\underline\rho}];C(\overline\Omega)),
		&\qquad
		u_m&\to u
		&&\text{in }C([0,T_{\underline\rho}];W^{1,q}),
		\\
		\phi_m&\to\phi
		&&\text{in }C([0,T_{\underline\rho}];H^2),
		&
		\mu_m&\to\mu
		&&\text{in }C([0,T_{\underline\rho}];H^1).
	\end{aligned}
\end{equation}
These strong convergences, Sobolev embeddings, and
$R_m\to R$ in $H^1$ imply
$\mathcal A_m^{\rm ch}\to\mathcal A^{\rm ch}$ in
$C([0,T_{\underline\rho}];L^3)$ and yield
\begin{align}
	\rho_mF(\phi_m)&\to\rho F(\phi)
	&&\text{in }C([0,T_{\underline\rho}];L^2),
	&\quad
	\rho_m\mu_m&\to\rho\mu
	&&\text{in }C([0,T_{\underline\rho}];L^2),\notag\\
	\rho_m u_m\cdot\nabla\phi_m&\to\rho u\cdot\nabla\phi
	&&\text{in }C([0,T_{\underline\rho}];L^2),
	&\quad
	\mathcal A_m^{\rm ch}\nabla\phi_m
	&\to\mathcal A^{\rm ch}\nabla\phi
	&&\text{in }C([0,T_{\underline\rho}];L^2),
	\label{app:nonlinear-strong-limit}
\end{align}
where
$\mathcal A^{\rm ch}=\rho\mu-\rho F(\phi)-R$. Hence
\[
\|\mathcal G_m\|_{C_tL^2_x}
\le \|\mathcal A_m^{\rm ch}-\mathcal A^{\rm ch}\|_{C_tL^2_x}
+\|(I-Q_m)\mathcal A^{\rm ch}\|_{C_tL^2_x}
\to0.
\]
Here the strong convergence $Q_m\to I$ on $L^2$ is uniform
on the compact image of $t\mapsto\mathcal A^{\rm ch}(t)$.
The uniform $H^3$ bound for $\phi_m$ also gives
$\mathcal G_m\nabla\phi_m\to0$ in $L^\infty_tL^2_x$.

The strong convergence of $\rho_m$ identifies the weighted
weak limit of $\phi_{m,t}$:
\begin{equation}\label{app:m-weighted-phit-limit}
\rho_m\phi_{m,t}\rightharpoonup^*\rho\phi_t
\quad\text{in }L^\infty(0,T_{\underline\rho};L^2).
\end{equation}
Testing the Galerkin phase equations with each fixed
eigenfunction $e_j$ and a smooth time test function, passing
to the limit using \eqref{app:nonlinear-strong-limit} and
\eqref{app:m-weighted-phit-limit}, and using density in
$H^1(\Omega)$, we recover
\[
\Delta\mu=\rho(\phi_t+u\cdot\nabla\phi),
\qquad
\rho\mu=-\Delta\phi+\rho F(\phi)+R,
\qquad
\partial_n\phi=\partial_n\mu=0.
\]
In particular,
$\mathcal A^{\rm ch}\nabla\phi=-(\Delta\phi)\nabla\phi$.
Passing to the limit in the conservative forms of the
continuity and momentum equations gives
\[
\begin{aligned}
\rho_t+\operatorname{div}(\rho u)&=0,\\
\rho(u_t+u\cdot\nabla u)-\mathcal Lu+\nabla\rho^\gamma
&=-(\Delta\phi)\nabla\phi.
\end{aligned}
\]
The homogeneous Dirichlet condition is preserved by weak
convergence in $L^2(0,T_{\underline\rho};H_0^1)$.

At this stage, the convergence in \eqref{app:m-strong-convergence},
together with the convergence of the approximate initial
data, identifies the initial values in the corresponding
spaces. Weak lower semicontinuity yields the uniform bounds
for the limits.

We now improve the time continuity to the
regularity asserted in Theorem~\ref{positive-density-thm}.
Since $u\in L^1(0,T_{\underline\rho};W^{2,q})$,
Lemma~\ref{app:transport-lemma} gives
$
\rho\in C([0,T_{\underline\rho}];W^{1,q}).
$
Differentiating the phase equation in time in the sense of distributions and using the
established bounds, together with the multiplier property
of $\rho^{-1}\in L^\infty(0,T_{\underline\rho};W^{1,q})$,
gives
\[
\phi_{tt}\in L^2(0,T_{\underline\rho};(H^1)').
\]
Since $\phi_t\in L^2(0,T_{\underline\rho};H^1)$ and
$\phi_{tt}\in L^2(0,T_{\underline\rho};(H^1)')$,
the Lions--Magenes continuity theorem \cite{LionsMagenes1972}
yields
$\phi_t\in C([0,T_{\underline\rho}];L^2)$.
Consequently, using \eqref{app:m-strong-convergence},
\[
\Delta\mu=\rho(\phi_t+u\cdot\nabla\phi)
\in C([0,T_{\underline\rho}];L^2).
\]
Neumann elliptic regularity and the continuity of $\mu$
in $L^2$ therefore give
$\mu\in C([0,T_{\underline\rho}];H^2)$.
The chemical identity then yields
\[
-\Delta\phi=\rho\mu-\rho F(\phi)-R
\in C([0,T_{\underline\rho}];H^1),
\]
and hence
$\phi\in C([0,T_{\underline\rho}];H^3)$.

Similarly, differentiating the momentum equation in time in the sense of distributions
and using the established bounds gives
\[
u_{tt}\in L^2(0,T_{\underline\rho};H^{-1}).
\]
Together with $u_t\in L^2(0,T_{\underline\rho};H_0^1)$,
this implies $u_t\in C([0,T_{\underline\rho}];L^2)$.
Thus
\[
\mathcal Lu
=\rho(u_t+u\cdot\nabla u)
+\nabla\rho^\gamma+(\Delta\phi)\nabla\phi
\in C([0,T_{\underline\rho}];L^2),
\]
and Dirichlet elliptic regularity gives
$u\in C([0,T_{\underline\rho}];H^2)$.

The resulting strong traces coincide with the initial
values identified above. This establishes the existence
and regularity assertions of
Theorem~\ref{positive-density-thm}, with
$T_*=T_{\underline\rho}$.
Uniqueness follows from Subsection~\ref{sec:uniqueness},
whose argument also applies to positive-density solutions
with the same residual $R$. This completes the proof of Theorem~\ref{positive-density-thm}.

\subsection{The continuation criterion}
\label{app:continuation-proof}
We conclude this appendix by proving
Proposition~\ref{positive-density-continuation}.
\begin{proof}
	By the phase equation and Neumann elliptic regularity,
	\[
	\|\mu(t)\|_{H^3}
	\le C\Bigl(
	\|\rho(\phi_t+u\cdot\nabla\phi)(t)\|_{H^1}
	+\|\mu(t)\|_{L^2}
	\Bigr)
	\le C(c_*^{-1},C_*)
	\]
	for almost every $t<T^*$.
	At such times, the equations yield the compatibility
	conditions in Theorem~\ref{positive-density-thm}, with
	uniformly bounded compatibility data.
	The local existence theorem and the uniform lifespan
	established in Subsection~\ref{subsec:positive-step-uniform}
	therefore allow us to restart the solution at these times
	with a common existence time $\theta>0$.
	Choosing a restarting time $t_0$ such that
	$T^*-t_0<\theta$ and using uniqueness on the overlap,
	we extend the solution beyond $T^*$.
\end{proof}

\end{document}